\documentclass[a4paper,reqno]{amsart}
\pdfoutput=1
\usepackage[T1]{fontenc}
\usepackage{lmodern}
\usepackage{dsfont}

\usepackage{amsmath,amsfonts,amssymb,amsthm,stmaryrd,bbm,graphicx,mathtools,enumerate}
\usepackage{mathrsfs}
\usepackage[utf8]{inputenc}  

\usepackage[english]{babel}
\usepackage[tracking,spacing,kerning,babel]{microtype}
\usepackage{wasysym} 
\usepackage{color}
\usepackage{xcolor}
\usepackage{framed} 

\usepackage{wrapfig} 

\usepackage[normalem]{ulem}
\usepackage{soul}

\usepackage[final]{hyperref}   
\hypersetup{
    linktoc=page,
    linkcolor=red,          
    citecolor=blue,        
    filecolor=blue,      
    urlcolor=cyan,
   colorlinks=true           
}

\usepackage{lipsum} 

\usepackage[toc,page]{appendix}

\DeclareRobustCommand{\SkipTocEntry}[5]{}

\usepackage{multicol}

\newtheoremstyle{mine}
{\baselineskip}
{\baselineskip}
{\itshape}
{
}
{\bfseries}
{.}
{.5em}
{#1 #2\ifx#3\relax\else~(#3)\fi}

\theoremstyle{mine}

\newtheorem{theorem}{Theorem}
\numberwithin{theorem}{section}
\newtheorem{corollary}[theorem]{Corollary}
\newtheorem{proposition}[theorem]{Proposition}
\newtheorem{lemma}[theorem]{Lemma}

\newtheorem{definition}[theorem]{Definition}
\newtheorem{conjecture}{Conjecture}
 
\numberwithin{equation}{section}

\theoremstyle{remark}
\newtheorem{remark}{Remark}

\theoremstyle{remark}
\newtheorem{question}{Question}
\newtheorem{example}{Example}

\colorlet{shadecolor}{blue!10}

\def\rm{\reversemarginpar}

\let\qed=\QED

\renewcommand{\epsilon}{\varepsilon}

\newcommand{\R}{\mathbb{R}}

\newcommand{\C}{\mathbb{C}}
\newcommand{\Z}{\mathbb{Z}}
\newcommand{\N}{\mathbb{N}}

\def\T{\mathbb{T}}

\def\calC{\mathcal{C}}

\def\calE{\mathcal{E}}
\def\calF{\mathcal{F}}

\def\calL{\mathcal{L}}
\def\calH{\mathcal{H}}

\def\calN{\mathcal{N}}
\def\calQ{\mathcal{Q}}
\def\calR{\mathcal{R}}
\def\calS{\mathcal{S}}

\newcommand{\hf}{\frac{_1}{^2}}

\newcommand{\caE}{{\mathcal E}}
\newcommand{\caF}{{\mathcal F}}

\newcommand{\caL}{{\mathcal L}}

\newcommand{\caO}{{\mathcal O}}

\newcommand{\caR}{{\mathcal R}}

\def\SLE{\mathrm{SLE}}
\def\CLE{\mathrm{CLE}}

\def\Var#1{\mathrm{Var}\bigl[ #1\bigr]}

\def\Cov#1{\mathrm{Cov}\bigl[ #1\bigr]}

\def\dist{\mathrm{dist}}

\def\P{\mathbb{P}} 
\def\E{\mathbb{E}} 
\def\md{\mid}

\def \eps {\epsilon}

\def\Bb#1#2{{\def\md{\bigm| }#1\bigl[#2\bigr]}}

\def\Pb{\Bb\P}
\def\Eb{\Bb\E}

\def\FK#1#2#3{{\def\md{\bigm| } \P_{#1}^{\,#2}  \bigl[  #3 \bigr]}}
\def\EFK#1#2#3{{\def\md{\bigm| } \E_{#1}^{\,#2}  \bigl[  #3 \bigr]}}

\def \p {{\partial}}

\def\<#1{\langle #1\rangle}

\definecolor{darkgreen}{rgb}{0,0.6,0.05}

\def\Wick#1{\,\colon\!\! \, #1 \, \!\colon}

\def\nn{\nonumber}
\def\bi{\begin{itemize}}  
\def\ei{\end{itemize}}
\def\bnum{\begin{enumerate}} 
\def\enum{\end{enumerate}}
\def\ni{\noindent}

\title[Examples of singular fields in QFT]
{Energy field of critical Ising model and examples of singular fields in QFT}

\author{Christophe Garban}
\author{Antti Kupiainen}

\address
{Universit\'e Claude Bernard Lyon 1, CNRS UMR 5208, Institut Camille Jordan, 69622 Villeurbanne, France, and Institut Universitaire de France (IUF)}
\email{garban@math.univ-lyon1.fr}

\address
{University of Helsinki}
\email{antti.kupiainen@helsinki.fi}

\def\hr{\mathrm{hr}}

\begin{document}

\maketitle

\begin{abstract}
The goal of this paper is to prove singularity of three natural fields in QFT with respect to their natural base measure. The fields we consider are the following ones:
\begin{enumerate}
\item The near-critical limit of the $2d$ Ising model (in the $\beta$-direction) is locally singular w.r.t the critical scaling limit of $2d$ Ising. (N.B. In the $h$-direction it is not locally singular). 
\item The $2d$ Hierarchical Sine-Gordon field is singular w.r.t the $2d$ hierarchical Gaussian Free Field for all $\beta\in[\beta_{L^2}, \beta_{BKT})$. 
\item The Hierarchical $\Phi^4_3$ field is singular w.r.t the $3d$ hierarchical GFF. 
\end{enumerate}
\smallskip
Item (1) gives the first strong indication that the energy field of critical $2d$ Ising model does not exist as a random Schwarz distribution on the plane.
Item (2) has been proved to be singular for the non-hierarchical $2d$ Sine-Gordon sufficiently far from the BKT point in  \cite{gubinelli2024fbsde} while  item (3) is proved to be singular for the non-hierarchical $3d$ $\Phi^4_3$ field in \cite{barashkov2021varphi,oh2021stochastic,hairer2024singularity}.

We believe our way to detect a singular behaviour at all scales is very much {\em down to earth} and may be applicable in all settings where one has a good enough control on the so-called effective potentials.
\end{abstract}

\tableofcontents

\section{Introduction}

\subsection{Main results.}
The goal of this paper is to prove that the following three probability measures on ``QFT fields'' are singular with respect to their natural ``non-interacting'' base measure:
\bnum
\item The scaling limit of the $2d$ near-critical Ising model (in the $\beta$-direction) is locally singular with respect to the $2d$ critical Ising model.
\item The hierarchical Sine-Gordon field is singular with respect to the hierarchical GFF for all inverse temperatures $\beta\in [\beta_{L^2}, \beta_{BKT})$, i.e. all the way to the BKT transition. 
\item The hierarchical $\Phi^4_3$ field is singular with respect to the hierarchical GFF (N.B. Our proof could be extended to all ``hierarchical dimensions'' $8/3\leq d < 4$, but for simplicity we will stick to $d=3$). 
\enum

In each of these cases, singularity is detected for a suitable choice of topology and $\sigma$-field. This is particularly relevant for item (1) where several drastically different topologies exist in order to consider the scaling limit of (near)-critical planar Ising model. For the Ising model, we shall detect singularity for events measurable with respect to the {\em quad crossing topology} introduced in \cite{schramm2011scaling} (see the beginning of Section \ref{s.Ising} for the different natural topologies which exist for the Ising model). 

The main theorems as well as the precise definitions of these 3 models will be given further below in the text:
\bnum
\item {\em Ising model.} Our main result is the singularity statement in Theorem \ref{th.mainI}. See also Proposition \ref{pr.tight} for the exact probability measures considered. 
\item {\em Hierarchical Sine-Gordon.} Our main result is Theorem \ref{th.mainSG} and the proper definition of the Sine-Gordon (SG) field is in Definition \ref{d.SG}. 
Note that the construction of the SG field itself is non-trivial and is of independent interest. We believe this section may serve as a gentle introduction to {\em RG flow} techniques in the {\em ultra-violet (UV)} regime. See also the survey 
\cite{gallavotti2014renormalization}. 

N.B. The case of the hierarchical  Sine-Gordon model in the {\em infra-red (IR)} limit and when $\beta$ is assumed to be large enough is the focus of the work by Dimock \cite{dimock1990kosterlitz}.

\item {\em Hierarchical $\Phi^4_3$.}  Our main result is found  in Theorem  \ref{th.mainPhi}. 
\enum

Since the early stages of this work, several notable singularity results  have been shown, in particular \cite{barashkov2021varphi,oh2021stochastic, hairer2024singularity} on the true (non-hierarchical) $\Phi^4_3$ field in a finite window as well as \cite{gubinelli2024fbsde} for the true (non-hierarchical) Sine-Gordon field on $[0,1]^2$ up to $\beta^2 < 6\pi < \beta^2_{BKT}=8\pi$. Our contributions to the near-critical Ising as well as to the Sine-Gordon field up to the BKT phase transition (though in the hierarchical setting which is simpler) are novel. We also believe that the present way of detecting singularity is very natural and should extend to all settings in QFT where singularity is suspected to hold (at least when a good control on {\em effective potentials} is known, see the discussion in the sketch of proof below).

\subsection{Motivation: non-existence of fields (in the sense of distributions).}\label{ss.motivation}
{\em This section is independent of the rest and may be skipped at first reading.} 

The main motivation which guided this work has to do with the nature of the fields which are at play in {\em conformal field theory} (CFT). 
For the celebrated $2d$ critical Ising model, the relevant CFT  is known to be generated by the three primary fields
\begin{align*}\label{}
1, \sigma, \eps
\end{align*}
which respectively stand for the identity operator, the {\em spin field} and the {\em energy field}. In CFT, these fields come with a rich structure called {\em $k$-point correlation functions} which quantify how these fields are correlated together and interact together. For example 
\begin{align*}\label{}
z_1,\ldots,z_k  \mapsto \<{\sigma(z_1) \ldots \sigma(z_k)}_\Omega \text{  or  } z_1,\ldots,z_k  \mapsto \<{\eps(z_1,\ldots,z_k)}_\Omega
\end{align*}
are functions of $k$ points inside some prescribed domain $\Omega \subset \C$ and describe the internal correlations of respectively the {\em magnetic field} $\sigma$ and the {\em energy field} $\epsilon$. The major breakthrough from \cite{belavin1984infinite} which initiated the construction of conformal field theory was the observation that if the underlying model is conformally invariant and if these $k$-point correlation functions do exist, they must satisfy substantial constraints. An additional structure arises when considering {\em operator product expansions} (OPE) which analyse the way mixed $k$-point correlation functions, such as 
\begin{align*}\label{}
z_1,z_2,z_3  \mapsto \<{\sigma(z_1) \sigma(z_2) \eps(z_3) }
\end{align*}
for example, behave as two insertions, $z_1,z_2$ say, come close to each other (an additional field called the {\em stress tensor $T(z)$} plays a key role there). The power of CFT is that by making few legitimate assumptions on the behaviour of these functions, the constraints on these become so vast that they can then be identified and computed exactly! We refer to the recent book \cite{clementCFT} and references therein for a beautiful account of this theory.

Since the impressive achievement of \cite{belavin1984infinite}, theoretical physicists interested in critical phenomena have  been focusing on the description of these correlation functions in more and more complex settings. (See for example the recent breakthrough on the physics side with the $3d$ conformal bootstrap program \cite{poland2019conformal}). 

On the mathematical side, one may ask the following two natural questions about these {\em fields}. 
\begin{question}\label{q.1}
How do these fields originate from concrete {\em local fields} on a small mesh  lattice ? 
\end{question}
In some cases, the corresponding fields on the lattice are well identified but proving they converge to the CFT correlation functions predicted in \cite{belavin1984infinite} is extremely challenging. For example the breakthrough works 
\cite{chelkak2015conformal,hongler2013energy,chelkak2021correlations, chelkak2023universality}
prove among other things that the following correlation functions of true random variables on a lattice $\Omega_n:=\frac 1 n \Z^2 \cap \Omega$ of mesh $n^{-1}$ do converge to their expected CFT limit:
\begin{align*}\label{}
&n^{- \frac k 8} \<{\sigma(z_1^{(n)}) \ldots z_{k}^{(n)}}_{\beta_c, \Omega_n}  \underset{n\to\infty}\longrightarrow \<{\sigma(z_1) \ldots \sigma(z_k)}_\Omega \\
& n^{-k} \<{(\sigma(z_1^{(n)})\sigma(z_1^{(n)}+\tfrac 1 n)-\tfrac {\sqrt{2}} 2) \ldots (\sigma(z_k^{(n)})\sigma(z_k^{(n)}+\tfrac 1 n)-\tfrac {\sqrt{2}} 2) }  \underset{n\to\infty}\longrightarrow \<{\eps(z_1) \ldots \eps(z_k)}_\Omega\,.
\end{align*}
As such these works identify what is the incarnation of these CFT fields on a lattice. Surprisingly, it is not always straightforward to find the right local discrete field for a continuum CFT field. This is for example the case for the {\em stress tensor $T(z)$}. A complete answer to Question 1, at least for special cases of very integrable CFT such as Ising, is the subject of the exciting current program \cite{hongler2022conformal,clementCFT}. 

Now that some of these fields have  representations on a discrete lattice, let us now turn to the second natural probabilistic question one may ask about these {\em fields}. 
\begin{question}\label{q.2}
In the scaling limit, do these {\em primary  fields} of CFT  correspond to random Schwartz distributions on the plane (or on the domain $\Omega$) ? 
\end{question}
To appreciate this question, let us rephrase it by drawing an analogy with the classical {\em moment problem:} imagine $X_n$ is some concrete random variable with values on $\frac 1 n \Z$ and suppose there exists real numbers $\{m_k\}_{k\geq 1}$ such that 
\begin{align*}\label{}
\Eb{X_n^k}  \underset{n\to\infty}\longrightarrow  m_k\,.
\end{align*}
A CFT theorist would look for some constraints satisfied by the sequence $\{m_k\}$ and under some assumptions on $X_n$ may eventually compute all of these. At the end of the day, the CFT theorist would probably not (need to) wonder whether there exists a limiting random variable $X$ compatible with these moments, i.e. such that $\Eb{X^k}=m_k$ for all $k\geq 1$. This is exactly what the mathematician would try to answer via the classical moment problem and this is what Question 2. is asking in the setting of fields. 

\ni
{\em Question 2 rephrased.}  For a given {\em field} $\phi$ whose $k$-point correlations functions 
\begin{align*}\label{}
z_1,\ldots,z_k \mapsto f_{k}(z_1,\ldots,z_k) = \<{\phi(z_1) \ldots \phi(z_k)}_\Omega
\end{align*}
have been identified via CFT, does there exist a random Schwartz distribution $\Phi$ on $\Omega$ compatible with the family of $k$-point correlation functions $\{ f_k \}_{k\geq 1}$ ? I.e. such that for any smooth test functions $g_1,\ldots, g_k$, 
\begin{align*}\label{}
\Eb{\<{\Phi, g_1} \ldots \<{\Phi, g_k}}  = \iint f_k(z_1,\ldots, z_k) \prod_{i=1}^k g_k(z_i)  dz_i  \, ?
\end{align*}
This is a very natural extension of the {\em moment problem} to the setting of random distributions instead of random variables. 

\smallskip

In certain cases of fields, it is not difficult to infer that only correlations functions make sense and that no compatible random Schwartz distribution may co-exist. For example, 
If $\phi$ is a GFF on $\Omega \subsetneq \C$, then if $\gamma > \gamma_c=2$, the {\em field} $e^{\gamma \phi}$ will exist in the sense of correlation functions but will not be realised as a random distribution (a random measure here). 
See for example the discussion in \cite{kang2011gaussian}.

\smallskip

Going back to the CFT of the critical $2d$ Ising, one may wonder whether the two primary fields $\sigma$ and $\eps$ correspond to random Schwartz distributions in the plane (or in a domain
$\Omega$). The case of the {\em magnetic field} $\sigma$ has been settled in \cite{camia2015planar}. The case of the {\em energy field} $\eps$ on the other hand has been open for a long time.  Indeed the singular  behaviour of its two point function 
\begin{align*}\label{}
\<{\eps(z) \eps(w)} \sim_{z\to w} \frac a {|z-w|^2}
\end{align*}
for some constant $a>0$ makes it difficult to build such an object. Yet, one cannot exclude the possibility of some suitable renormalisation on the diagonal so that a non-trivial energy distribution would still exist.  
Let us also mention another very promising attempt which has been tried for a while in order to define such an energy distribution: the {\em conformal loop ensemble} $\CLE_3$ is known to be the scaling limit of the $+/-$ contours of the Ising model (\cite{benoist2019scaling}). In this sense it captures the scaling limit of the low-temperature representation of the spin-field and should in this respect contain all the relevant information for an energy field of Ising. Yet, no one found a way to produce the desired random Schwartz distribution out of a $\CLE_3$.

The main motivation behind this work is to provide a partially negative answer to this question of existence of an {\em energy distribution field}  for Ising (N.B. interestingly,  in Section \ref{s.potts},  we will give arguments strongly suggesting the existence of such an energy field for 3-Potts and 4-Potts). We shall argue by contradiction as follows. Let $\sigma_n$ be an instance of a $2d$ critical Ising model on $\frac 1 n \Z^2 \cap \Omega$. Note that the discrete energy field $\eps_n=\eps_n(\sigma_n):= (\sigma_n(x)\sigma_n(x+\tfrac 1 n) - \frac{\sqrt{2}} 2)_{x\in \frac 1 n \Z^2 \cap \Omega}$  is measurable w.r.t $\sigma_n$ and the opposite is also true (modulo a global spin flip). Some care is needed when taking a scaling limit especially in the case of the critical Ising model: indeed in the discrete, all fields of interest are measurable w.r.t $\sigma_n$ but in the continuum limit this may no longer be the case! A deep example of this kind is the decoupling of the $\CLE_6$ and the $2d$ white noise. They are both induced in the discrete by critical percolation but they become independent of each of other in the scaling limit due to {\em noise sensitivity}. This goes back to the seminal work \cite{benjamini1999noise} (see also \cite{garban2012noise}).

Having in mind the possible decouplings which arise for critical percolation, let us view the critical Ising percolations $\sigma_n = \{\sigma_n(x)\}_{x\in \frac 1 n \Z^2}$ through the lenses of three different topological spaces:
\begin{align*}\label{}
(\sigma_n, \sigma_n,\eps_n) \in \mathrm{CLE}_{space} \times H^{-s_1} \times H^{-s_2}
\end{align*}
where $\CLE_{space}$ is some suitable space to keep track of $+/-$ loops generated by the spin configuration $\sigma_n$ (the so-called low temperature expansion) and $H^{-s_1},H^{-s_2}$ are Sobolev spaces of sufficiently negative index.  As far as the first lens is concerned, it is known that $\sigma_n \in \CLE_{space}$ converges in law to a nested-$\CLE_3$ collection of conformally invariant loops as proved in \cite{benoist2019scaling}. 
Using a suitable renormalisation of the spin field $\sigma_n$, namely 
\begin{align*}\label{}
S_n(\sigma_n):= n^{-15/8} \sum_{x\in \frac 1 n \Z^2 } \sigma_n(x) \delta_x\,,
\end{align*}
it is also known that $S_n(\sigma_n)$ converges in law to a random distribution $\sigma$ in a sufficiently negative Sobolev space. See \cite{camia2015planar,furlan2017tightness}. 
It is then natural to ask whether there is a suitable renormalization map $E_n: \eps_n \to E_n(\eps_n)\in H^{-s_2}$ so that $E_n(\eps_n)$ converges to a limiting non-trivial field (N.B it can be checked that the first natural guess for $E_n$ leads to $2d$ white noise; this is by no means considered an interesting limit and the goal is somehow to deconvolve this white noise coming from the short distance singularity $1/|z-w|^2$).  
 
Suppose now that the first and third lenses jointly converge to limit $(\Gamma, \eps) \in \CLE_{space} \times H^{-s_2}$ in such a way that $\eps =\eps(\Gamma)$ and the random field $\eps$ has exponential moments, i.e. 
\begin{align}\tag{Hexp}\label{e.Hexp}
\forall \lambda\in \R, \;\;\;  \EFK{}{}{\exp(-\lambda \<{\eps, 1_\Omega}} <\infty
\end{align}
where we assume $\Omega$ to be a bounded domain of the plane.  Then under the assumption~\eqref{e.Hexp}, one would be able to build a new measure on spin configurations $\Gamma\in \CLE_{space}$ as follows:
\begin{align*}\label{}
d \P_\lambda(\Gamma) & \propto \exp\big(-\lambda \<{\eps(\Gamma), 1_\Omega}\big) \P(d\Gamma) \\
& \propto \exp\Big(-\lambda  \int_\Omega [\eps(\Gamma)](x) dx \Big) \P(d\Gamma)
\end{align*}
If the field $\eps=\eps(\Gamma)$ captures in a suitable way the {\em energy field} of Ising, then the measure $\P_\lambda$ would correspond to the {\em near-critical} (also called {\em massive}) limit of the Ising model around $\beta_c$. 
Assumption~\eqref{e.Hexp} would furthermore show that this near-critical limit $\P_\lambda$ is \textbf{absolutely continuous} with respect to the critical Ising scaling limit $\P_0$. Now comes one of the main contributions of this paper: in Section \ref{s.Ising}, we will prove that any subsequential scaling limit of near-critical Ising model is in fact singular with respect to the scaling limit of critical Ising model
\footnote{This is established under a slightly different topology based on $\CLE_{16/3}$ FK-clusters rather than $\CLE_3$ clusters.}!
See Theorem \ref{th.mainI}. 
\smallskip

As such, this singularity result gives to, our knowledge, the best rigorous indication that the energy field of critical Ising CFT \textbf{does not exist as a random Schwartz distribution} (N.B. this is under the assumption that the latter field satisfies the assumption~\eqref{e.Hexp}).

The two other singularity results proved in this paper may be motivated in the same way:
\bi
\item Let $\varphi$ be a Gaussian Free Field in a bounded domain $\Omega \subset \C$ with, say, Dirichlet boundary conditions. Can one make sense (after a suitable renormalisation) of the random distribution $e^{i \beta \varphi}$ up the Berezinskii-Kosterltitz-Thouless critical point $\beta_{BKT}$ ?  It is easy to check that it exists thanks to Wick ordering up to $\beta_{L^2} < \beta_{BKT}$, see \cite{lacoin2015complex,junnila2020imaginary}. 
Our singularity result about the hierarchical Sine-Gordon measure, Theorem \ref{th.mainSG}, implies that there does no exist such distributions (modulo the existence of an exponential moment such as~\eqref{e.Hexp}). In particular when $\beta_{L^2} \leq \beta < \beta_{BKT}$ the field $e^{i \beta \varphi}$ only exists in the sense of $k$-point correlation functions. 
\item Let $\varphi$ be a GFF on $\R^d$. Can one make sense (after a suitable renormalisation) of the distribution $\varphi^4$ ? In $d=2$ and up to $d=8/3$ (in the hierarchical sense) this may be done thanks to Wick ordering. The main result of Section \ref{s.Phi} then implies the non-existence of such a distribution above $d=8/3$. 
\ei

\subsection{Short introduction to the main objects.}\label{ss.intro}

Let us now briefly describe  the three main models considered in this work, i.e. 
\bnum
\item The critical and near-critical Ising model scaling limits (Section \ref{s.Ising})
\item The hierarchical $2d$ Sine-Gordon model (Section \ref{s.SG}) 
\item The hierarchical $3d$ $\Phi^4$ model (Section \ref{s.Phi}) 
\enum

\ni
{\em Ising model.} Consider the Ising model on the small mesh lattice $\frac 1 n \Z^2$ with $n\in \N^*$. This means spins are sampled in $\Lambda \subset \frac 1 n \Z^2$ according to $\Pb{\sigma}\propto \exp(\beta \sum_{i\sim j} \sigma_i \sigma_j)$ for any spin configuration $\sigma \in \{\pm 1\}^\Lambda$ where $\beta$ is the inverse temperature. One can consider either the infinite volume limit $\Lambda \nearrow \frac 1 n \Z^2$ or stick to a finite window $\Lambda = \Lambda^{(n)}:= \frac 1 n \Z^2 \cap \Omega$ for some bounded domain $\Omega \subset \C$. 
If $\beta$ is set to be the critical inverse temperature $\beta_c=\frac 1 2 \log(1+\sqrt{2})$ then it is now well known that  as $n\to \infty$, a beautiful scaling limit emerges (\cite{smirnov2010conformal,chelkak2012universality}). Several natural topologies may be considered to capture the scaling limit, see the discussion at the beginning of Section \ref{s.Ising}. 
If one suitably rescales the Ising model near its critical inverse temperature $\beta_c$, an interesting scaling limit is expected to arise, called the {\em massive} or the {\em near-critical} limit of the Ising model in the $\beta$-direction (N.B. there is also a near-critical Ising model in the $h$-direction, \cite{camia2015planar,camia2016planar}). The appropriate tuning for $\beta\to \beta_c$ is proved in \cite{duminil2022planar} to be as follows
\begin{align*}\label{}
\beta = \beta_c + \frac \lambda n 
\end{align*}
for any $\lambda\in \R \setminus \{0\}$. (N.B. This scaling was also found in \cite{duminil2014near} up to log corrections). 
In Section \ref{s.Ising}, we consider sub-sequential scaling limits of near-critical lsing and we prove that for a suitable topology the limiting measures are singular w.r.t. the critical scaling limit (Theorem \ref{th.mainI}).

\medskip

\ni
{\em Hierarchical Gaussian Free Field.} The mathematical analysis of hierarchical versions of classical spin systems is  often more accessible than their non-hierarchical counter-part (see for example \cite{dyson1969existence}). 
On the side of the Renormalisation Group, hierarchical models are particularly convenient since they correspond to a dynamical system operating on one-variable functions. See \cite{gawkedzki1983non} or the recent book \cite{bauerschmidt2019introduction} or Sections \ref{s.SG} and \ref{s.Phi} where this will play a key role. The building brick of {\em hierarchical RG} is the {\em hierarchical GFF} which replaces the classical Gaussian Free Field (GFF) in a hierarchical setting. It is defined through a recursive procedure in such a way that it mimics the covariance structure of the GFF on $\R^d$. The price to pay will be that the Euclidean metric will be traded for a hierarchical ultrametric. 
Let us briefly introduce the hierarchical GFF in its two classical different forms: the \textbf{ultra-violet (UV)} presentation and the \textbf{infra-red (IR)} form. They are related to each other by a simple scaling rule (see~\eqref{e.SR} below). We present both versions as the IR form is more commonly used to run the RG flow while the UV form is naturally embedded in the continuum space $\R^d$. See for example \cite{bauerschmidt2019introduction,hutchcroft2024critical} for  recent references on hierarchical systems in the infrared (IR) setting.

Let us first set some notations which will be used for both IR and UV forms. 
Let us define
\footnote{In the literature, in the IR case, it is more common to introduce instead $\Lambda_N:=\{x\in\Z^d:-\hf L^N<x_i<\hf L^N, i=1,\dots, d\}$ so that it eventually covers the entire $\Z^d$. Note that up to translation, this is the same setting as ours.} 
 for $L \geq 2$
\begin{align}\label{e.Lambda}
& \Lambda_N:=\{x\in\Z^d: 0 \leq x_i  < L^N, i=1,\dots, d\}  \nn \\
& \bar \Lambda_N : = \frac 1 {L^N} \Lambda_N \subset [0,1]^d
\end{align}
\bi
\item[a)] {\em Infra-red (IR) hierarchical GFF $\phi^N$ in $d$ dimensions.} 
The hierarchical GFF $\phi^N(x)$, $x\in\Lambda_N$ is given by
\begin{equation}\label{e.phiN}
\phi^N(x)=\sum_{n=0}^{N}L^{-\frac{d-2}{2}n}z^{(N-n)}_{[\frac{x}{L^n}]}
 \end {equation}
 where $[\cdot]$ denotes integer part. Here $z^{(n)}_x$, $x\in\Lambda_{n}$, $1 \leq n \leq N$ are independent Gaussian processes with covariance
 \begin{align}\label{e.COV}
\begin{cases}
\Eb{z^{(n)}_xz^{(n)}_y}=\delta_{xy} \;\;\; \text{ in Section \ref{s.SG} }\\
\Eb{z^{(n)}_xz^{(n)}_y}=\delta_{xy}-L^{-d}\delta_{[\frac{x}{L}],[\frac{y}{L}]} \;\;\;  \text{ in Section \ref{s.Phi}}
\end{cases}
 \end {align}
and $z^{(0)}_0$ is unit Gaussian. 
(N.B. The second covariance structure of $z^{(n)}, n\geq 1 $ will  ensure in Section \ref{s.Phi} that $\sum_{x\in \text{ one $L\times L$ box}} z^{(n)}_x =0$. Indeed $\Var{\sum z_x}) = L^d(1-L^{-d}) + L^d(L^d-1)*(-1)*L^{-d}=0$).
\item[b)] {\em Ultra-violet (UV) hierarchical GFF $\varphi^N$ in $d$ dimensions.} 
It is defined on the continuum cube $[0,1)^d$ as follows: for each $\bar x \in [0,1)^d$ 
\begin{align}
\varphi^N(\bar x) & =\sum_{n=0}^{N}L^{\frac{d-2}{2}(N-n)}z^{(N-n)}_{[L^{N-n} \bar x ]}\label{e.varphiN} \\
& = L^{\frac {d-2} 2 N} \phi^N([ L^N \bar x ] ) \label{e.SR}
\end{align}
The second equality gives us the simple scaling which relates the UV field $\varphi^N(\bar x)$ with the IR field $\phi^N(x)$.
(We used the fact that $[ L^{-n} [L^N \bar x] ] = [L^{N-n} \bar x ]$). In the rest of this text, we will use the variable $x$ in both IR and UV settings when the context will be clear. Also, when comparing the  $UV$ hierarchical field with the GFF on $\R^d$, we will add a subscript ``$\hr$'' to $\varphi^N_\hr$ to make the comparison clearer. 
\ei
%
When $d=2$, the hierarchical field $\varphi^N$ has the following simple expression on $[0,1)^2$
\begin{align*}\label{}
\varphi^N(x) := \sum_{n=0}^{N}  z^{(n)}_{[ L^{n}  x ]}\,.
\end{align*}
It is easy to check that in any Sobolev space $H^{-\eps}$, the field  $\varphi^N_\hr(x) = \varphi^N(x)$ converges in law as $N\to \infty$ to a limiting \textbf{hierarchical GFF} on $[0,1]^2$ $\varphi_\hr$ whose covariance structure is given by 
\begin{align*}\label{}
\Cov{\varphi_\hr(x),\varphi_\hr(y)} = \log_L \dist_\hr(x,y)^{-1}
\end{align*}
where $\dist_\hr$ is the hierarchical ultrametric on $(0,1)^2$ defined by 
\begin{align*}\label{}
d_\hr(x,y):= L^{-h(x,y)}\,,
\end{align*}
and where $h(x,y)$ is the largest $k\geq 0$ so that both $x$ and $y$ belong to the same $L$-adic square of side-length $L^{-k}$. As such 
$\varphi_\hr$ is a good tree-like approximation of the isotropic $2d$ GFF. One nice feature of this field is that all $\{ \varphi^N_\hr\}_{N\geq 1}$ are naturally coupled with each other (they share the same Gaussian variables on the first layers). 

If instead $d=3$, we now have $\varphi^N_{d=3,\hr}\to \varphi_{d=3,\hr}$ in any $H^{-1/2-\eps}$ and 
\begin{align}\label{e.3dgffH}
\Cov{\varphi_{d=3,\hr}(x),\varphi_{d=3,\hr}(y)} = \frac{L^{h(x,y)+1}-1}{L-1}\asymp \frac 1 {\dist_\hr(x,y)}\,.
\end{align}


\medskip
\ni
{\em Hierarchical QFT fields.} Let us now shortly introduce the hierarchical Sine-Gordon and $\phi_3^4$ models which will be the focus of Sections \ref{s.SG} and \ref{s.Phi}.

Recall the $2d$ Sine-Gordon model on the two-dimensional torus $\T^2$ is informally defined as follows for $\beta\in[0,8\pi)$:
\begin{align*}\label{}
\mu_{\beta,\mu}^{SG}(d\varphi) ``\propto \exp\Big(- \mu \int_{\T^2} \Wick{\cos(\sqrt{\beta} \varphi(x))} dx \Big) \mu_{GFF}(d\varphi) "\,,  
\end{align*}
where $\Wick{\cos(\sqrt{\beta} \varphi(x))}=\lim_{\eps\to 0} e^{+\frac \beta 2 \Var{\varphi_\eps(x)}}\cos(\sqrt{\beta} \varphi_\eps(x))$ stands for the so-called {\em Wick ordering}.  
When $\beta<4\pi$, the Sine-Gordon field is easily seen to be absolutely continuous w.r.t the GFF, but when $\beta\in[4\pi,8\pi)$ its construction is more involved. See \cite{hairer2016dynamical,chandra2018dynamical,gubinelli2024fbsde} as well as \cite{lacoin2023probabilistic} in the 1d setting.  It has been shown in \cite{gubinelli2024fbsde} that the SG field in $2d$ is singular with respect to the GFF in the regime $\beta\in[4\pi, 6\pi)$. 
In Section \ref{s.SG}, we will define the hierarchical version of the Sine-Gordon field and we will prove that it is singular w.r.t $\varphi_\hr$ for all $\beta\in[\beta_{L^2},\beta_{BKT})$ (where $\beta_{L^2}$ plays the same role as $4\pi$ in $\R^2$ and $\beta_{BKT}$ corresponds to $8\pi$ in $\R^2$). The hierarchical Sine-Gordon field is informally defined as follows
\begin{align}\label{e.SGF}
\mu_{\hr,\beta,\mu}^{SG}(d\varphi)  ``\propto \exp\Big(- \mu \int_{[0,1]^2} \Wick{\cos(\sqrt{\beta} \varphi(x))} dx \Big) \mu_{\hr}(d\varphi) "\,.  
\end{align}

\medskip

In the same way, inspired by the construction of the $\phi^4_3$ field on $\R^3$ (see \cite{glimm1968boson,glimm1973positivity,feldman1974lambdaphi,hairer2014theory,gubinelli2015paracontrolled,kupiainen2016renormalization,barashkov2021varphi}), we shall define and analyse its hierarchical version in Section \ref{s.Phi} which may be informally written as follows:
\begin{align*}\label{}
\mu^{\phi^4_3}_{\hr,\lambda}(d\varphi)  ``\propto \exp\Big(-  \int_{[0,1]^3} \big[\lambda\,{\varphi(x)^4} - \infty \times \varphi(x)^2\big] dx \Big) \mu_{\hr}(d\varphi) "\,,  
\end{align*}
where $\infty$ stands for the  $a\lambda\epsilon^{-1}+b\lambda^2\log\frac 1 \eps$ counter term which needs to be added 
in order to produce a non-trivial QFT field. 

\medskip
\ni
{\em Hierarchical RG flow and effective potentials.}
The later will be a key feature of Sections \ref{s.SG} and \ref{s.Phi}. Suppose we wish to define the probability measure informally defined in~\eqref{e.SGF}.  As one would do in the Euclidean case, we look for a regularisation scheme $\varphi \to \varphi_\eps$. In the hierarchical setting, it is particularly convenient as we may simply consider $\varphi_\hr \to \varphi^N_\hr$ (which sets to zero all gaussians variables at scales smaller than $L^{-N}$). In the Sine-Gordon case (equation~\eqref{e.SGF}), let us choose some large integer $N$ and let us consider the potential 
\begin{align*}\label{}
v^{(N)}_N(\varphi) := a_N \cos( \sqrt{\beta} \, \varphi)\,,
\end{align*}
where the {\em coupling constant} $a_N$ needs to be well chosen for the later RG flow to be well behaved. This potential allows us to consider the following probability measure on fields:
\begin{align}\label{e.SGF2}
\mu^{(N)}(d\varphi) := \frac 1 {Z^{(N)}_{N}}  \exp\Big(- \sum_{x\in \bar \Lambda_N } v_N^{(N)} (\varphi^N(x))  \Big) \mu_{\hr}(d\varphi)\,,  
\end{align}
where $\varphi^N(x)$ is defined as in~\eqref{e.SR},  and $x$ runs over the points in $\bar \Lambda_N$ defined in~\eqref{e.Lambda} 
\begin{remark}\label{}
Given the informal definition~\eqref{e.SGF}, it would  have been probably more natural to consider instead the definition 
\begin{align*}\label{}
\tilde \mu^{(N)}(d\varphi) := \frac 1 {Z^{(N)}_{N}}  \exp\Big(- \int_{[0,1]^2} v_N^{(N)} (\varphi^N(x))  dx \Big) \mu_{\hr}(d\varphi)\,. 
\end{align*}
It will turn out to be more convenient to set up our RG flow with the former definition~\eqref{e.SGF2} (in any case the only difference between the two involves an additional  $L^{dN}$ factor in the coupling constant $a_N$).   
\end{remark}

The hope is then to show that if $\{a_N\}_{N\geq 1}$ is well tuned, then 
\begin{align*}\label{}
\mu^{(N)} \underset{N\to \infty}\longrightarrow \mu^{SG}_{\hr, \beta}
\end{align*}
The suitable choice of coupling constants together with a proof of such a convergence will go through the so-called {\em hierarchical RG flow}. It may be introduced as follows: the first iteration of the RG flow will operate from scale $L^{-N}$ up to scale $L^{-N+1}$ and will seek for a potential $v_{N-1}^{(N)}$ such that if $f_{N-1}$ is a test function which is constant on $L$-adic squares of side-length $L^{-N+1}$, then the following identity holds
\begin{align*}\label{}
\mu^{(N)}\big[ \<{f_{N-1},  \varphi^{N-1}} \big] & =  \frac 
1 {Z^{(N)}_{N-1}}  \int  \<{f_{N-1}, \varphi^{N-1}}\exp\Big(-  \sum_{x\in \bar \Lambda_{N-1}} v_{N-1}^{(N)} (\varphi^{N-1}(x))  \Big) \mu_{\hr}(d\varphi)\,.  
\end{align*}
It is an interesting computation (by integrating out the Gaussian variables at scale $L^{-N}$, i.e $z^{(N)}_{x}$) to check that it is indeed achieved by the following renormalized potential
\begin{align*}\label{}
v^{(N)}_{N-1}(\varphi)& :=  - \log\Big(\Eb{e^{-v^{(N)}_N(\varphi+z)}}^{L^2}\Big)\,,
\end{align*}
where the expectation $\E$ is w.r.t to a normal Gaussian variable $z$. This brings us to introduce the following RG flow $\calR$ acting on functions of one real variable: 
\begin{align}\label{e.RG0}
[\calR v](\varphi) := - L^2 \log \Big( \int_\R dx \frac{e^{-x^2}}{\sqrt{2\pi}} e^{-v(\varphi+x)} \Big)\,.
\end{align}
One may then iterate this hierarchical RG flow and obtain for any intermediate $1\leq n \leq N$ the {\em effective potential} $v_n^{(N)}$ describing the fluctuations at scale $L^{-n}$:
\begin{align*}\label{}
v_n^{(N)} := \calR^{N-n} v_N^{(N)} 
\end{align*}
Our definition of the Sine-Gordon hierarchical field in Definition \ref{d.SG} will include the fact that if the coupling constants $\{a_N\}_N$ are tuned the right way then one obtains for each $n\geq 1$ non-trivial asymptotic effective potentials 
\begin{align*}\label{}
v_n^{\infty} := \lim_{N\to \infty} v_n^{(N)}\,.
\end{align*}
As for the $\phi^4_3$ field, a very similar hierarchical RG flow will be used in Section \ref{s.Phi} in order to construct and analyse the hierarchical $\phi^4_3$ field  out of the $3d$ hierarchical GFF from~\eqref{e.3dgffH}.


\subsection{Main ideas of proofs.}

\begin{figure}[!htp]
\begin{center}
\includegraphics[width=\textwidth]{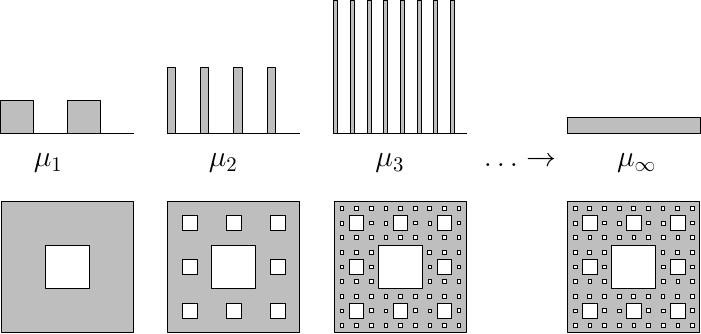}
\end{center}
\caption{Illustration in a simpler setting of our main difficulty to identify singularity. The first line represents the sequence of measure $\{\mu_n\}$ on the unit interval $[0,1]$ with Radon-Nikodim derivative $d \mu_n(x) = \big(\sum_{i=0}^{2^n-1} 2^{2n} 1_{[i 2^{-n}, i 2^{-n} + 2^{-3n}]} \big)(x) dx $. These measures are more and more {\em singular} as $n\to \infty$ with respect to Lebesgue measure. Yet, their limiting measure is nothing but the Lebesgue measure on $[0,1]$.  We thus need to prove that the three QFT fields we are interested in  behave more like the second line with fractal behaviour. Namely the singularity spotted on an $\eps$ regularisation remains after taking the {\em UV limit} in the RG flow. }\label{f.image1}
\end{figure}

Let us first highlight the main ideas and difficulties of the proofs in Sections \ref{s.Ising},\ref{s.SG} and \ref{s.Phi} by considering the following more classical situation of a singular measure $\nu$ (w.r.t to a second measure $\mu$, say) for which we have a sequence of approximations measures $\{\nu_n\}_n$ so that $\nu_n \to \nu$.  
If the measures $\nu_n$ have a simpler structure than the limiting measure $\nu$, it is then tempting to prove that $\nu \perp \mu$ by checking that $\nu_n$ is ``more and more singular'' w.r.t $\mu$ for large $n$ and then pass to the limit. 

This is of course far from sufficient as it is illustrated in the top example of Figure \ref{f.image1}. As we shall explain below, such degenerate cases may happen in the context of QFT fields. In the simpler case of measures on $\R^d$ (as opposed to measures on $\calS(\R^d)$ or $H^{-a}(\R^d)$), one may still prove $\nu \perp \mu$ using $\nu_n \to \nu$ if the singular behaviour of $\nu_n$ is visible at all fixed  ``mesoscopic scales'' uniformly as $n\to \infty$. This is typically what happens with measures supported on fractal sets as pictured in Figure \ref{f.image1} and this is reminiscent of what we will look for in the setting of QFT fields.

\smallskip
\ni
{\em Singularity for near-critical Ising model. (Section \ref{s.Ising}).}
 
Let us fix the near-critical parameter to be $\lambda \neq 0$ and let us consider a large integer $n \gg 1$. Let $\Lambda_n$ be the finite graph $\frac 1 n \Z^2 \cap [-1,1]^2$. 
Before taking the scaling limit, we may compare the critical and near-critical measures defined on the same configuration space $\{-1,1\}^{\Lambda_n}$ by 
\begin{align*}\label{}
\begin{cases}
& \FK{\Lambda_n,\beta_c}{}{\sigma} \propto \exp(\beta_c \sum_{i\sim j} \sigma_i \sigma_j)    \\
& \text{   and } \\
& \FK{\Lambda_n,\beta_c,\lambda}{\mathrm{n.c}}{\sigma} \propto \exp((\beta_c + \frac \lambda n) \sum_{i\sim j} \sigma_i \sigma_j) \propto \FK{\Lambda_n,\beta_c}{}{\sigma}\exp(\frac \lambda n \sum_{i\sim j} \sigma_i \sigma_j)
\end{cases}
\end{align*}

In Section \ref{ss.discrete}, we will prove that $\P_{\Lambda_n,\beta_c}$ and $\P_{\lambda_n,\beta_c,\lambda}^{\mathrm{n.c}}$ are more and more singular as $n\to \infty$ in the sense that one can find events $A_n\subset \{-1,1\}^{\Lambda_n}$ such that 
\begin{align*}\label{}
 \FK{\Lambda_n,\beta_c}{}{A_n} \to 1 \text{    while    }  \FK{\Lambda_n,\beta_c,\lambda}{\mathrm{n.c}}{A_n} \to 0 \,. 
\end{align*}
The main idea to achieve this is to prove that the random variable representing the energy, i.e. $E_n(\sigma):=  \sum_{i\sim j} \sigma_i \sigma_j$ has fluctuations much higher than $n$ when $n$ is large and when $\sigma\sim \P_{\Lambda_n, \beta_c}$. This implies in particular that the Radon-Nykodim derivative $\frac {d \P_{\Lambda_n,\beta_c,\lambda}^{\mathrm{n.c}}} {d\P_{\Lambda_n, \beta_c}}$ is more and more singular on the space $\{-1,1\}^{\Lambda_n}$.  We then face two difficulties:
\bi
\item[a)] First, there is no obvious limit for the state space $\{-1,1\}^{\Lambda_n}$ as $n\to \infty$. A suitable topology needs to be used in order to take the scaling limit.
\item[b)] One then needs  to check that the singularity still prevails in the scaling limit. This is far from obvious as an analogous situation to the example illustrated in Figure \ref{f.image1} can be constructed in the setting of the Ising model as we now explain.  
\ei
\begin{example}\label{ex.1}
Imagine we consider instead a near-critical Ising measure whose temperature shift is 
\begin{align}\label{e.CE}
\tilde \beta = \beta_c + \frac {\lambda} {(\log n)^{1/10} n} \text{   instead of  } \beta = \beta_c +  \frac {\lambda} {n}\,.
\end{align}
We claim that the analysis of Section \ref{ss.discrete} still applies and proves that $\P_{\Lambda_n, \beta_c +\frac {\lambda} {(\log n)^{1/10} n} }$ and $\P_{\Lambda_n, \beta_c}$ are asymptotically singular. Yet, interestingly, it can be checked that in the scaling limit (under the topology considered in section \ref{s.Ising}), both limits are in fact equal! We thus found the same type of counter-example as the one illustrated in Figure \ref{f.image1} in the less classical setting of the critical Ising model. 
\end{example}

This example shows that step $b)$ requires some care. In order to detect a singular behaviour in the scaling limit, we will construct in Section \ref{s.Ising} an event based on the behaviour of FK-clusters at small mesoscopic scales (which in a sense are uniform as the mesh $1/n \to 0$).

\medskip
\ni
{\em Singularity for hierarchical Sine-Gordon and $\phi^4_3$ (Sections \ref{s.SG} ans \ref{s.Phi}).}

Let us explain the main idea in the case of the hierarchical Sine-Gordon field. As explained above, the measure will be well approximated by the following one for large $N$:
\begin{align}\label{e.SGF3}
\mu^{(N)}(d\varphi) := \frac 1 {Z^{(N)}_{N}}  \exp\Big(- \sum_{x\in \bar \Lambda_N} v_N^{(N)} (\varphi^N(x))  \Big) \mu_{\hr}(d\varphi)\,,  
\end{align}
with $v_N^{(N)}(\varphi) = a_N \cos(\sqrt{\beta} \varphi)$ and $a_N$ tuned the right way. 
It will be rather straightforward\footnote{this step will be  much easier than in the Ising case} to check that this measure is asymptotically (as $N\to \infty$) more and more singular with respect to the Gaussian measure $\mu_\hr$ when $\beta \geq \beta_{L_2}$. This will follow from the fact that if $\varphi^N \sim \mu_{\hr}$, then the energy $ \sum_{x\in \bar \Lambda_n} v_N^{(N)} (\varphi^N(x)) $ has a positive probability to take very large negative values as $N\to \infty$.  Note that at this stage we only need $\beta \geq \beta_{L_2}$ and the other constraint $\beta<\beta_{BKT}$ is not visible yet. This latter constraint only appears when taking the scaling limit. 

Once again the same situation as in Figure \ref{f.image1} may arise in this context. Despite this singularity as small scales, it could well be that as $N\to \infty$, we only reproduce the Gaussian field $\mu_\hr$ or something absolutely continuous. In fact it is not hard to build such an example:
\begin{example}\label{ex.2}
In the setting of the Sine-Gordon field, we claim that if we choose the coupling constants $\tilde a_N$ to be slightly too small, namely $\tilde a_N:= N^{-1/10} a_N$  in Section \ref{s.SG}, then for any $\beta \geq \beta_{L_2}$ we still have enough wide energy fluctuations to ensure that $\tilde \mu^{(N)}$ is more and more singular w.r.t $\mu_\hr$. Yet, if $\beta\in [\beta_{L^2}, \beta_{BKT})$ the proof in Section \ref{s.SG} also shows that the effective potentials $\tilde v_n^{(N)}$ are converging to zero so that no singularity survives in the scaling limit $N\to \infty$!
\end{example}

This means that in the QFT setting, we may build potentials $v_N^{(N)}$ which create very singular behaviour at small scales, and yet for which limiting effective potentials are better behaved. Similarly as in for the Ising model, the key aspect of the proof is then to keep track of the singularity behaviour at ``mesoscopic scales'' uniformly as the mesh $L^{-N}\to 0$.  This will boil down to obtaining a good control on the limiting effective potentials for each fixed $n\geq 1$: 
\begin{align*}\label{}
v_n^\infty  := \lim_{N\to \infty}  v_n^{(N)} = \lim_{N\to \infty} \calR^{N-n} v_N^{(N)} \,.
\end{align*}
This is precisely where the second condition $\beta < \beta_{BKT}$ is important. Indeed, above this threshold, the hierarchical RG flow $\calR$ stops being super-renormalizable in the UV limit. The key estimates which allow us to spot the singularity are found in Lemma \ref{l.LLN} and Lemma \ref{l.liminf}.


%
%

$ $ 

\smallskip
\ni
\textbf{Acknowledgments.}
The first author wishes to warmly thank Clément Hongler for illuminating discussions on the interplay between singularity and non-existence of fields as random distributions. These discussions in the case of the Ising model were the motivations to look into the analogous questions for the $\Phi^4_3$ and Sine-Gordon fields.  The authors wish to thank Martin Hairer for several important discussions at the beginning of this project as well as Hugo Duminil-Copin and Ioan Manolescu for adding specific key estimates in their key work \cite{duminil2022planar} which made the writing of this paper much easier. We also thank Abdelmalek Abdesselam, Paul Cahen, Alessandro Giuliani, Massimiliano Gubinelli, Rémy Mahfouf  and Slava Rychkov for fruitful discussions. Finally, we wish to thank the anonymous referee for a very careful reading of the manuscript.

C.G. acknowledges support from the Institut Universitaire de France (IUF), the ERC grant VORTEX 101043450 and the French ANR grant ANR-21-CE40-0003. A.K. was partially supported by the Academy of Finland.

\section{Near-critical Ising model}\label{s.Ising}

\medskip

Our goal in this section is to prove that in a compact window, any (subsequential) scaling limit of near-critical Ising model (for a near-critical perturbation in the temperature direction) is singular with respect to the scaling limit of critical Ising model. Before stating our main result, let us make a few remarks:
\bnum
\item The singularity statement is more interesting on a {\em compact region} (it detects the ultra-violet singularity). It would also hold for the scaling limit in infinite volume (i.e. on $\R^2$) but this would be much easier to detect. 

For this reason, we consider the Ising model on $\Omega=[-1,1]^2 \subset \R^2$. It may be equipped with several  natural boundary conditions:
\bi
\item[a)] Plane boundary conditions (i.e. we consider the scaling limit in the full $\R^2$ and consider its restriction to the window $[-1,1]^2$). This is the main setup considered in \cite{duminil2022planar} 
\item[b)] Periodic boundary conditions
\item[c)] $+/$free boundary conditions around $[-1,1]^2$. 
\ei 

\item The near-critical scaling limit of the Ising model in the {\em magnetic direction} has already been analysed in \cite{camia2016planar,camia2020exponential} and turns out to be {\em absolutely continuous} in compact regions w.r.t to the critical model. 

\item The correct notion of near-critical limit is provided to us by \cite{duminil2022planar} (see also \cite{duminil2014near} for up to log estimates as well as \cite{park2022convergence} for 2-point and 4-point correlation functions). For any of the above boundary conditions $a)$ to $d)$, we consider the Ising model in $\frac 1 n \Z^2 \cap \Omega$ and at inverse temperature
\begin{align*}\label{}
\beta_{\lambda,n}:= \beta_c+ \frac \lambda n\,,
\end{align*}
where $\lambda\in \R$ is a fixed near-critical parameter. We wish to show that if $\lambda \neq 0$, then as $n\to \infty$,  any subsequential scaling limit is singular w.r.t to the critical ($\lambda=0$) limit. 

\item To detect such a singularity, we need to choose a suitable topology for the limiting process. In the case of the Ising model, it turns out there are several natural spaces in order to consider the scaling limit: 
\bi
\item[a)] {\em Space of Fields.} For example by choosing the Sobolev space $H^{-1/8-\eps}(\Omega)$ (see \cite{furlan2017tightness,camia2015planar}) 
\item[b)] {\em Spin clusters, $\mathrm{CLE}_3$.} 
\item[c)] {\em FK clusters, $\mathrm{CLE}_{16/3}$.}  
\ei
\enum
\ni
We will prove that singularity holds under an FK percolation scaling limit (item $c)$). 

Let us start by checking (only at the informal level)  that when the discrete mesh is still visible, the critical and near-critical measures are indeed nearly singular. As discussed in the introduction, this is by no means a guaranty that the limiting measures will indeed be singular (see Examples \ref{ex.1} and \ref{ex.2}). 

\subsection{(Near)-singularity at the discrete level.}\label{ss.discrete}

Let $\lambda\neq 0$ be fixed and let us analyse boundary conditions $c)$ with free boundary conditions, namely  $[-1,1]^2 \cap \frac 1 n \Z^2$ equipped with free boundary conditions. 
\begin{align*}\label{}
\begin{cases}
& \mu_{n}^0(\sigma):= \frac 1 {Z_n^0} \exp \left( \beta_c \sum_{i\sim j} \sigma_i \sigma_j  \right) \\
& \mu_{n}^\lambda(\sigma):= \frac 1 {Z_n^\lambda} \exp\left(  (\beta_c + \frac \lambda n) \sum_{i\sim j} \sigma_i \sigma_j  \right) 
\end{cases}
\end{align*}
To compare both measures, let us introduce $E_n^{bulk}$ to be the {\em amount} of energy inside the sub-domain
\begin{align*}\label{}
I_n:=[-\tfrac 1 2 , \tfrac 1 2] \cap \frac1 n \Z^2 \subset \Omega_n:=[-1,1]^2 \cap \frac 1 n \Z^2\,.
\end{align*}
\begin{align*}\label{}
E_n^{bulk} = E_n^{bulk}(\sigma):= \sum_{i\sim j \in I_n} \sigma_i \sigma_j
\end{align*} 
We start with the following Lemma which follows readily from the analysis in \cite{duminil2022planar}. 
\begin{lemma}\label{l.varE}
For any fixed $\lambda \in \R$ (including $\lambda=0$ which is the critical measure), 
\begin{align*}\label{}
\mathrm{Var}_{\mu_n^\lambda}\big[\frac {E_n^{bulk}(\sigma)} n \big] \asymp \log(n)\,.
\end{align*}
Furthermore, the constants involved in $\asymp$ are uniform over $\lambda \in K$ for any compact set $K\in \R$ (but do depend on the choice of the compact set $K$). 
\end{lemma}

\ni
{\em Proof of Lemma \ref{l.varE}.}
We first recall the classical fact that the $k$-point correlation function  (for $k\in \{1,2\}$) of the energy field are up to constant the same as the $k$-point correlation functions of the FK-Ising percolation. Indeed, 
\bi
\item For $k=1$ (one edge). If $i\sim j$. Then 
\begin{align*}\label{}
\FK{}{FK}{\omega_{ij}=1} & = \Pb{\sigma_i=\sigma_j \text{ and  the edge }e=(i,j) \text{   is open }}\\
& =\Pb{\sigma_i=\sigma_j} p \\
&= \frac p 2 (\Eb{\sigma_i \sigma_j}+1)\,,
\end{align*}
where $p=p(\beta)$ and we relied on the classical FK / spin Ising coupling. 
\item For $k=2$, If $(i_1,j_1)$ and $(i_2,j_2)$ are two distinct edges, the same classical coupling leads us to 
\begin{align*}\label{}
\Cov{\omega_{i_1j_1}, \omega_{i_2j_2}}= \frac{p^2} 4 \Cov{\sigma_{i_1}\sigma_{j_1}, \sigma_{i_2}\sigma_{j_2} }
\end{align*}
\ei 
Using the estimate (1.32) from Theorem 1.13 in \cite{duminil2022planar}, we readily obtain 
\begin{align*}\label{}
\mathrm{Var}_{\mu_n^\lambda}\big[\frac {E_n^{bulk}(\sigma)} n \big]  & =  \frac {4} {p^2 n^2 }  \sum_{e_1 \neq e_2 \in I_n} \mathrm{Cov}_{\lambda}^{FK}(\omega_{e_1}, \omega_{e_2})  +  \frac 1 {n^2} \sum_{e=\{i,j\} \in I_n} (1 - \Eb{\sigma_i \sigma_j}^2)  \\
& \asymp  \log(n) 
\end{align*}
More details will be given below for similar sums which appear when dealing with subsequential scaling limits. 
\qed
\begin{remark}
The reason why we only consider the {\em bulk} contribution to the energy (i.e. inside $I_n=[-\tfrac 1 2, \tfrac 1 2]^2\cap \frac 1 n \Z^2$ at macroscopic distance from the boundary) is because  covariance estimates for edges close to the boundary would require additional work. See Remarks 5.5 and 6.7 in \cite{duminil2022planar}. 
\end{remark}

Our second Lemma establishes that the mean value of the bulk energy $E_n^{bulk}$ evolves faster than the fluctuations under the near-critical perturbation $\beta=\beta_c + \frac \lambda n$.  

\begin{lemma}\label{l.mean}
For any $\lambda > 0$, there exists a constant $c=c(\lambda)>0$ such that 
\begin{align*}\label{}
\mu_n^\lambda(\frac {E_n^{bulk}} n) - \mu_n^0( \frac {E_n^{bulk}} n) \geq c \log(n). 
\end{align*}
\end{lemma}

\ni
{\em Proof of Lemma \ref{l.mean}.}
For any $\beta \geq 0$,
\begin{align*}\label{}
\frac d {d \beta} E_n^{bulk} & = \sum_{i\sim j \in I_n}  \frac d {d \beta}  \EFK{\Omega_n,\beta}{free}{\sigma_i \sigma_j} \\
& = \sum_{i\sim j \in I_n}  \frac d {d \beta} \big(\frac 2 p \FK{p}{FK}{\omega_{ij}=1}  -1 \big) \\
& =  \sum_{i\sim j \in I_n}  \left(  \frac d {dp} \FK{p}{FK}{\omega_{ij}=1}  \frac 2 p  - \frac 2 {p^2} \FK{p}{FK}{\omega_{ij}=1} \right) \frac d {d\beta} p(\beta)
\end{align*}
Recall $p=p(\beta) = 1- e^{-2\beta}$. Using Theorem 1.13 from \cite{duminil2022planar} we obtain the existence of a constant $c_1(\lambda)>0$, s.t. for any $p_c\leq p \leq 1- e^{-2 \beta_c - 2 \frac \lambda n}$ and any edge $e=\{i,j\}, i\sim j$ in $I_n$ (i.e. at macroscopic distance from the boundary), we have 
\begin{align*}\label{}
\frac d {dp} \FK{p}{FK}{\omega_{ij}=1}  \geq c_1(\lambda)  \log(n) \,.
\end{align*}
This implies
\begin{align*}\label{}
\frac d {d \beta} E_n^{bulk} & \geq  c_2(\lambda) n^2 \log(n)\,.
\end{align*}
Integrating this differential inequality along $\beta \in [\beta_c, \beta_c + \frac \lambda n]$ finishes the proof. 
\qed

The above two Lemmas are sufficient to find a (discrete) event making a clear distinction between $\mu_n^0$ and $\mu_n^\lambda$ as $n\to \infty$. 
\begin{proposition}\label{pr.singI}
For any $\lambda\neq 0$, there exists a sequence of events $A_n$ s.t. 
\begin{align*}\label{}
\mu_n^0(A_n) \to 0  \,\, \text{    while   } \,\, \mu_n^\lambda(A_n) \to \infty
\end{align*} 
\end{proposition}
Let us insist once again that such a result is by far not sufficient to imply singularity of the near-critical continuum limits. Yet, our proof will be inspired by the setup of this Proposition.

\ni
{\em Proof of the Proposition.}
Let us define the following sequence of events:
\begin{align*}\label{}
A_n:= \left\{  E_n^{bulk}  \geq \mu_n^0(E_n^{bulk}) + \log(n)^{3/4} \right\}
\end{align*}
Using Markov together with Lemma \ref{l.varE}, we obtain that 
\begin{align*}\label{}
\mu_n^0(A_n) \leq O(1) \frac 1 {\log(n)^{3/2}} \log(n) \to 0\,.
\end{align*}
On the other hand, using Lemma \ref{l.mean}, we have that for any $\lambda>0$ and $n$ sufficiently large, 
\begin{align*}\label{}
\mu_n^0(E_n^{bulk}) + \log(n)^{3/4}  \leq \mu_n^\lambda(E_n^{bulk}) - \log(n)^{3/4}\,.
\end{align*}
Using again Markov together with Lemma \ref{l.varE}, we conclude that  $\mu_n^\lambda(A_n^c) \to 0$ as desired. 
\qed

\smallskip

We will follow the same strategy for the continuum near-critical limits, except in that case, we no longer have access to the status of individual spins or edges. Instead, we will count the number of mesoscopic FK crossing events in ``small'' mesoscopic squares. 

\subsection{Existence of subsequential scaling limits.}
It is not known whether there is a unique near-critical Ising scaling limit. Let us state it as a conjecture. 

\begin{conjecture}\label{c.SL}
Let $\lambda \neq 0$ be fixed and let $(\omega_n,\sigma_n)$ be an FK-Ising coupling of an Ising model at $\beta = \beta_c + \frac \lambda n$ in $[-1,1]^2 \cap \frac 1 n \Z^2$ with any of the boundary conditions $a)$, $b)$, $c)$ listed at the beginning of this section. We may then view $(\omega_n,\sigma_n)$ as a random point in several possible topological spaces:
\bi
\item FK clusters with prescribed colours. Possible topologies here are the Schramm-Smirnov {\em quad-topology} from \cite{schramm2011scaling} (see also \cite{garban2018scaling}). Another natural choice here would be the space of coloured loops equipped with an Hausdorff-like topology (\cite{camia2006two}). In fact the equivalence between both topological spaces  has been shown in \cite{camia2006two, garban2018scaling,holden2023convergence} in the case of $q=1$ critical percolation (what this means here is that if $(\omega, \Gamma)$ are the scaling limits of  critical percolation under the quad topology and the loop topology, then $\Gamma$ is a.s. determined by $\omega$ and vice-versa). 
\item We may view $\sigma_n$ as a random magnetic field $\Phi_n$ in some negative index Sobolev space (following \cite{camia2015planar, camia2016planar, furlan2017tightness}). 
\item We may view $\sigma_n$ as a collection of loops separating $+$ and $-$ clusters (at criticality this leads to the nested $\CLE_3$ limit from \cite{benoist2019scaling}). 
\ei
Under each of these topologies, and when $\lambda\neq 0$ is fixed,  $(\omega_n,\sigma_n)$ should converge to a unique limiting $\lambda$-near-critical scaling limit. 
\end{conjecture}

On the other hand, subsequential scaling limits are known to exist. We state it as a (folklore knowledge) Proposition.
\begin{proposition}\label{pr.tight}
Under each of the above choices of topological spaces, the near-critical  FK-Ising model $(\omega_n, \sigma_n)$ at $\beta= \beta_c + \frac \lambda n$ is tight as $n\to \infty$. 
\end{proposition}

This is obvious in the case of the {\em quad-crossing} topology (which is a compact metric space). For loop topologies ($\CLE_3$ or $\CLE_{16/3}$), this can be seen as a consequence of RSW / strong RSW (following \cite{duminil2011connection,chelkak2016crossing}) as well as the stability of these RSW estimates in the near-critical window (\cite{duminil2022planar}). Finally, tightness for the magnetization field follows from the same arguments as in \cite{camia2015planar} together with the stability of the one-arm exponent of FK in the near-critical regime which is derived  in \cite[Theorem 1.4]{duminil2022planar}. We shall not give further details on tightness here. 

\begin{remark}\label{}
The above conjecture is also supported by the work \cite{park2022convergence} which proves convergence of the 2-point and 4-point fermionic obersvables for the associated massive FK-Ising percolation.
\end{remark}

\subsection{Singularity under $\CLE_{16/3}$ scaling limit.}

Let $(\sigma_n,\omega_n)$ an FK-Ising model in $[-1,1]^2 \cap \frac 1 n \Z^2$ equipped with boundary conditions a),b) or c) listed at the beginning of the present section. This coupling induces a pair of quad crossing configurations $(\omega_n^+, \omega_n^-)$ which belongs to $\calH^2$, where $(\calH, d_\calH)$ is the quad crossing space from \cite{schramm2011scaling,garban2018scaling}. In words $\omega_n^+$ (resp. $\omega_n^-$) provides all the topological quads which are traversed from left to right by an FK cluster with $+$ colour (resp. $\omega_n^-$). We refer to  \cite{schramm2011scaling,garban2018scaling} for a precise definition of the compact space metric space $(\calH, d_\calH)$. 

As observed before, by compactness of $(\calH, d_\calH)$, subsequential scaling limits of $(\omega_n^+,\omega_n^-)$ do exist.  

\begin{theorem}\label{th.mainI}
For any fixed $\lambda \neq 0$, any subsequential scaling limit of $(\omega_n^+, \omega_n^-) \in \calH^2$ under $\P^{Ising}_{\beta_c + \frac \lambda n}$ with boundary conditions a), b) or c) is singular with respect to the (unique) critical scaling limit ($\beta=\beta_c$). 
\end{theorem}

\begin{remark}
The proof below will in fact detect the singularity by relying only on the non-coloured FK percolation $\omega:=\omega^+ \cup \omega^- \in \calH$ (where $\cup$ is a natural measurable function on $\calH$). Yet, we prefer to state it this way since, at least on the discrete level, the coupling $(\omega_n^+, \omega_n^-)$ contains all the information from $\sigma_n$ (which is not the case for the single FK percolation $\omega_n$). 
\end{remark}

\ni
{\em Proof of Theorem \ref{th.mainI}.}
$ $

We will follow the same idea as for Proposition \ref{pr.singI} except we will introduce mesoscopic scales and we shall detect the singularity by observing all these mesoscopic scales ``simultaneously'' (in a suitable sense). 

For any dyadic scale $\eps=\eps_k:=2^{-k}$, we consider the mesoscopic grid $\eps_k \Z^2 \cap [-1,1]^2$, which consists in
$4\cdot 2^{2k}$  squares of side-length $\eps_k$.   Let $\calQ_k$ be the set of {\em quads} corresponding to the $\eps_k$ squares contained in $[-\tfrac 1 2 , \tfrac 1 2]^2$ and which detect an horizontal crossings (in \cite{schramm2011scaling}, a {\em quad} is a topological rectangle, here a square, together with two arcs, here the left and right sides of the square). Note that as in Proposition \ref{pr.singI}, we will need to work at macroscopic distance away from boundary conditions in order to rely on \cite{duminil2022planar}.  

Let $(\omega^+,\omega^-)$ be any subsequential scaling limit of coloured FK-Ising percolation at $\beta=\beta_c + \frac \lambda n$ and let $\omega = \omega^+ \cup \omega^-$ be the un-coloured FK subsequential scaling limit. 
We shall denote these two limiting measures by $\P^\lambda$ and $\P^0$ (and their expectations by $\E^\lambda$ and $\E^0$). We thus have, for a certain subsequence $\{n_\ell\}_{\ell \geq 1}$,
\begin{align*}\label{}
\P^{FK}_{\beta=\beta_c+\frac \lambda {n_\ell}} \overset{\text{in law in }(\calH,d_\calH)}\longrightarrow \, \P^\lambda\,\,\text{ as } \,\, \ell\to\infty \,. 
\end{align*}

As in the discrete setting, let us consider for each $k\geq 1$ the following mesoscopic random variables:
\begin{align}\label{e.Zk}
Z_k=Z_k(\omega):= \sum_{Q\in \calQ_k \,\, \big(\text{ i.e. squares in } [-\tfrac 1 2, \tfrac 1 2]^2  \big) } 1_{Q \in \omega}\,.
\end{align}
(Following \cite{schramm2011scaling}, an element $\omega \in \calH$ is a set of {\em quads} satisfying some closedness and compatibility conditions. The notation $Q\in \omega$ means that the quad $Q$ is traversed by the (continuum) percolation configuration $\omega$). 
As such, $Z_k$ is a measurable function on $(\calH,d_\calH)$ (see \cite{garban2018scaling}) and $Z_k$ is counting the number of $\eps_k$ squares with a left-right crossing event inside $[-\tfrac 1 2 , \tfrac 1 2 ]^2$.

We will construct a measurable event $A$ in~\eqref{e.eventA} at the end of this proof which will spot the singularity of $\P^\lambda$ w.r.t $\P^0$. We first need to prove a few intermediate results. 


To start with, we will crucially need the following analog of Corollary 5.2 from \cite{schramm2011scaling}
\begin{proposition}\label{pr.SS}
For any $\lambda\in \R$, let $\P^\lambda = \lim_{\ell \to \infty} \P^{FK}_{\beta_c + \frac \lambda {n_\ell}}$ be any of the above subsequential scaling limits. Then for any topological quad $Q$ at macroscopic distance from the boundary $\p [-1,1]^2$,
\begin{align*}\label{}
\FK{\beta_c + \frac \lambda {n_\ell} }{FK}{Q \in \omega_{n_\ell}} \underset{\ell \to \infty}\longrightarrow
\FK{}{\lambda}{Q \in \omega} 
\end{align*} 
(N.B. In the notations of \cite{schramm2011scaling}, this corresponds to $\FK{}{\lambda}{\partial \boxminus_Q}=0$). 
\end{proposition} 

\ni
{\em Proof.} 
We follow the same proof as in \cite{schramm2011scaling} and adapt it to the setting of critical FK-Ising percolation instead of $q=1$ critical percolation. One bound is straightforward as the event $\{ Q \in \omega \}$ is compact in $(\calH,d_\calH)$ (see \cite{schramm2011scaling, garban2018scaling}). This ensures
\begin{align*}\label{}
\limsup_{\ell \to \infty} \FK{\beta_c + \frac \lambda {n_\ell} }{FK}{Q \in \omega_{n_\ell}}  \leq 
\FK{}{\lambda}{Q \in \omega}\,.
\end{align*} 
The other bound requires more work. One classical way to prove such a continuity statement under deformations of the topological quads is to argue via {\em three-arm exponents} along the edges and {\em two-arms exponents} near the corners. The three-arm exponent as discovered by Aizenman is universal and its value (2) also applies to critical (and near-critical) FK in the bulk. But the two-arm exponent is model dependent (and also depends on the local geometry of the corner). Fortunately, the proof provided in \cite{schramm2011scaling} (via Appendix A.1) is very robust for our present setting as it only relies on the decay of the plane {\em one arm-event}. We thus claim that the only ingredient missing is the decay of the {\em one-arm event} for the near-critical discrete measure $\P_{n_\ell}^{FK}$ at macroscopic distances from the boundary and uniformly in $\ell$. This is precisely the {\em Stability of the one-arm event} property from Theorem 1.4  in \cite{duminil2022planar}.  
\qed

Using the fact that topological boundaries satisfy $\p(A\cap B)\subset \p A \cup \p B$ and $\p(A\cap B) \subset \p A \cup \p B$, we readily obtain:
\begin{corollary}\label{c.quads}
With the same setup of as in the above proposition, for any quads $Q_1,\ldots,Q_k$ at macroscopic distance from $\p[-1,1]^2$, we have 
\begin{align*}\label{}
\FK{\beta_c + \frac \lambda {n_\ell} }{FK}{Q_i \in \omega_{n_\ell},\, \forall 1 \leq i \leq k} \underset{\ell \to \infty}\longrightarrow
\FK{}{\lambda}{Q_i  \in \omega , \, \forall 1 \leq i \leq k}\,. 
\end{align*} 
\end{corollary}  

The Lemma below gives quantitative estimate on the fluctuations of $Z_k=Z_k(\omega)$. 

\begin{lemma}\label{l.0}
There exists $\lambda_0>0$ such that for any fixed $\lambda \in [-\lambda_0,\lambda_0]$ (including $\lambda=0$ which is the critical measure), 
\begin{align*}\label{}
\mathrm{Var}_{\P^\lambda}\big[\frac {Z_k(\omega)} {2^k} \big] \asymp k\,.
\end{align*}
Furthermore, the constants involved in $\asymp$ are uniform over $\lambda \in [-\lambda_0,\lambda_0]$.
\end{lemma}

\ni
{\em Proof of Lemma \ref{l.0}.}

Using  Theorem 1.13 from \cite{duminil2022planar}, we have that the correlation length statisfies
\begin{align*}\label{}
L(p) \asymp |p-p_c|^{-1}\,.
\end{align*}
(See \cite{duminil2022planar} for the definition of the correlation length). Furthermore, $p\mapsto L(p)$ is increasing on $p\geq p_c$. This implies that we can find $\lambda_0>0$ such that uniformly in $n\geq 0$, 
\begin{align*}\label{}
L(\beta_c+ \frac {\lambda_0} n) \geq 4*n\,.
\end{align*}
(N.B. With a slight abuse of notations here, we parametrise correlation length both in terms of $\beta$ and $p$ depending on context. Also the definition of $L(p)$ in \cite{duminil2022planar} depends on a small fixed cut-off $\delta$. We choose the same cut-off as in \cite{duminil2022planar} which is chosen small enough so that all their estimates hold). 
In particular, for any $\lambda \in [-\lambda_0, \lambda_0]$ and any subsequence $n_\ell$ so that 
$\P_{\beta_c + \frac \lambda {n_\ell}}^{FK} \to \P^\lambda$, we have
\begin{align*}\label{}
L(\beta_c+ \frac {\lambda} {n_\ell}) \geq 4*n_\ell \,.
\end{align*}

\begin{remark}
Restricting our analysis to the window $[-\lambda_0,\lambda_0]$ is probably not of key importance here but this allows us to only work with crossing events of size smaller than the correlation length. In particular, we may readily apply results from \cite{duminil2022planar}. 
 \end{remark}

For $\lambda\in [-\lambda_0,\lambda_0]$, 
\begin{align*}\label{}
\mathrm{Var}_{\P^\lambda}\big[Z_k(\omega) \big]
& =   
\sum_{Q,Q'\in \calQ_k  } \EFK{}{\lambda}{\big(1_{Q \in \omega} - \EFK{}{\lambda}{1_{Q \in \omega}} \big)
\big( 1_{Q \in \omega} - \EFK{}{\lambda}{1_{Q \in \omega}}  \big) }  \\
& = \sum_Q  \mathrm{Var}_{\P^\lambda}\big[ 1_{Q \in \omega}  \big]  + \sum_{Q,Q'}\mathrm{Cov}_{\P^\lambda}
\big[ 1_{Q \in \omega}, 1_{Q' \in \omega}  \big] 
\end{align*}
Using the stability of RSW (\cite{duminil2022planar}) as well as Proposition \ref{pr.SS}, we obtain that the left diagonal term is $\asymp 2^{2k}$ (i.e. the number of $\eps_k$ squares). 

For the right term, we first rely on Corollary \ref{c.quads} to work with a discrete lattice. Now, in the discrete case, these covariances can be controlled thanks to Remarks 5.5 and 5.6 in \cite{duminil2022planar}.  (In particular Remark 5.5. allows us to work with arbitrary boundary conditions at macroscopic distance from the crossing events considered). Namely for any quads $Q,Q' \in \calQ_k$ at euclidean distance at least $3 2^{-k}$ from each other, we obtain in the discrete and uniformly in $\ell$, 
\begin{align*}\label{}
\mathrm{Cov}_{\P^{FK}_{\beta_c+ \frac{\lambda}{n_\ell}}}
\big[ 1_{Q \in \omega}, 1_{Q' \in \omega}  \big] 
& \asymp \Delta_{p_c(q=2) + \Omega(\frac \lambda {n_\ell})}(\frac{n_\ell} {2^k}, \mathrm{dist}(Q,Q'))^2\,, 
\end{align*}
where the distance between $Q$ and $Q'$ is measured in the graph-distance and where $\Delta_p(r,R)$ is the main quantity introduced in the work \cite{duminil2022planar}. It is called the {\em mixing rate} between radii $r<R$ and is defined as follows (we give both the one radius and the two radii definitions, a quasi-multiplicativity statement realting both is proved in \cite{duminil2022planar}):
\begin{align*}\label{}
\begin{cases}
& \Delta_p(R):= \phi^1_{\Lambda_R,p}\big[ \omega_e\big] - \phi^0_{\Lambda_R,p}\big[ \omega_e\big] \\
& \Delta_p(r,R):= \phi^1_{\Lambda_R,p}\big[ \calC(\Lambda_r)] - \phi^0_{\Lambda_R,p}\big[ \calC(\Lambda_r)\big]\,,  
\end{cases}
\end{align*}
where $\phi^1,\phi^0$ are the FK-measure at $(q=2,p)$ with resp wired/free boundary conditions around $\Lambda_R$ and $\calC(\Lambda_r)$ is the left-right crossing event in $\Lambda_r$. 
See \cite{duminil2022planar} for more details.

Now, when $q=2$ (FK-Ising), Theorem 1.13 from \cite{duminil2022planar} states that 
\begin{align*}\label{}
\Delta_p(R) \asymp R^{-1}
\end{align*}
for every radius $R \leq L(p)$. (N.B. This is where our compact window $[-\lambda_0,\lambda_0]$ enters to ensure $R \leq L(p)$). 

Together with the quasimultiplicativity property for $\Delta_p$ which is proved for all $1<q \leq 4$ in Theorem 1.6 in \cite{duminil2022planar}, we obtain
 
\begin{align*}\label{}
\Delta_{p_c(q=2) + \Omega(\frac \lambda {n_\ell})}(\frac{n_\ell} {2^k}, \mathrm{dist}(Q,Q')) \asymp \frac{n_\ell} {2^k \dist(Q,Q')}\,.
\end{align*}
Given these estimates, it now readily follows that 
\begin{align*}\label{}
\sum_{Q,Q', \mathrm{dist}(Q,Q') \geq 3 2^{-k} n_\ell}
\mathrm{Cov}_{\P^{FK}_{\beta_c+\frac \lambda {n_\ell} }}
\big[ 1_{Q \in \omega}, 1_{Q' \in \omega}  \big]  & \asymp \sharp\{ Q, Q\in \calQ_k\} \cdot \sum_{j=1}^k 2^{2j} \frac 1 {2^{2j}} \\
& \asymp 2^{2k} \, k
\end{align*}
where $2^{2j}$ is, up to constants, the number of $\eps_k$ squares $Q'\in \calQ_k$ at graph distance $2^{j-k} n_\ell$ from $Q$. These estimates are uniform in  $\lambda \in [-\lambda_0, \lambda_0]$   and along the subsequence $n_\ell$ (N.B. the subsequence may need to depend on the value of $\lambda$ to ensure convergence).
 
Using Corollary \ref{c.quads} plus the fact that the near-diagonal terms $\mathrm{dist}(Q,Q') \leq 3 2^{-k}n_\ell$ contribute less than $O(1) 2^{2k}$, we conclude about the subsequential scaling limit $\P^\lambda$ that 
\begin{align*}\label{}
\mathrm{Var}_{\P^\lambda}\big[\frac {Z_k(\omega)} {2^k} \big] \asymp k\,.
\end{align*}
\qed  

\begin{lemma}\label{l.1}
For any $\lambda > 0$, there exists a constant $c=c(\lambda)>0$ such that 
\begin{align*}\label{}
\E^\lambda\big( \frac{Z_k} {2^k}\big) - \E^0\big(\frac{Z_k}{2^k}\big) \geq c  k . 
\end{align*}
\end{lemma}

\ni
{\em Proof.} 

Using the monotonicity properties of FK-Ising percolation, it is enough to prove this Lemma for any $\lambda\in(0,\lambda_0]$. 

For each square $Q\in \calQ_k$ contributing to the random variable $Z_k$, we may get back to the discrete using Proposition \ref{pr.SS} which gives us. 
\begin{align*}\label{}
\begin{cases}
&\FK{\beta_c }{FK}{Q \in \omega_{n}} \underset{n \to \infty}\longrightarrow
\FK{}{0}{Q \in \omega} \\
& \FK{\beta_c + \frac \lambda {n_\ell} }{FK}{Q \in \omega_{n_\ell}} \underset{\ell \to \infty}\longrightarrow
\FK{}{\lambda}{Q \in \omega} 
\end{cases}
\end{align*}
The first convergence is the critical one and thus does not require a suitable subsequence. In particular the first limit also holds under the subsequence $\{n_\ell\}$. (N.B. This step would be troublesome if one would wish to prove the singularity of $\P^{\lambda_1}$ and $\P^{\lambda_2}$ for $0<\lambda_1<\lambda_2$, see  Remark \ref{r.lam} below). 

We are thus left with obtaining a good enough lower bound (uniform in $\ell \geq 1$) for 
\begin{align*}\label{}
\FK{\beta_c + \frac \lambda {n_\ell} }{FK}{Q \in \omega_{n_\ell}}  - 
\FK{\beta_c }{FK}{Q \in \omega_{n_\ell}}\,.
\end{align*}

%
%
One can control the above difference thanks to Corollary 1.7, Lemma 5.3 and Remark 5.5  from \cite{duminil2022planar} which shows that for quads $Q$ at macroscopic distance from the boundary $\p[-1,1]^2$ and of size smaller than the correlation length, then in our present notations, for any $0 \leq u \leq \frac {\lambda} {n_\ell}$,

\begin{align*}\label{}
\frac d {du} \FK{\beta_c + u }{FK}{Q \in \omega_{n_\ell}} \asymp R^2 \Delta_{p=p_c(2) + \Omega(u)}(R) 
+ \sum_{l=R}^{L(\beta_c+u)} l \Delta_p(l) \Delta_p(R,l)\,,
\end{align*}
where $R$ is the scale of the quad $Q$, which in our case is $\frac {n_\ell} {2^k}$. $L(\beta_c + u)$ scales like  $ u^{-1}$.   Since $u\leq \frac{\lambda_0}{n_\ell}$, we are always working below the correlation length by our choice of $\lambda_0$. 
Using the above estimates of $\Delta_p$ below the correlation length as well as the quasi-multiplicativity property of $\Delta_p(r_1,r_2)$, $r_2 \leq L(p)$, we obtain 
\begin{align*}\label{}
\frac d {du} \FK{\beta_c +  u  }{FK}{Q \in \omega_{n_\ell}} & \geq c \Big( R^2 \frac 1 R 
+ \sum_{l=R}^{n_\ell} l \frac 1 l \frac R l \Big)\,,
\end{align*}
where we used here the fact that $L(\beta_c+u) \geq n_\ell$, when $u \leq \frac{\lambda} {n_\ell} \leq \frac{\lambda_0}{n_\ell}$. 

Now the scale of our quads $Q\in \calQ_k$ is $ R = 2^{-k } n_\ell$,  which gives us
\begin{align*}\label{}
\frac d {du} \FK{\beta_c +  u  }{FK}{Q \in \omega_{n_\ell}} & \geq c \Big( 2^{-k} n_\ell 
+  2^{-k} n_\ell  \log(\frac{n_\ell}{2^{-k} n_\ell}) \\
& \geq \tilde c \,  k \,  2^{-k} n_\ell 
\end{align*}

This implies that uniformly in $\lambda \leq \lambda_0$ and uniformly in $\{n_\ell\}_\ell$, 
\begin{align*}\label{}
\FK{\beta_c + \frac \lambda {n_\ell} }{FK}{Q \in \omega_{n_\ell}}  - 
\FK{\beta_c }{FK}{Q \in \omega_{n_\ell}}
& \geq \tilde c k\, 2^{-k} n_\ell  \int_0^{\frac \lambda {n_\ell}} du  \geq \tilde c  \lambda k \, 2^{-k} 
\end{align*}

By taking the limit $\ell \to \infty$ and using the continuity result Proposition \ref{pr.SS} as well as the fact that there are $2^{2k}$ squares $Q\in \calQ_k$ contributing to $Z_k$, we obtain the desired result when $0 \leq \lambda \leq \lambda_0$:
\begin{align*}\label{}
\E^\lambda\big( \frac{Z_k} {2^k}\big) - \E^0\big(\frac{Z_k}{2^k}\big) \geq c \lambda \,  k . 
\end{align*}

\qed

\begin{remark}\label{r.lam}
We may wish to prove singularity for two near-critical subsequential scaling limits at $0 < \lambda_1 < \lambda_2$. But this step would require some further work as the subsequence for $\lambda_1$ and $\lambda_2$ may not agree (i.e. $n_{1,\ell}$ and $n_{2,\ell}$).  One possible way out would be to consider the special case where $n_2$ is a subsequence of $n_1$ but this would not be entirely satisfying. 
\end{remark}

\begin{corollary}\label{c.main}
For any $\lambda>0$, there exist  constants  $c_1=c_1(\lambda)>0$ and $c_2= c_2(\lambda)>0$ such that for any subsequential scaling limit $\P^\lambda$ we have for any $k\geq 1$, 
\begin{align*}\label{}
\FK{}{0}{Z_k \geq \EFK{}{0}{Z_k} + c_1 2^k k }  \leq   \frac { c_2} k 
\end{align*} 
and 
\begin{align*}\label{}
\FK{}{\lambda}{Z_k \leq \EFK{}{0}{Z_k} + c_1 2^k k } \leq \frac {c_2} k\,.
\end{align*} 
\end{corollary}

\ni
{\em Proof.}
When $0< \lambda \leq \lambda_0$,  this follows easily using Markov together with Lemmas \ref{l.0} and \ref{l.1}, indeed:
\begin{align*}\label{}
\FK{}{0}{Z_k \geq \EFK{}{0}{Z_k} + \frac{c(\lambda)} 2 2^k k }  \leq  \frac{O(1)}{c(\lambda)^2}  \frac{k}{k^2} \leq \frac {c_1} k 
\end{align*} 
and 
\begin{align*}\label{}
\FK{}{\lambda}{Z_k \leq \EFK{}{0}{Z_k} + \frac{c(\lambda)} 2 2^k k }  \leq  \frac{O(1)}{c(\lambda)^2}  \frac{k}{k^2} \leq \frac {c_2} k\,,
\end{align*} 
where $c(\lambda)>0$ is the constant from Lemma \ref{l.1}.

If $\lambda > \lambda_0$, then we loose our control on the variance from Lemma \ref{l.0}, but using monotonicity of FK, we still obtain using the above bounds at $\lambda=\lambda_0$, 
\begin{align*}\label{}
\FK{}{0}{Z_k \geq \EFK{}{0}{Z_k} + \frac{c(\lambda_0)} 2 2^k k }  \leq  \frac{O(1)}{c(\lambda_0)^2}  \frac{k}{k^2} \leq \frac {c_2} k 
\end{align*} 
and 
\begin{align*}\label{}
\FK{}{\lambda}{Z_k \leq \EFK{}{0}{Z_k} + \frac{c(\lambda_0)} 2 2^k k } &  \leq  
\FK{}{\lambda_0}{Z_k \leq \EFK{}{0}{Z_k} + \frac{c(\lambda_0)} 2 2^k k } \\
& \leq \frac{O(1)}{c(\lambda_0)^2}  \frac{k}{k^2} \leq \frac {c_2} k\,,
\end{align*} 
\qed

We are now ready to define an event detecting the singularity of the limiting measure $\P^0$ and $\P^\lambda$ for any $\lambda>0$ (the proof is the same when $\lambda<0$).

Fix $\lambda>0$ and let $\P^\lambda$ be any subsequential scaling limit at $\beta=\beta_c+\frac \lambda n$ under boundary conditions a),b),c)  on $\p [-1,1]^2$.
Define the event $A$ to be 
\begin{align}\label{e.eventA}
A:=  \bigcup_{N\geq 1} \bigcap_{m\geq N}  \left\{\omega \in \calH, \,  Z_{m^2}(\omega)  \leq \EFK{}{0}{Z_{m^2}}  + c_1(\lambda) 2^{m^2} m^2 \right\}\,.
\end{align}
Notice that $A$ is measurable w.r.t to the quad topology induced by $(\calH,d_\calH)$ from \cite{schramm2011scaling}.

Using Corollary \ref{c.main} together with Borel-Cantelli, we get 
\begin{align*}\label{}
\FK{}{0}{A} = 1
\end{align*}
while on the other hand, for any $m\geq 1$, we know by the same Corollary \ref{c.main} that 
\begin{align*}\label{}
\FK{}{\lambda}{\omega \in \calH, \,  Z_{m^2}(\omega)  \leq \EFK{}{0}{Z_{m^2}}  + c_1(\lambda) 2^{m^2} m^2}
\leq  \frac{c_2} {m^2}\,,
\end{align*}
which ensures as desired that 
\begin{align*}\label{}
\FK{}{\lambda}{A} = 0\,.
\end{align*}
\qed

\subsection{Singularity under $\CLE_3$ scaling limit ?}\label{ss.CLEp}
As stated in Proposition \ref{pr.tight}, we may also consider near-critical subsequential scaling limit under the loop topology from \cite{camia2006two}. At $\beta=\beta_c$, the critical Ising measure is known to converge under this topology towards the nested $\CLE_3$ limit (\cite{benoist2019scaling} or \cite{miller2017cle} together with \cite{kohler2022fuzzy}). 
For $\lambda>0$, we may consider a subsequential scaling limit along $\{n_\ell\}$.
Our main results for Ising hints that such a near-critical $\CLE_3$ is likely to be singular w.r.t the critical $\CLE_3$.  Yet, as it was pointed to us by Paul Cahen, this is by far not an easy consequence of our main Theorem \ref{th.mainI}. 
\smallskip

Let us also highlight that the case of a single near-critical curves is very interesting as well. For spin-Ising, the near-critical $\SLE_3$ is discussed in the work \cite{makarov2010off} and is believed to be absolutely continuous with respect to $\SLE_3$. The analogous case of $\kappa=2$ with a massive version of $\SLE_2$, arising from massive loop erased random walks, is proved to be absolutely continuous w.r.t $\SLE_2$ in the work \cite{chelkak2021convergence}.

\subsection{On the singularity of near-critical $\SLE_{16/3}$ ?}\label{ss.NW}
For near-critical percolation ($q=1$), the first proof of singularity for near-critical subsequential scaling limits appeared in  \cite{nolin2009asymmetry} before the proof of uniqueness of near-critical scaling limits in \cite{garban2013pivotal, garban2018scaling}. 

In terms of singularity, it proved a stronger (more descriptive) result identifying that near-critical interfaces happen to be singular w.r.t the critical interfaces which are known to be $\SLE_6$ curves. In \cite{garban2018scaling} an easier proof of singularity of the entire percolation configurations is provided and this is the type of event which inspired the above proof for FK-Ising percolation ($q=2$). 

One may then wonder whether near-critical FK-Ising interfaces are singular w.r.t $\SLE_{16/3}$ which, if so, would  give another proof of singularity. Interestingly, it seems it is not the case. We state it as a conjecture and give a heuristical argument below in favour of this.  

\begin{conjecture}
Let $\lambda\neq 0$ be fixed. Consider any subsequential scaling limit for the FK-Ising percolation associated to $\beta=\beta_c+\frac \lambda n$ in the domain $\frac 1 n \Z^2 \cap [-1,1]^2$ with Dobrushin free/wired boudnary conditions on left and right sides of $\p [-1,1]^2$. 

Then, the limiting interface $\gamma$ between the primal cluster connected to the wired arc and the dual cluster connected to the free arc is absolutely continuous w.r.t to the $\SLE_{16/3}$ measure. (We may only focus on the subsequential scaling limit of the interface itself or otherwise proceed as in \cite{holden2023convergence} to extract $\gamma$ out of the percolation configuration $\omega\in\calH$). 
\end{conjecture}

\ni
{\em Arguments in favour of this conjecture:}
The main idea behind the proof in \cite{nolin2009asymmetry} that near-critical percolation interfaces are singular w.r.t $\SLE_6$ path measure is the observation that the near-critical interface will ``turn to the right'' more often than what the centred $\SLE_6$ is allowed to do by the central-limit theorem. This is easier to identify at the discrete level than in the continuous limit (and in fact, similarly as what we do in this paper, most of the work from \cite{nolin2009asymmetry} is to translate the discrete singularity observation to the limiting continuous setting). 

In the case of critical $q=1$ site percolation, the discrete critical exploration process is visiting $n^{7/4+o(1)}$ sites (on a piece of triangular lattice of diameter $n$). 
Each new site visited away from the boundary  by the exploration process is a (unbiased) Bernoulli variable. This allows a fluctuation of $n^{7/8+o(1)}$ left/right terms around the mean (which in a symmetric domain is 0).  
Now, by the stability of the two-arm event in the near-critical window, it is not hard to check that the discrete near-critical interface is also made of $n^{7/4+o(1)}$ sites. 
The difference is that now, each new visited site (away from the boundary which is not visited much) is a biased Bernoulli coin with bias the near-critical drift, namely $\delta p = \lambda n^{-3/4+o(1)}$ when $q=1$. This implies that the number of right turns will exceed the number of left turns by $n^{1+o(1)}$ which is much larger than the fluctuation $n^{7/8+o(1)}$. 
As shown in \cite{nolin2009asymmetry}, this discrepancy still prevails in the continuum for all possible subsequential scaling limits. 

If we now consider the critical FK-Ising percolation interface. It is known that at criticality, it is made of $n^{5/3+o(1)}$ steps. 
(The dimension $5/3$ is obtained via $1+\frac \kappa 8$, with $\kappa=16/3$). See \cite{beffara2008dimension,wu2018polychromatic}. 
 By the stability of the two-arm event in the near-critical window (Theorem 1.4 and Remark 7.2 in \cite{duminil2022planar}), this estimate still holds for the near-critical interface. 

The fluctuation of left/right turns is less immediate than for $q=1$ percolation as we are no longer summing  i.i.d variables. 
Yet, at each step, given the previous steps, we are adding a new coin whose conditional bias is uniformly bounded away from 0 and 1. In great generality one cannot exclude possible subtle {\em localisation phenomenon} such as in \cite{gurel2014localization} but we believe these do not hold here (this may be justified by relying on the decay of the mixing rate $\Delta_p(r,R)$).  This suggests that  typical fluctuations should then be at least of order $\sqrt{n^{5/3+o(1)}}$.  


On the other hand, when turning-on the near-critical temperature shift, at each step it induces a bias which can be easily seen to be upper bounded by $O(1) \frac 1 n$ uniformly on what has been revealed so far. (This is due to the FK formula $\Pb{\omega_e \text{ is open }} = p \Pb{e^+ \text{ and } e^- \text{ not connected in }e^c} + \frac p {p+ (1-p)  2} \Pb{e^+ \text{ and } e^- \text{ connected }}$). This means that one should not expect a deviation due to this temperature shift larger than  $\frac 1 n \times n^{5/3+o(1)}$. Since 
\begin{align*}\label{}
n^{2/3+o(1)} \ll \sqrt{n^{5/3+o(1)}}\,, 
\end{align*}
this suggests that the near-critical curve does not turn significantly more than the critical one. The same argument should apply in the continuum limit. Notice that  this is only a heuristical argument in favour of the conjecture, first because of the full justification of $\sqrt{n^{5/3}}$-fluctuations which is lacking and more importantly one would need to prove that there does not exist other events (than counting mesoscopic right and left turns) which may potentially spot a singularity.

\section{Near-critical Potts model, $q\in \{3,4\}$}\label{s.potts}
In this short informal Section, we wish to explain why we expect the energy fields of critical 3-Potts and 4-Potts to be well defined random Schwartz distributions in the plane, as opposed to critical Ising 2-Potts. 
Let us state the following precise conjecture which is based on the expected exponents given in \cite{duminil2022planar}.

\begin{conjecture}\label{c.q3}
FK percolation for 3-Potts and 4-Potts has near-critical scaling limits as $n\to \infty$ under the following temperature shifts
\begin{align*}\label{}
\beta^{q=3} = \beta_c^{q=3} + \frac \lambda {n^{6/5}} \text{   and  }
\beta^{q=4} = \beta_c^{q=4} + \frac \lambda {n^{3/2}} 
\end{align*}
(This follows from the expected correlation exponent $\nu=5/6$ for $q=3$ and $\nu=2/3$ for $q=4$). 

Furthermore, these near-critical scaling limits (under, say, the quad topology $\calH,d_\calH)$) are \underline{absolutely continuous} w.r.t to the critical scaling limits. (N.B. The latter one are not known to be unique until now). 
\end{conjecture}

Let us justify this conjecture by two different heuristical arguments.

\ni
\underline{\em Argument 1. Counting crossing events.} 
We proceed as we did for $q=2$ and check what happens with the  same mesoscopic observable as in the Ising case, namely the variable $Z_k$ from~\eqref{e.Zk}. 
\bnum
\item First, the variance of $Z_k$ (both under the critical and near-critical limits) should scale like 
\begin{align*}\label{}
\sum_{Q,Q' \in \calQ_k} 
\mathrm{Cov}_{\P^{FK}_{\beta_c+\frac \lambda {n^{1/\nu(q)}} }}\big[ 1_{Q \in \omega}, 1_{Q' \in \omega}  \big]   \asymp 2^{4k}  \Delta^{(q)}_{p_c(q)}(2^{-k} n, n)^2 
\end{align*}
When $q=3$, it is expected (\cite{duminil2022planar} p10) that $\Delta_{p_c}(R) \asymp R^{-4/5}$ and when $q=4$ that 
$\Delta_{p_c}(R) \asymp R^{-1/2}$. Using the quasi-multiplicativity of $\Delta_p(r,R)$ which is rigorously established in \cite{duminil2022planar} when $q\in \{3,4\}$, this implies that one should expect
\begin{align}\label{}
\begin{cases}
\Var{Z_k } \asymp (2^{k})^{\tfrac {12} 5} & \text{ when } q=3 \\
\Var{Z_k} \asymp  2^{3k}  & \text{ when } q=4 \,.
\end{cases}
\end{align}

\item Second, using Corollary 1.7  from \cite{duminil2022planar}, we know that for any $u$ in the near-critical window, for any quad $Q \in \calQ_k$ (recall the notations above~\eqref{e.Zk})  and when $q=3$: 
\begin{align*}\label{}
\frac d {du} \FK{\beta_c + u }{FK}{Q \in \omega_{n}} & \asymp R^2 \Delta_{p=p_c(3) + \Omega(u)}(R) 
+ \sum_{l=R}^{L(\beta_c+u)} l \Delta_p(l) \Delta_p(R,l) \\
& \asymp R^{6/5} + \sum_{l=R}^n  l \cdot  l^{-4/5} (R/l)^{4/5} \\
& \asymp R^{6/5} + R^{4/5} n^{2/5} \leq R^{4/5}n^{2/5}
\end{align*}
(where $R$ denotes the mesoscopic scale $\frac n {2^k}$).

By integrating these derivatives along the near-critical window, we thus expect the following shift of expectations when $q=3$:
\begin{align*}\label{}
\E^\lambda\big( Z_k \big) - \E^0\big( Z_k \big) & \leq 2^{2k} \times \int_{\beta_c(3)}^{\beta_c(3)+ \lambda n^{-6/5}}  
(2^{-k} n)^{4/5} n^{2/5} du \\
& \leq 2^{2k} 2^{-\tfrac 4 5 k} = 2^{\tfrac 6 5 k}
\end{align*}
which is of the same order as $\sqrt{\Var{Z_k}}$.

If now $q=4$, we obtain instead
\begin{align*}\label{}
\frac d {du} \FK{\beta_c + u }{FK}{Q \in \omega_{n}} & \asymp R^2 \Delta_{p=p_c(4) + \Omega(u)}(R) 
+ \sum_{l=R}^{L(\beta_c+u)} l \Delta_p(l) \Delta_p(R,l) \\
& \asymp R^{3/2} + \sum_{l=R}^n  l \cdot  l^{-1/2} (R/l)^{1/2} \asymp   R^{1/2} n \,. 
\end{align*}
Integrating this estimate over the expected near-critical window at $q=4$ gives  
\begin{align*}\label{}
\E^\lambda\big( Z_k \big) - \E^0\big( Z_k \big) & \leq 2^{2k} \times \int_{\beta_c(4)}^{\beta_c(4)+ \lambda n^{-3/2}}  
(2^{-k} n)^{1/2} n \,  du \\
& \leq 2^{2k} 2^{-\tfrac 1 2 k} = 2^{\tfrac 3 2 k} \asymp \sqrt{\Var{Z_k}}\,. 
\end{align*}
\enum

\medskip
We now give a second heuristical argument in the spirit of the motivation of this work explained in Section \ref{ss.motivation}.

\smallskip
\ni
\underline{\em Argument 2. Existence of a critical  energy field when $q\in \{3,4\}$.} 

The 2-point correlation functions for the energy field of critical $3$-Potts and $4$-Potts are predicted in \cite{nienhuis1998critical} to be 
\begin{align*}\label{}
\begin{cases}
\<{\calE(z_1), \calE(z_2)}\sim |z_1-z_2|^{-8/5} & \text{ for } q=3 \\
\<{\calE(z_1), \calE(z_2)}\sim |z_1-z_2|^{-1} & \text{ for } q=4 
\end{cases}
\end{align*}
(N.B. The 4-point correlation functions are also computed  explicitly in \cite{nienhuis1998critical} using the Coulomb-gas approach).

Given these exponents, it seems likely that as opposed to $q=2$, the energy field should exist as a random Schwarz distribution both for $q=3$ and $q=4$. Higher order correlation functions would be needed in order to make this more plausible (and also in order to check the existence of exponential moments for these fields). 
In turn, using these fields, one would be able to define a near-critical scaling limit absolutely continuous w.r.t the critical one. As such this gives a second justification to our Conjecture \ref{c.q3}.

\section{Singularity for hierarchical Sine-Gordon model away from the $L^2$ phase}\label{s.SG}

%
%

%

Our goal in this section is to prove the singularity of the hierarchical Sine-Gordon field w.r.t the free field all the way to critical point $``\beta=8\pi"$. 

\subsection{Setup and definition of the hierarchical Sine-Gordon field.}\label{ss.setup}

The main idea of the construction of the hierarchical Sine-Gordon field and its associated RG flow was outlined in Section \ref{ss.intro}. 

Let $L\geq 2$ be an integer (for example $L=2$ will do) and $N\geq 1$ be the number of hierarchical steps. 
 We will rely on the same $2d$-hierarchical GFF as in~\eqref{e.varphiN}
 except we shall normalise the i.i.d Gaussians $\{z^{(n)}_x\}_{n\geq 0, x\in \Lambda_n}$ so that $\E z(x)^2=2\log L$ (instead of 1 in the introduction. The advantage of this normalisation is that $\beta_{L^2}$ and $\beta_{BKT}$ do not depend on $L$). 
As mentioned in the introduction, we shall run our RG flow analysis in the IR setting~\eqref{e.phiN} and will get back to the UV setting~\eqref{e.varphiN} at the end of this section to establish a singularity statement with the hierarchical field $\varphi_\hr$ on $[0,1]^2$. We will therefore use the notation $\phi$ from~\eqref{e.phiN} for the analysis of the RG flow in the next few subsections.

Consider the potential $v(\phi)=\mu\cos\sqrt{\beta}\phi$ and recall the non-linear RG flow $\calR$ introduced in~\eqref{e.RG0}.
It is given by 
 \begin{align}\label{e.RG}
v'(\phi') = (\calR v)(\phi')& := - \log \Eb{e^{- \sum_{x\in \Lambda_1} v(\phi'+z(x))}}\,,
\end{align}
where $\Lambda_1$ corresponds to the $L^2$ points which are integrated out at the smaller scale. 

\begin{remark}\label{}
Notice that the RG flow is preserving the $\frac{2\pi}{\sqrt{\beta}}$-periodicity of the potential. I.e. $v'$ is also a $\frac{2\pi}{\sqrt{\beta}}$-periodic function of $\phi'$ and so on and so forth along the RG flow. 
\end{remark}

Let us also introduce the so-called {\em linearised RG}, which is the first order expansion of the RG flow $\calR$. It  gives with our above choice of normalisation and applied to the above potential $v(\phi)=\mu\cos\sqrt{\beta}\phi$
\begin{equation*}
\caL v(\phi'):=\E\sum_{x\in\Lambda_1}v(\phi'+z(x))=\mu'v(\phi').
 \end {equation*}
 with
\begin{equation*}
\mu'=L^{2-\beta}\mu.
 \end {equation*}
 In this normalization the Berezinskii-Kosterlitz-Thouless point is therefore at $\beta=\beta_{BKT}=2$.
 (N.B. Indeed as soon as $\beta\geq 2$, the RG flow is no longer super-renormalisable in the UV).

\subsection{Analysis of the RG flow.}

%
%
%
 
We will analyse in this Section the above RG flow~\eqref{e.RG}. 
Recall  the non-linear RG flow reads as follows:
 \begin{align*}\label{}
v'(\phi') = (\calR v)(\phi') & :=  - \log \Eb{e^{  - \sum_{x\in \Lambda_1} v(\phi'+z(x))} }
\end{align*}

 Recall that if $X$ is a real-valued random variable  with finite exponential moments $\Eb{e^{tX}}$, for all $t\in[-1,1]$ then it has the following cumulant expansion for any $p\geq 1$:
\begin{align*}\label{}
\log \Eb{e^X} = \sum_{l=1}^p \frac{\kappa_l}{l!} + \int_0^1 \frac{(1-t)^p}{p!} f^{(p+1)}(t) dt\,,
\end{align*}
where the function $f(t)$ is defined as 
\begin{align*}\label{}
f(t)= \log \Eb{e^{t X}}
\end{align*}
At $\phi'$ fixed, when applying this cumulant expansion to the random variable $X:=-\sum_{x\in \Lambda_1} v(\phi'+z(x))$, we obtain the following expansion for the renormalized potential $v$ for any $p\geq 1$:

 \begin{align}
\label{Taylor2}
v'&= \calR v \\
& = \caL v+\sum_{l=2}^p \frac{(-1)^{l+1}}{l!}\langle V;V;\dots ;V\rangle^{(l)}+\frac{(-1)^{p}}{p!}\int_0^1(1-t)^{p}\langle  V;\dots ;V\rangle_t^{(p+1)}\,,
 \end {align}

where $V$ is the random variable $\sum_{x\in \Lambda_1} v(\phi'+z(x))$,  $\<{V;V;\ldots;V}^{(l)}$ denotes the $l^{th}$ cumulant of $V$  under the Gaussian measure $\mu$ (which samples $\{z(x)\}_{x\in \Lambda_1}$) and $\langle V; V;\ldots; V\rangle^{k}_t$ is the $k^{th}$- cumulant of $V$ under the  normalized measure $e^{-tV}\mu$. 
We will rely below on this cumulant expansion of the renormalized potential for a suitable choice of the integer $p=p(\beta)$.

Let $\beta<\beta_{KT}=2$ and take

 \begin{equation}
v^{(N)}(\phi):=L^{-N(2-\beta)}\mu \cos\sqrt{\beta}\phi.
\label{vdefini}
 \end {equation}
 We want to prove that for each finite depth $n$ in the hierarchical tree, then modulo global additive constants, $\lim_{N\to\infty}v_n^{(N)}$ exists.

\begin{remark}
In this Section, as opposed to the next one, we will consider potentials as in statistical physics, namely up to translations by a global energy. I.e. we will not need to follow how the free energy is renormalized along the RG flow but only how the potential gets renormalized. This means we shall not pay attention to constants: $v(\phi)$ is identified with $v(\phi)+C$. Of course, the underlying probability measure involves a global constant via the partition function $Z_N^{(N)}$ as written in~\eqref{e.SGF2}. (See also the later  equation~\eqref{e.RN} where we only need to work modulo constants).
Even though we will not need to control how the  free energy behaves under the RG flow, one could have access to this information by following carefully enough how it flows under RG. This approach will in fact be implemented in Section \ref{s.Phi}. \end{remark}

  Let $p=p(\beta)$ be the smallest integer so that
 \begin{align}\label{e.pbeta}
p \frac   2 3 (2-\beta) >2 \,.
\end{align}
(N.B. $p(\beta=1)=4$ and $p(\beta)\to \infty$ as $\beta \to \beta_{BKT}=2$).  
 We write for each level $1\leq n \leq N$ 
  \begin{equation}\label{e.decomp}
v_n^{(N)}=\sum_{k=1}^{p}\mu_{k,n}^{(N)}c_k(\phi)+w_n^{(N)}(\phi):=u_n^{(N)}(\phi)+w_n^{(N)}(\phi)
 \end {equation}
 where for each $1\leq k \leq p$, we denote
 \begin{align*}\label{}
c_k(\phi):=\cos(k\sqrt{\beta}\phi)
\end{align*}
 and $w_n^{(N)}$ will be thought of as a negligible error term. Note that the linearized RG flow satisfies 
 \begin{equation}
\caL c_k=L^{2-k^2\beta}c_k.
\label{vdefini2}
 \end {equation}
 
\begin{remark}\label{}
One may wonder why we did not choose instead the smallest integer $\hat p$ so that $\hat p(2-\beta)>2$ (which would be sufficient to counter-balance the number of leafs in $\Lambda_1$), but as we shall see below, we will need more lower order terms  in order to control to RG flow, see in particular \eqref{e.iterate} below. 
\end{remark}

By relying on the above cumulant expansion of degree $p=p(\beta)$ defined in~\eqref{e.pbeta},
we will  now control inductively the RG flow:
\begin{align*}\label{}
v_{n}^{(N)} \mapsto v_{n-1}^{(N)}:= \calR v_n^{(N)}\,,
\end{align*}

\smallskip

Let
\begin{align*}\label{}
\mu_n:=L^{-n(2-\beta)}\mu\,.
\end{align*}
 We assume inductively for $n_0\leq n\leq N$
  \begin{equation}
 \begin{cases}
 & |\mu_{k,n}^{(N)}|\leq \epsilon\mu_n^{2k/3},\,  \text{ for all } 2 \leq k \leq p \\
 &  |\mu_{1,n}^{(N)}-\mu_n|\leq A\mu_n^2  \\
&  \|w_n^{(N)}\|_\infty\leq \epsilon\mu_n^{2p/3}
\end{cases}
\label{b}
 \end{equation}
 where $A=A(L,p)$ is taken large and $\epsilon=\epsilon(L,p)$ small depending on $L,p$. Furthermore  $n_0=n_0(L,p,\mu)$ is taken large (so that $\mu_{n}$ can be used to kill constants $C(L,p)$).

 We shall rely on the above cumulant expansion:
 \begin{equation}
\label{Taylor2}
v'=\caL v+\sum_{l=2}^p \frac{(-1)^{l+1}}{l!}\langle V;V;\dots ;V\rangle^{(l)}+\frac{(-1)^{p}}{p!}\int_0^1(1-t)^{p}\langle  V;\dots ;V\rangle_t^{(p+1)}\,.
 \end {equation}
 Let us start by analyzing the first step i.e. $v=v^{(N)}$. Then notice that 
   \begin{equation*}
\sum_{x_i\in\Lambda_1, i=1,\dots l}\langle c_{1}(\phi_{x_1});\dots ;c_{1}(\phi_{x_l})\rangle^{(l)}=\sum_{k=0}^la_kc_{k}(\phi')
 \end{equation*}
(N.B. we are using here {\em joint cumulants} which extend the above notations and whose definition is recalled  below). 
This yields the bounds
  \begin{equation*}
|\mu_{k,N-1}^{(N)}|\leq C(L,k)\mu_N^{k},\ \ k>1, \ \ |\mu_{1,N-1}^{(N)}-\mu_{N-1}|\leq C(L)\mu_N^{2},\ \ \|w_{N-1}^{(N)}\|_\infty\leq C(L,p) \mu_N^{p+1}.
 \end{equation*}
 Hence \eqref{b} holds for $n=N-1$ with $\epsilon$ that can be taken arbitrary small as $N\to\infty$.
 
 \medskip
 For the induction step we keep relying on the cumulant expansion~\eqref{Taylor2}
except we shall now decompose $V=U+W$ according to \eqref{e.decomp}. This naturally brings us to rely on joint cumulants. Recall that if $X_1,\ldots,X_n$ are real random variables with finite exponential moments in a neighbourhood of $0$, then 
\begin{align*}\label{}
\log \Eb{e^{t_1X_1 + \ldots + t_n X_n}} = \sum_{k_1,\ldots,k_n\geq 0} \kappa_{k_1,\ldots,k_n} \frac {\prod t_i^{k_i}}{\prod k_i!}
\end{align*}
The coefficients $\kappa_{k_1,\ldots,k_n}$ are the {\em joint cumulants} and will be denoted by the bracket notation compatible with the above one. For example, for a random vector $(X_1,X_2,X_3,X_4$) we denote 
\begin{align*}\label{}
 \<{X_2; X_2 ; X_3; X_4}^{(4)}:= \kappa_{0,2,1,1}
\end{align*}
Also, the null joint cumulant is zero (i.e. $\kappa_{0,0,...,0}=0$). 
Note that this notation is consistent with the one-variable $n^{th}$ cumulant of $V$, i.e. 
\begin{align*}\label{}
\kappa_n(V)= \kappa_{n,0,\ldots,0}(V,V,\ldots,V) = \kappa_{1,1,\ldots,1}(V,V,\ldots,V) =  \<{V;V;\ldots;V}^{(n)}
\end{align*}
A basic but crucial property we shall use is the fact joint cumulants are multilinear. 
Using the  decomposition  $V=U+W = \sum_{l=1}^p U_l + W$ from \eqref{e.decomp} and made precise below in \eqref{e.Uk},  we may write
\begin{align}\label{e.jointC}
v'&= \calR v \nn \\
& = \caL v+\sum_{l=2}^p (-1)^{l+1}  \sum_{k_1+\ldots+k_{p+1}=l} 
\frac {\kappa_{k_1,\ldots,k_{p+1}}(U_1,U_2, \ldots,U_p,W)}
{\prod k_i!} \\
& \,\,\,\,\,\,+(-1)^{p}  (p+1) \int_0^1 dt (1-t)^{p} \sum_{k_1+\ldots+k_{p+1}=p+1} 
\frac{\kappa_{k_1,\ldots,k_{p+1}}(U_1,U_2, \ldots,U_p,W,t)} 
{\prod k_i!} \nn
\end{align}
where $\kappa_{k_1,\ldots,k_{p+1}}(U_1,U_2, \ldots,U_p,W,t)$ denotes the $(p+1)^{th}$ joint cumulant  
under the joint-normalized measure $e^{-t(U_1+\ldots+U_p+W)}\mu$.
\smallskip

Recall that for each $1\leq k \leq p$, the random variable $U_k$ corresponds to  
\begin{align}\label{e.Uk}
U_k = \sum_{x\in \Lambda_1} \mu_{k,n}^{(N)} c_k(\phi' + z(x))\,.
\end{align}

Now comes our main observation which justifies why we use the structure of joint cumulants:  indeed we claim they have the  following form for any $l\geq 1$ and any $\mathbf k = (k_1,\ldots,k_p)\in (\N)^p$ with $k_1+\ldots+k_p =l$:

  \begin{equation*}
\sum_{x_i\in\Lambda_1, i=1,\dots l}\langle  c_{1}(\phi_{x_1});\ldots ; c_{1}(\phi_{x_{k_1}}); c_{2}(\phi_{x_{k_1+1}})) ; \ldots ;  c_{p}(\phi_{x_l})\rangle^{(l)}=\sum_{k=0}^{l}a_k({\mathbf k})c_{k}(\phi')
 \end{equation*}
 where the first $k_1$ points $x_1,\ldots,x_{k_1}$ are assigned to $c_{1}$, the next $k_2$ points to $c_2$ etc. The deterministic coefficients $a_k({\mathbf{k}})$ also depend on $L$ and $\beta$. (Recall $\phi_x:=\phi'+z(x)$ where $\{z(x)\}_{x\in \Lambda_1}$ is a Gaussian process).
 One way to see why this holds is to notice that cumulants are given by explicit polynomials of the moments. Now, joint moments of the random variables $c_{j}(\phi'+\sum_x z(x))$ under the Gaussian measure $\mu$ are of the above form and this structure is easily seen to be stable under polynomials thanks to trigonometric identities.

%

 Collecting terms we obtain the following recursion
 
 
\begin{align*}
\mu^{(N)}_{k,n-1}=L^{2-k^2\beta}\mu^{(N)}_{k,n}+\sum_{l=k}^p (-1)^{l+1}\sum_{k_1+\ldots+k_p=l} 
\frac
{a_k({\mathbf k})}
{ \prod_{i=1}^p k_i !}
\prod_{i=1}^{p}(\mu^{(N)}_{i,n})^{k_i}
\end{align*}

and put all the other terms in $w^{(N)}_{n-1}$. For $1<k\leq p$ we thus get
\begin{align*}
|\mu^{(N)}_{k,n-1}|\leq L^{2-k^2\beta}\epsilon \mu_n^{2k/3}+C(L,p)(\mu_n^k+\epsilon\mu_n^{1+2(k-1)/3}+\epsilon^2\mu_n^{2k/3})
\end{align*}

We have
  \begin{equation*}
 \epsilon L^{2-k^2\beta}\mu_n^{2k/3}=\epsilon L^{2-k^2\beta}L^{-2k(2-\beta)/3}\mu_{n-1}^{2k/3}\leq  L^{-2/3}\epsilon \mu_{n-1}^{2k/3}
 \end{equation*}
 since $2-k^2\beta-2k(2-\beta)/3\leq 2-4\beta-8/3+4/3\beta=-2/3-8/3\beta\leq -2/3$. Thus, 
 \begin{align*}
|\mu^{(N)}_{k,n-1}|\leq \Bigl(L^{-2/3}+C(L,p,k)(\epsilon^{-1}\mu_{n-1}^{k/3}+\mu_{n-1}^{1/3}+\epsilon)\Bigr)\epsilon\mu_{n-1}^{2k/3}\leq \epsilon\mu_{n-1}^{2k/3}
\end{align*}
 where we used $\epsilon^{-1}\leq C(L,p)$ and $C(L,p)\mu_{n-1}^{1/3}\leq C$
if if $n_0=n_0(L,p)$ is large enough.
 Let next $k=1$. Then, 
  \begin{align*}
|\mu^{(N)}_{1,n-1}-\mu_{n-1}|\leq (L^{\beta-2}A\mu_{n-1}^2+C(L,p)(\mu_{n-1}^{2}+(\epsilon\mu_{n-1}^{2/3})^2+\epsilon\mu_{n-1}^{4/3}\mu_{n-1})\leq A\mu_{n-1}^2
\end{align*}
 if $A$ is taken large enough depending on $L,p$.

 Finally, in order to control the error function $w_{n-1}^{(N)}$ we use the following facts:
 \bnum
 \item First the linearised flow gives 
\begin{equation*}
\|\caL w\|_\infty\leq L^2\| w\|_\infty
 \end{equation*}
\item Then, from the expression~\eqref{e.jointC}, the terms which contribute to  $\calL w_n^{(N)}$  coming from joint cumulants involving $\mu_{1,n}^{(N)} c_1$ and $w_n^{(N)}$ are controlled (in $\|\cdot\|_\infty$) by $C(A,p,L)\mu_n \cdot \epsilon\mu_n^{2p/3}$.

\item The  joint cumulants in~\eqref{e.jointC}  involving higher order terms such as \\ 
$\mu_{4,n}^{(N)}c_4 \cdot \mu_{6,n}^{(N)}c_6 \cdot \mu_{p-2,n}^{(N)} c_{p-2}$ or $\mu_{2,n}^{(N)} c_2 \cdot w_n^{(N)}$ contribute at most 
\begin{align*}
C(A,p,L)\epsilon^2\mu_n^{2(p+1)/3}\,.
\end{align*}
\item Finally it remains to control the contribution of the remaining integral in~\eqref{e.jointC}. These terms are also joint cumulants but against a different probability measure (i.e. the joint-normalized measure $e^{-t(U_1+\ldots+U_p+W)}\mu$). 
Yet by using the fact that joint cumulants are controlled (up to combinatorial constants) by the product of $\|\cdot\|_\infty$ norms, we find that this integral term is a function whose norm is less than $\tilde C(A,L,p) \mu_n^{\frac 2 3 (p+1)}$.
 \enum
 
 All together, one thus obtains
\begin{equation*}
\|w_{n-1}^{(N)}\|_\infty\leq  (L^2+C(A,p,L)\mu_n)\epsilon\mu_n^{2p/3}+C(A,p,L)\epsilon^2\mu_n^{2(p+1)/3}\,.
 \end{equation*}
 
 Since $L^2\mu_n^{2p/3}=L^{2-(2-\beta)2p/3}\mu_{n-1}^{2p/3}=L^{-f} \mu_{n-1}^{2p/3}$ where
 \begin{align}\label{e.iterate}
f:= (2-\beta)2p/3  - 2>0
\end{align}
so that the $w$ bound iterates as well if $n_0(L,p)$ is large enough.

 \subsection{Convergence as $N\to \infty$ limit.}
 
In this subsection, we wish to show that each effective potential has a limit as $N\to \infty$. They are further compatible with each other in the sense of the RG flow. 
More precisely. 
$ $ \\

\begin{theorem}\label{th.conv}
For any $\beta<\beta_{BKT}=2$ and any $a>0$, the sequence of effective potentials $\{v_{n}^{(N)}\}_{1\leq n \leq N}$ which are initialised with 
\begin{align*}\label{}
v_N^{(N)} \,:  \phi \mapsto  \,\, a L^{-N(2-\beta)} \cos(\sqrt{\beta} \phi)
\end{align*}
converge (for the $\| \cdot\|_\infty$ norm on $\calC_b(\R,\R)$) as $N\to \infty$ to a sequence of effective potentials 
\begin{align*}\label{}
\{ \phi \mapsto v_{n}^{\infty,a}(\phi) \}_{n=1}^\infty \,.
\end{align*}
Furthermore, these limiting potentials satisfy the following properties:
\bi
\item[i)] They are compatible under the RG flow, i.e. for any $n\geq 1$, one has
\begin{align*}\label{}
\calR \big[ v_n^{\infty,a} \big]  = v_{n-1}^{\infty,a}
\end{align*}
where the RG flow $\calR$ was defined in~\eqref{e.RG}
\item[ii)] They can be decomposed as follows
\begin{align*}\label{}
v_n^{\infty,a}(\phi) & = \sum_{k=1}^p \mu_{k,n}^{\infty,a}  \cos(k \sqrt{\beta} \phi) + w_n^{\infty,a}(\phi)
\end{align*}
so that there exists a constant $C=C(a)<\infty$ s.t. 
\ei
\begin{align*}\label{}
\begin{cases}
& |\mu_{1,n}^{\infty,a} - \mu_n| \leq C \mu_{n}^{3/2}  \\
& |\mu_{k,n}^{\infty,a}| \leq C \mu_{n}^{2 k/3} \,\,\, \text{ for all } 2 \leq k \leq p \\
& \|w_n^{\infty,a} \|_\infty \leq C \mu_{n}^{2p/ 3}
\end{cases}
\end{align*}
(where $\mu_{n}$ is set to be $a\, L^{-n(2-\beta)}$ for the rest of this section). 
\end{theorem}

The main step of the proof will be the following proposition which quantifies how errors do {\em flow} under the RG flow. 
We will then use this key property to show that each fix $n\geq 1$, the sequence of functions $v_n^{(N)}$ is a Cauchy sequence in $(\calC_b(\R,\R), \| \cdot\|_\infty)$. 

\begin{proposition}[flow of erros in the RG flow]\label{pr.ErrorFlow}
Fix $\beta<\beta_{BKT}=2$. Let $p=p(\beta)$ be as in~\eqref{e.pbeta}. 
If $\mu>0$ is small enough, then for any $\alpha<1$ we have that if 
\begin{align*}\label{}
v(\phi) & = \sum_{k=1}^p \mu_k c_k(\phi) + w(\phi) \\ 
\tilde v(\phi) & = \sum_{k=1}^p (\mu_k + \delta \mu_k)c_k(\phi) + w(\phi) + \delta w(\phi) 
\end{align*}
so that 
\begin{align*}\label{}
\begin{cases}
& |\mu_{1} - \mu| \leq A \mu^{2}  \\
& |\mu_k| \leq \eps \mu^{\tfrac 2 3 k}, \,\,\, \forall  \,  2 \leq k \leq p \\
& \| w \|_\infty \leq  \eps \mu^{\tfrac 2 3 p}
\end{cases}
&
\,\,\,\text{ \,\,  and } 
&
\begin{cases}
& |\delta \mu_{1}| \leq \alpha \mu^{2}  \\
& |\delta \mu_k| \leq \alpha  \mu^{ \tfrac 2 3 k} \,\,\, \text{ for all } 2 \leq k \leq p \\
& \|\delta w \|_\infty \leq  \alpha \mu^{\tfrac 2 3  p}
\end{cases}
\end{align*}
(where the constants $A$ and $\eps$ are the same as in the recursion hypothesis~\eqref{b}). 
then, if one denotes\footnote{We recall the RG flow $\calR$ was defined in~\eqref{e.RG}} $v' = \calR v$, $\tilde v' = \calR \tilde v$, and accordingly $\delta \mu'_1$, ... $\delta \mu'_p$, $\delta w'$, the above estimates propagate
with \underline{same $\alpha$} for all these quantities and with  
\begin{align*}\label{}
\mu' := L^{2-\beta} \mu\,.
\end{align*}
\end{proposition}

\ni
{\em Proof of Proposition \ref{pr.ErrorFlow}.}
We shall only explain the main ideas as the proof follows the same expansion into joint cumulants as in the previous Section. 
In fact the previous section shows that that the hypothesis on $\mu_1, \ldots, \mu_k , w$ propagates to $\mu'_1, \ldots, \mu'_k , w'$.
We need to check that the hypothesis for the error terms $\delta \mu_1$ etc. also propagates. 
We may then write the following expansion for the RG iteration of the potential $\tilde v$ as in~\eqref{e.jointC}:
\begin{align*}
\tilde v'&= \calR \tilde v \nn \\
& = \caL \tilde v+\sum_{l=2}^p (-1)^{l+1}  \sum_{k_1+\ldots+k_{p+1}=l} 
\frac {\kappa_{k_1,\ldots,k_{p+1}}(\tilde U_1,\tilde U_2, \ldots,\tilde U_p, \tilde W)}
{\prod k_i!} \\
& \,\,\,\,\,\,+(-1)^{p}  (p+1) \int_0^1 dt (1-t)^{p} \sum_{k_1+\ldots+k_{p+1}=p+1} 
\frac{\kappa_{k_1,\ldots,k_{p+1}}(\tilde U_1,\tilde U_2, \ldots,\tilde U_p, \tilde W,t)}
{\prod k_i!} \nn
\end{align*}

Let us first collect the term contributing to $\delta \mu'_1$.  There is one coming from the linear RG $\calL \tilde v$ which by assumption contributes at most $L^{2-\beta} |\delta \mu_1| \leq \alpha  L^{2-\beta} \mu^{2} \ll \alpha L^{2(2-\beta)} \mu^{2} = \alpha (\mu')^{2}$ which shows we have a lot of margin on the linear RG side. 
The other contributions to $\delta \mu'_1$ arise only from the terms in 
\begin{align*}\label{}
\sum_{l=2}^p (-1)^{l+1}  \sum_{k_1+\ldots+k_{p+1}=l} 
\frac {\kappa_{k_1,\ldots,k_{p+1}}(\tilde U_1,\tilde U_2, \ldots,\tilde U_p, \tilde W)}
{\prod k_i!} 
\end{align*}
which do not involve $\tilde W$ and produce a $c_1(\phi)=\cos(\sqrt{\beta} \phi)$ function. For example if $l=2$, the smallest order contribution to $\delta \mu'_2$ is given by $\<{\tilde U_1, \tilde U_2}$. This is  because one ends up with terms such as 
\begin{align*}\label{}
& \Eb{e^{i \sqrt{\beta} (\phi'+\phi_x)} e^{-\sqrt{2\beta} (\phi'+\phi_y)} } \\
& = e^{- i \sqrt{\beta} \phi'} \Eb{ e^{i \sqrt{\beta} \phi_x -  2 i \sqrt{\beta} \phi_y } }   = \text{constant}(x,y)\,   e^{- i \sqrt{\beta} \phi'} 
\end{align*}
Now observe that 
\begin{align}\label{e.delta1}
|\<{\tilde U_1, \tilde U_2} - \<{U_1,U_2}| \leq C(L) \big[ \mu_1 |\delta \mu_2|  + |\mu_2| |\delta \mu_1| + |\delta \mu_1 | \cdot |\delta \mu_2|  \big]  \leq  \tilde C(L) \mu \mu^{4/3}
\end{align}
where the combinatorial terms $C(L), \tilde C(L)$ absorb all counting terms coming from $\sum_{z_1,z_2 \in \Lambda_1}$. The point is that one obtains a better exponent than in our assumption. As such, by making $\mu$ small enough, one can counter-balance these $L$-dependent factors.

For the other $\cos(k\sqrt{\beta} \phi)$ terms, $k\geq 2$, let us first check that the linear RG flow behaves fine. For any $2 \leq k \leq p$, it gives a contribution to $\delta \mu'_k$ bounded by 
\begin{align*}\label{}
L^{2 - k^2 \beta} |\delta \mu_k| \ll L^{\tfrac 2 3 k(2 -\beta)} \mu^{2k/3}  = (\mu')^{2k/3} 
\end{align*}
This is because for any $\beta\geq 1$ and any $k \geq 2$,
\begin{align*}\label{}
2 -k^2\beta \leq \tfrac{2k} 3 (2-\beta)\,. 
\end{align*}
For the non-linear RG terms, we will not do a systematic analysis for all terms. Let us only check the following two important ones:  $\delta \mu'_p$ and $\delta w'$. (In particular we will see that for both of them one cannot do much better than $2/3 k$ on our iterative assumption).
The lowest order terms contributing to $\delta \mu'_p$ arise from $\langle\tilde U_1^p\rangle-\langle U_1^p\rangle$ and from $\langle \tilde U_1; \tilde U_{p-1}\rangle- \langle U_1; U_{p-1}\rangle$
This is given at most by 
\begin{align*}\label{}^^
L^{(2-\beta)p}(\mu_1)^{p-1} |\delta \mu_1| + L^4 (|\delta \mu_1|  |\mu_{p-1}|) \leq C(L) \mu^{2 + 2/3(p-1)} \ll \mu^{2p/3 }
\end{align*}
(if $\mu$ is chosen sufficiently small given $L$). 

The lowest order terms contributing to the function $\delta w'$ arise from the following four contributions:
\bnum
\item $ \| \calL [\delta w] \|_\infty $
\item $\|\tilde U_1^{p+1} \rangle-\langle U_1^{p+1}\rangle\|_\infty $
\item $\|\<{ \tilde U_1; \tilde U_{p}} - \<{U_1; U_{p}}\|_\infty$
\item $\|\<{ \tilde U_1; \tilde W} - \<{U_1; W}\|_\infty$
\enum
The Linear RG term is controlled by 
\begin{align*}\label{}
 \| \calL [\delta w] \|_\infty  \leq L^2  \| \delta w \|_\infty  \leq L^2 \alpha \mu^{\tfrac 2 3 p}
\end{align*}
Using the fact that $\tfrac 2 3 p (2-\beta) > 2$ (recall the definition of $p=p(\beta)$ from~\eqref{e.pbeta}), we get 
\begin{align*}\label{}
 \| \calL [\delta w] \|_\infty  \ll  \alpha L^{\tfrac 2 3 p (2-\beta)}(\mu)^{\tfrac 2 3 p} = \alpha (\mu')^{\tfrac 2 3 p}
\end{align*}
For the second term, we obtain 
\begin{align*}\label{}
\|\tilde\langle U_1^{p+1} \rangle-\langle U_1^{p+1}\rangle\|_\infty   \leq C(L) |\delta \mu_1| |\mu_1|^{p-1}
\end{align*}
which is more than needed. The third term is controlled by $C(L) |\delta \mu_1|  |\mu_1|^{\tfrac 2 3 p}$. 
Finally the last term is upper bounded by 
\begin{align*}\label{}
C(L) \big[ |\delta \mu_1| \| w\|_\infty  + |\mu_1| \| \delta w \|_\infty \big]  \leq C(L) \mu^{2} \mu^{\frac 2 3 p} \ll (\mu')^{\tfrac 2 3 p}\,.
\end{align*}
\qed

\begin{corollary}
There exists a constant $C=C(a,L,\beta)>0$ so that for all $1 \leq n \leq N$, 
\begin{align}\label{e.PrN}
\|\phi\in \R \mapsto v_n^{(N+1)}(\phi)-v_n^{(N)}(\phi)\|_\infty\leq CL^{-\frac{1}{2}(2-\beta)N}
\end{align}

so that  $\lim_{N\to\infty}v_n^{(N)}$ exists in the space $\calC_b(\R,\R)$. 
\end{corollary}
\proof
To prove the above estimate, we will check that $\tilde v_N^{(N)}:=\calR[ v_{N+1}^{(N+1)}]$ is close enough to $v_N^{(N)}$ in the sense of Proposition \ref{pr.ErrorFlow} with 
\begin{align*}\label{}
\alpha=\alpha_N = L^{-\frac{1}{2}(2-\beta)N}\,.
\end{align*}
Let us then consider 
\begin{align*}\label{}
v_{N+1}^{(N+1)}(\phi) = \mu_{N+1} \cos(\sqrt{\beta} \phi) \,\,\, \text{ (with $\mu_{n}$ is set to be $a\, L^{-n(2-\beta)}$)}
\end{align*}
Using the (single)-cumulant expansion~\eqref{Taylor2} (we do not need in this case the joint cumulant expansion~\eqref{e.jointC} for the first iteration), 
\begin{align*}
\tilde v_N^{(N)}&:= \calR v_{N+1}^{(N+1)}  \\
&  = \caL v_{N+1}^{(N+1)}+\sum_{l=2}^p \frac{(-1)^{l+1}}{l!}\langle v_{N+1};v_{N+1};\dots ;v_{N+1}\rangle^{(l)} \\
& \;\;\;\;\; +\frac{(-1)^{p}}{p!}\int_0^1(1-t)^{p}\langle  v_{N+1};\dots ;v_{N+1}\rangle_t^{(p+1)}\,.
\end{align*}
The linear RG term $\calL v_{N+1}^{(N+1)}$ is exactly by construction the initial potential $v_N^{(N)}$ at scale above. All other terms are affected to $\{ \mu_{k,N}\}_{1\leq k \leq N}$ and $w_N$ which are all seen here as error terms next to $v_N^{(N)}$. We need to check that each of them satisfies the assumptions of Proposition \ref{pr.ErrorFlow} with the above choice of $\alpha=\alpha_N$. 

As in the proof of Proposition \ref{pr.ErrorFlow} (see equation~\eqref{e.delta1}), the coefficient in front of $\cos(\sqrt{\beta} \phi)$ in the potential $\tilde v_N^{(N)}$ is bounded by 
\begin{align*}\label{}
\tilde C(L) \mu_{1,N+1}^3
\end{align*}
(N.B. the power three comes from the fact one needs 3 $\cos(\sqrt{\beta} (\phi_x+z+x))$ in order to produce (after averaging out $z_x,z_y,z_t$) one term $\cos(\sqrt{\beta} \phi)$. All other contributions are higher order contributions).
If $N$ is large enough, one has 
\begin{align*}\label{}
\tilde C(L) a^3 \mu_{1,N+1}^3  \leq  \mu_{1,N}^{1/2} \mu_{1,N}^2
\end{align*}
which allows us (for the coefficient in front of $c_1$) to take $\alpha_N := \mu_{1,N}^{1/2}$.

Let us analyse the second term. 
Interestingly this term has {\em less margin} in this special case than the first term. This is because one can produce a $\cos(2 \sqrt{\beta} \phi)$ out of only two $\cos(\sqrt{\beta} \phi)$. The leading contribution to $c_2$ gives at most 
\begin{align*}\label{}
C\, L^4 a^2 \mu_{1,N+1}^2 
\end{align*}
which, if $N$ is large enough is indeed much smaller (though with less margin) than
\begin{align*}\label{}
\mu_{1,N}^{1/2} \mu_{1,N}^{\tfrac 4 3}
\end{align*}
Finally we claim that, when $N$ is large enough, the higher order terms $\mu_{2,N}^{(N+1)}, \ldots, \mu_{p,N}^{(N+1)}$ and $w_N^{(N+1)}$ also satisfy the assumptions of Proposition \ref{pr.ErrorFlow} with same $\alpha_N$ and with more comfortable margin than for the second term.

By using Proposition \ref{pr.ErrorFlow}, we can propagate this control up to some finite depth $n_0$. Indeed we need $\mu$ to be sufficiently small in Proposition \ref{pr.ErrorFlow}, since we apply this proposition to $\mu=\mu_{1,n}=L^{-n(2-\beta)}$, it will work for all $n\geq n_0$ for some large enough $n_0$. 

By using the triangle inequality for $\| \cdot \|_\infty$ we easily deduce~\eqref{e.PrN} for $n\geq n_0$ with $C= 2 \sum_{m\geq 0} L^{- m \frac 2 3 (2-\beta)}$. Next we note that 
\begin{align}\label{Rcont}
\|\calR v-\calR u\|_\infty\leq L^2 \|v-u\|_\infty.
\end{align}
Indeed denoting $\tilde v=\sum_xv(\phi+z_x)$ we get
$$
\|\calR v-\calR u\|_\infty=\|\log \frac{\E e^{\tilde u}}{\E e^{\tilde v}}\|_\infty=\|\log \E_ve^{\tilde u-\tilde v}\|_\infty\leq L^2\|u-v\|_\infty
$$
where $ \E_vf:=\E fe^{\tilde v}/  \E e^{\tilde v} $.
Using this bound we deduce~\eqref{e.PrN} for $n\leq n_0$ by replacing $C$ by $L^{2n_0}C$.
 \qed
\medskip

{\em Proof of Theorem \ref{th.conv}.}
Since $(\calC_b(\R,\R), \| \cdot \|_\infty)$ is complete, the above Corollary implies that the effective potentials $v_n^{(N)}$ have a limit $v_{n}^{\infty,a} \in \calC(\R,\R)$ as $N\to \infty$. In fact the above proof provides more: it readily implies that these limiting potentials have the desired expansion stated in item $ii)$ of Theorem \ref{th.conv} (we use $\mu_{n_0}^{1/2}$ to dispose of the coefficient $A$ in \eqref{b}.

Finally, item $i)$ is easily checked by passing to the limit $N\to \infty$ inside  
\begin{align*}\label{}
\calR[v_n^{(N)}] = v_{n-1}^{(N)}
\end{align*}
and by noticing that the operator $\calR$ defined in~\eqref{e.RG} is continuous on the space $(\calC_b(\R,\R), \|\cdot \|_\infty)$ by \eqref{Rcont}. 
\qed

%
%

\subsection{Definition of the hierarchical Sine-Gordon measure and main statement.}

The previous section allows us to define the following {\em hierarchical Sine-Gordon measure}. 
We shall now work in the UV setting and will rely on the scaling relation~\eqref{e.SR} with the IR setting when needed.
Instead of building this measure on $H^{-\eps}(\R^2)$, we will build it on the product space
$\Omega:= \{ (z^{(n)}_{x})_{n \geq 0, \,  x\in \Lambda_n}  \} $.
For each $n\geq 1$, let $\calF_n$ be the filtration induced by the events measurable w.r.t. the first $n$ layers, i.e.
$\{
(z^{(k)}_{x})_{0 \leq k \leq n, \,  x\in \Lambda_k}  \}$. Finally let $\P^{GFF}_\hr$ be the measure used to define the corresponding hierarchical GFF as outlined in the introduction under  the product measure $\calN(0,2 \log L)^{\otimes_{k\geq 0}  \Lambda_k}$. More precisely, for any $(z^{(k)}_{x})_{k\geq 0, x \in \Lambda_k} \in \Omega$ and any $n\geq 1$, we associate as we did in~\eqref{e.varphiN} the following fields on $[0,1)^2$
\begin{align*}\label{}
\begin{cases}
& \varphi_\hr^n(x) := \sum_{k=0}^{n}  z^{(k)}_{[ L^{k}  x ]}   \\
&\varphi_\hr(x) := \sum_{k=0}^{\infty}  z^{(k)}_{[ L^{k}  x ]}    
\end{cases}
\end{align*}

\begin{definition}\label{d.SG}
For any $\beta  < \beta_{BKT}=2$ and any $a\neq 0$, there exists a probability measure $\P^{SG}_{\hr,\beta,a}$ on fields on $\Omega$  which is such that for any finite depth $n\geq 1$ one has 
\begin{align}\label{e.RN}
\frac
{d \P^{SG}_{\hr,\beta,a}}
{d \P^{GFF}_\hr}\Big|_{\calF_n}
\propto \exp( - H_n(\varphi^n_\hr)) = \exp(- \sum_{x\in \bar \Lambda_{n}} v_n^{\infty,a}(\varphi^n_\hr(x)))\,.
\end{align}
(Where $\bar \Lambda_n$ denotes the set of points at scale $L^{-n}$ as in~\eqref{e.Lambda} and~\eqref{e.SGF2}). 

This measure $\P^{SG}_{\hr,\beta,a}$ is the limit in law as $N\to \infty$ of the GFF at depth $N$ weighted by 
\begin{align*}\label{}
 \exp\Big( -\sum_{x\in \bar \Lambda_{N}} a \,  \mu_N \cos(\sqrt{\beta} \varphi^N_\hr(x)) \Big) 
\end{align*}
 (for the topology induced by the finite depth layers $\{\varphi^k_\hr\}_{k\geq 1}$). 
\end{definition}

(N.B note that our proof works equally well for $a<0$ and $a>0$. For convenience we stick to the case $a>0$ so that $\mu_n:= L^{-(2-\beta)n} a$ is positive (otherwise we would just need to cary $|a|$).
We may now state the main result of this Section.
\begin{theorem}\label{th.mainSG}
For any $1 \leq \beta < \beta_{BKT}=2$, and any $a\neq 0$, the hierarchical Sine-Gordon and hierarchical GFF probability measures are singular, i.e. 
\begin{align*}\label{}
\P^{SG}_{\hr,\beta, a} \perp \P^{GFF}_\hr\,. 
\end{align*}
\end{theorem}

 \subsection{Construction of the singular event.}
 
 We start by computing the second and fourth moment of the effective Hamiltonian $H_n$ defined below using the effective potentials $v_n^{\infty,a}(\varphi)$ from Theorem \ref{th.conv} (and when $\varphi$ is a Hierarchical GFF). 
Let $H_n$ be the random variable which corresponds to the ``energy'' at depth $n$., i.e.
\begin{align*}\label{}
H_n:= \sum_{x\in \bar \Lambda_{n}} v_n^{\infty,a}(\varphi^n_x) 
\end{align*}

\begin{proposition}\label{pr.moments}
For any $a\in \R\setminus \{0\}$ (this is the constant used to initialise the RG flow in Theorem \ref{th.conv}), as $n\to \infty$, 
\begin{align}\label{e.varH}
\mathrm{Var}_{GFF}\Big[ H_n(\varphi)  \Big] &  \asymp  
\begin{cases}
n & \text{ if } \beta=1 \\
L^{(2\beta - 2)n} & \text{ if } \beta\in (1,2) 
\end{cases}
\end{align}

\begin{align}\label{e.fourth}
\Eb{H_n(\varphi)^4} &  \leq C \Eb{H_n(\varphi)^2}^2 \,.
\end{align}
\end{proposition}

\ni
{\em Proof.}

We will drop the dependency in $a \neq 0$ which does not play any significant role.

\ni
\underline{First moment.}

To start with, we notice the first moment scales as follows
\begin{align*}\label{}
\Eb{H_n(\varphi)} & = \sum_{x\in \bar \Lambda_{n}} \Eb{v_n^\infty(\varphi^n_x)} \\
& = L^{2n} \Eb{ \sum_{k=1}^p \mu_{k,n}^{\infty}  \cos(k \sqrt{\beta} \varphi^n_x) + w_n^{\infty}(\varphi^n_x)} \\
& = 1 + L^{2n} \Eb{ \sum_{k=2}^p \mu_{k,n}^{\infty}  \cos(k \sqrt{\beta} \varphi^n_x) + w_n^{\infty}(\varphi^n_x)} \\
& = 1 + o(1). 
\end{align*}
Indeed, it follows readily from  Theorem \ref{th.conv} that the higher order terms only contribute $o(1)$ to this first moment for any $\beta\in [1,2)$. 

\smallskip

\ni
\underline{Second moment.}
Noticing that there are $L^{2n}$ ways to choose the ``first'' point $x_0$ and $\asymp L^{2k}$ ways to choose a point $x_k$ at hierarchical distance $k$ from $x_0$, we get 
\begin{align*}\label{}
\Eb{H_n(\varphi)^2}
& \asymp L^{2n}  \sum_{k=0}^{n} L^{2k}   \Eb{v_n^\infty(\varphi^n_{x_0})  v_n^\infty(\varphi^n_{x_k})}\,.
\end{align*}
Now it easy to check that the leading contribution to $\Eb{v_n^\infty(\varphi^n_{x_0})  v_n^\infty(\varphi^n_{x_k})}$ comes from 
\begin{align*}\label{}
(\mu_{1,n}^\infty)^2  \Eb{\cos(\sqrt{\beta} \varphi^n_{x_0}  )   \cos(\sqrt{\beta} \varphi^n_{x_k} )} \asymp L^{-(2-\beta) 2 n} L^{-2 \beta k}\,.
\end{align*}
This leads us to 
\begin{align*}\label{}
\Eb{H_n(\varphi)^2}
& \asymp L^{2n}  \sum_{k=0}^{n} L^{2k}   L^{-(2-\beta) 2 n} L^{-2 \beta k} \\
& \asymp  L^{(2\beta -2) n}  \sum_{k=0}^n L^{(2-2\beta) k}
\end{align*}
which ends the proof of~\eqref{e.varH}.

\begin{remark}\label{}
It can be checked that the lower order terms in $v_n^\infty(\varphi)$ such as $\mu_{2,n}^\infty \cos(2 \sqrt{\beta} \varphi)$ may also contribute with a diverging term to $\Var{H_n(\varphi)}$ (when $\beta\in(1,2)$ is large enough) but the corresponding blow up is negligible w.r.t $\mu_{1,n}$ terms. This can be seen directly by using  $|\Eb{\cos(k \sqrt{\beta} \varphi^n_x)   \cos(k' \sqrt{\beta} \varphi^n_y) }| \lesssim  |\Eb{\cos(\sqrt{\beta} \varphi^n_x)   \cos(\sqrt{\beta} \varphi^n_y) }|$ when $k,k' \geq 1$.  
One may check for example that the contribution to the variance coming from the square of $\sum_x \mu_{2,n} \cos(2 \sqrt{\beta} \varphi^n_x)$ is controlled by $L^{[2 -(2-\beta)\tfrac 8 3] n}$ which is $o(L^{(2\beta-2)n})$ as long as $\beta<2$. 
\end{remark}

\smallskip

\ni
\underline{Fourth moment.}
We need to estimate 
\begin{align*}\label{}
\sum_{x_1,x_2,x_3,x_4} \mu_{1,n}^4  \Eb{\prod_{i=1}^4 \cos(\sqrt{\beta} \varphi^n_{x_i}) }
\end{align*}
(As for the second moment, the higher order terms in the expansions of $v_n^\infty(\varphi)$ give a negligible contribution to the fourth moment). 
The hierarchical geometry makes it easier than on $\T^2$ to group these 4 points depending on their respective distances. We will consider the following two cases (which do overlap slightly, but this is fine as we are only looking for up to constants upper bounds).

\ni
\textbf{Case 1.} {\em Two galaxies.} Two stars $x$ and $y$ are far appart, at hierarchical distance $k$ from each other.  And they each have one satellite $x'$ and $y'$ at respective hierarchical distances $k_1$ and $k_2$ from $x$ and $y$. (We thus have $k_1\wedge k_2 \leq k$).   There are $L^{2n}$ choices for the first star $x$,  of order $L^{2k}$ choices for the $y$ and finally $L^{2 k_1}$ and $L^{2 k_2}$ choices for the satellites $x'$ and $y'$. 

Now for each such configuration with two galaxies, we want to upper bound
\begin{align*}\label{}
\mu_{1,n}^4 \Eb{\cos(\sqrt{\beta} \varphi^n_{x})\cos(\sqrt{\beta} \varphi^n_{x'}) \cos(\sqrt{\beta} \varphi^n_{y}) \cos(\sqrt{\beta} \varphi^n_{y'})}
\end{align*}
which boils down to 
\begin{align*}\label{}
\mu_{1,n}^4 \Eb{e^{ i \sqrt{\beta}   \big( \sigma_x \varphi^n_{x} + \sigma_{x'} \varphi^n_{x'} +  \sigma_y  \varphi^n_{y} + \sigma_{y'} \varphi^n_{y'} \big)}}
\end{align*}
for any choice of signs $(\sigma_x,\sigma_{x'},\sigma_y,  \sigma_{y'}) \in \{ \pm 1\}^4$.  It is easy to check that the galaxies with alternating signs $\{+,-,+,-\}$ or $\{-,+,-,+\}$ give the main contribution which is of order 
\begin{align*}\label{}
\mu_{1,n}^4 * L^{- 4 \tfrac 1 2  * 2 \log L * n \beta} * \Eb{e^{\frac 1 2   * 2 * (2 \log L) * \beta * (2n - k_1-k_2)}} 
\end{align*}
The contribution coming from these {\em 2 galaxies} configurations is thus bounded by 
\begin{align*}\label{}
L^{2n} \sum_{k_1, k_2 \leq k }  L^{2k} L^{2(k_1 + k_2)}   L^{-4n(2-\beta)}   L^{-4n \beta} L^{4n\beta} L^{-2 \beta (k_1+k_2)}\,.
\end{align*}
Let us first analyse this bound when $\beta \in (1,2)$. It is upper bounded by 
\begin{align*}\label{}
L^{2n}  L^{-4n (2-\beta)} \sum_{k=1}^n  L^{2k}  = L^{(4\beta - 4) n} \asymp \Eb{H_n(\varphi)^2}\,,
\end{align*}
as desired. 
If instead $\beta=1$,  the above 2 galaxies contribution is upper bounded by 
\begin{align*}\label{}
L^{2n} \sum_{k_1, k_2 \leq k }  L^{2k} L^{2(k_1 + k_2)}   L^{-4n}  L^{-2 (k_1+k_2)} \asymp L^{-2n} \sum_{k=1}^n k^2  L^{2k} \asymp n^2\,, 
\end{align*}
as desired as well.

It may also happen that $\{x_1,x_2,x_3,x_4\}$ does not form two galaxies. This is our second case below which will turn out to be negligible. 

\smallskip
\ni
\textbf{Case 2.} {\em One galaxy.}  One star $x_1$ with 3 satellites $x_2, x_3, x_4$  at  respective hierarchical distance $k_1 \leq k_2 \leq k_3$ from $x_1$ (the roles of $x_1$ and $x_2$ may be interchanged here).  
There are $L^{2n}$ choices for the ``star'' $x_1$,  and then of order $L^{2k_1}$ choices for the $x_2$,  $L^{2 k_2}$ choices for $x_3$ and $L^{2 k_3}$ choices for $x_4$. With the Coulomb assignment $\{+,-,+,-\}$ to the points $x_1,x_2,x_3,x_4$ we find a contribution of order 
\begin{align*}\label{}
\mu_{1,n}^4 * L^{- 4 \tfrac 1 2  * 2 \log L * n \beta} * \Eb{e^{\frac 1 2   * 2 * (2 \log L) * \beta * [(n - k_1)  + (n-k_3)]}} 
\end{align*}
which yields to a control of the one-galaxy contribution of order 
\begin{align*}\label{}
L^{2n} \sum_{1 \leq k_1 \leq k_2 \leq k_3  }  L^{2(k_1+k_2+k_3)}  L^{-4n(2-\beta)}   L^{-4n \beta} L^{4n\beta} L^{-2 \beta (k_1+k_3)}\,.
\end{align*}
If $\beta\in (1,2)$, we find 
\begin{align*}\label{}
& L^{2n} \sum_{1 \leq k_1 \leq k_2 \leq k_3  }  L^{2(k_1+k_2+k_3)}  L^{-4n(2-\beta)}   L^{-2 \beta (k_1+k_3)} \\
& \leq  O(1) L^{2n} \sum_{1 \leq k_1 \leq k_2  }  L^{2(k_1+2 * k_2)}  L^{-4n(2-\beta)}   L^{-2 \beta (k_1+k_2)} \\
& \leq O(1) L^{2n} \sum_{k_1=1}^n L^{2 k_1 - 2\beta k_1}  L^{-4n(2-\beta)} L^{ 4n - 2\beta n} \\
& \leq O(1) L^{2n -8n +4\beta n +4n  -2\beta n  } = O(1) L^{(2\beta  - 2)n}  \ll L^{(4\beta -4) n}\,.
\end{align*}
If $\beta=1$, we find in the same way an upper bound of order
\begin{align*}\label{}
L^{-2n} \sum_{k_2=1}^n k_2 (n-k_2) L^{2 k_2}  \leq  O(1) n \ll n^2 \,,
\end{align*}
which ends the proof of the fourth moment estimate~\eqref{e.fourth}. 
\qed

%

Proposition \ref{pr.moments} implies (using also that $\Eb{H_n(\varphi)}\sim 1$) that the sequence of random variables 
\begin{align*}\label{}
\left( \frac{H_n(\varphi) }
{\sqrt{\mathrm{Var}_{GFF}\big[ H_n(\varphi) \big]}} \right)_{n\geq 1}
\end{align*}
is tight as $n\to \infty$ and more importantly (using the fourth moment estimate) that it has non-degenerate subsequential limits in law. I.e. there exists a subsequence $\{n_k\}_k$ and a non-degenerate law $\nu$ on $\R$ so that 
\begin{align*}\label{}
\frac{H_{n_k}(\varphi)}
{\sqrt{\mathrm{Var}_{GFF}\big[ H_{n_k}(\varphi) \big]}} \overset{(d)}\longrightarrow \nu
\end{align*}
as $k\to \infty$. (N.B. Indeed, we obtained a control on the fourth moment of $H_n$ in order to prevent the random variable $\frac{H_n(\varphi)} {\sqrt{\mathrm{Var}_{GFF}\big[ H_n(\varphi) \big]}}$ to converge in law to a Dirac point mass at $0$). 
 
 
\smallskip
This implies the existence of $u< 0 < v \in \R$  and $\delta>0$ which are such that 
\begin{align}\label{e.KeyE}
\liminf_{k\to \infty} \Pb{ \frac{H_{n_k}(\varphi) }
{\sqrt{\mathrm{Var}_{GFF}\big[ H_{n_k}(\varphi) \big]}} < u} \wedge \liminf_{k\to \infty} \Pb{\frac{H_{n_k}(\varphi) }
{\sqrt{\mathrm{Var}_{GFF}\big[ H_{n_k}(\varphi) \big]}} > v}  \geq \delta 
\end{align}

The singularity stated in Theorem \ref{th.mainSG} follows readily from the combination of the following two lemmas whose respective outcomes are incompatible under a single probability measure.
 
\begin{lemma}\label{l.LLN}
For sparse enough subsequence $\{m_\ell\}_{\ell \geq 1}$ of $\{n_k\}_k$, one has under the hierarchical GFF measure 
\begin{align*}\label{}
\frac 1 M \sum_{\ell=1}^M  \frac{H_{m_\ell}(\varphi) }
{\sqrt{\mathrm{Var}_{GFF}\big[ H_{m_\ell}(\varphi) \big]}}   \overset{p.s.}\longrightarrow 0
\end{align*}
\end{lemma}

\begin{lemma}\label{l.liminf}
For any $\beta\in [1,2)$ and for sparse enough subsequence $\{m_\ell\}_{\ell \geq 1}$ of $\{n_k\}_k$, one has under the Hierarchical Sine-Gordon measure 
\begin{align*}\label{}
\FK{\beta,a}{SG} 
{\liminf_{\ell \to \infty } \{   \frac{H_{m_\ell}(\varphi) }
{\sqrt{\mathrm{Var}_{GFF}\big[ H_{m_\ell}(\varphi) \big]}}   \leq  \frac u  2 \}} =1 
\end{align*}
(where $a>0$ is the parameter used to initialise the SG-RG flow and where $u<0$ is the real number obtained in~\eqref{e.KeyE}). 
\end{lemma} 

\smallskip
{\em Proof of Lemma \ref{l.LLN}.}


By the standard weak-$L^2$ law of large number and up to taking yet another subsequence (to upgrade a convergence in probability to an almost sure convergence), it is enough to prove that for any fixed $n$, 

\begin{align*}\label{}
\limsup_{m \to \infty} \mathrm{Cov}_{GFF}
\left[
\frac{H_{n}(\varphi) }
{\sqrt{\mathrm{Var}_{GFF}\big[ H_{n}(\varphi) \big]}} 
 \,,\,
 \frac{H_{m}(\varphi) }
{\sqrt{\mathrm{Var}_{GFF}\big[ H_{m}(\varphi) \big]}} 
\right] =0\,.
\end{align*}
As we argued in the above proof, it is sufficient to consider the behaviour of the leading order effective potentials $v_n^\infty(\varphi)$ and $v_m^\infty(\varphi)$ at depth $n$ and $m$. 
 Recall \begin{align*}\label{}
H_n(\varphi):= \sum_{x\in  \bar \Lambda_{n}} v_n^{\infty}(\varphi^n_x)\,.
\end{align*}

We need to estimate 
\begin{align*}\label{}
\sum_{x\in \bar \Lambda_{n},y\in \bar \Lambda_{m}} \mu_{1,n} \mu_{1,m} 
 \Eb{\cos(\sqrt{\beta} \varphi^{n}_{x}  )   \cos(\sqrt{\beta} \varphi^{m}_{y} )}\,.
\end{align*}
The largest possible terms are given when $y$ is a descendent in $\bar \Lambda_{m}$ of $x\in \bar \Lambda_{n}$. In those cases, we have 
\begin{align*}\label{}
 \Eb{\cos(\sqrt{\beta} \varphi^{n}_{x}  )   \cos(\sqrt{\beta} \varphi^{m}_{y} )} \leq L^{-\beta(m-n)}
\end{align*}
so that with a crude bound, 
\begin{align*}\label{}
& \sum_{x\in \bar \Lambda_{n},y\in \bar \Lambda_{m}} \mu_{1,n} \mu_{1,m} 
 \Eb{\cos(\sqrt{\beta} \varphi^{n}_{x}  )   \cos(\sqrt{\beta} \varphi^{m}_{y} )} \\
 & \leq L^{2(n+m)} L^{-(2-\beta)n} L^{-(2-\beta)m} L^{-\beta(m-n)} \\
 & = L^{2\beta n }  \\
 & \ll \sqrt{\mathrm{Var}_{GFF}\big[ H_{n}(\varphi) \big]} \sqrt{\mathrm{Var}_{GFF}\big[ H_{m}(\varphi) \big]}  \\
 & \asymp  \sqrt{nm} \,\, \text{ if $\beta=1$  and }\asymp \sqrt{L^{(\beta-1)(n+m)}}  \text{ if $\beta\in(1,2)$  }
\end{align*}

\qed

\smallskip
\ni
{\em Proof of Lemma \ref{l.liminf}.}
This follows from Borel-Cantelli. 
Indeed, recall from~\eqref{e.RN} that for each finite depth $n$ 
the Radon-Nikodym derivative of the hierarchical field $\varphi$ for the $a$-SG measure w.r.t to the GFF measure is
\begin{align*}\label{}
\propto \exp( - H_n(\varphi)) = \exp(- \sum_{x\in \bar \Lambda_{n}} v_n^{\infty,a}(\varphi^n_x))\,.
\end{align*}
In particular, for each $n$ which belongs to the subsequence $\{n_k\}$ used in the estimate~\eqref{e.KeyE}, we find 
\begin{align}\label{e.11}
\FK{\beta,a}{SG}{\frac 
{H_n(\varphi)}
{\sqrt{\mathrm{Var}_{GFF}\big[ H_{n}(\varphi) \big]}}
\geq \frac  u 2} 
& = \frac 1 {Z_n}  \int 1_{
\frac 
{H_n(\varphi)}
{\sqrt{\mathrm{Var}_{GFF}\big[ H_{n}(\varphi) \big]}}
\geq \frac  u 2
}
 e^{-H_n(\varphi)} \mu_{GFF}(d\varphi)   \nn \\
 & \leq \frac 1 {Z_n} 
 e^{- \frac u 2 \sqrt{\mathrm{Var}_{GFF}\big[ H_{n}(\varphi) \big]}}
\end{align}
while 
\begin{align}\label{e.22}
\FK{\beta,a}{SG}{\frac 
{H_n(\varphi)}
{\sqrt{\mathrm{Var}_{GFF}\big[ H_{n}(\varphi) \big]}}
\leq   u } 
& =
 \frac 1 {Z_n}  \int 1_{
\frac 
{H_n(\varphi)}
{\sqrt{\mathrm{Var}_{GFF}\big[ H_{n}(\varphi) \big]}}
\leq   u 
}
 e^{-H_n(\varphi)} \mu_{GFF}(d\varphi)  \nn  \\
 & \geq \frac 1 {Z_n} 
 e^{- u  \sqrt{\mathrm{Var}_{GFF}\big[ H_{n}(\varphi) \big]}}  
 \FK{}{GFF}
 {
 \frac 
{H_n(\varphi)}
{\sqrt{\mathrm{Var}_{GFF}\big[ H_{n}(\varphi) \big]}}
\leq  u 
 } \nn 
 \\ 
 &  \geq \frac{\delta} {Z_n} e^{
 - u  \sqrt{\mathrm{Var}_{GFF}\big[ H_{n}(\varphi) \big]}} \,,
\end{align}
where we used the fact $n \in \{n_k\}_{k\geq 1}$ so that estimate~\eqref{e.KeyE} is satisfied (N.B. this is the only place where we use our control on the fourth moment). 

Comparing~\eqref{e.11} with~\eqref{e.22}, we see that under the Sine-Gordon measure it is more and more likely as $n \in \{n_k\} \to \infty$ that $H_n(\varphi)$ is very negative. 
By taking a suitable subsequence so that Borel-Cantelli applies, it concludes the proof of Lemma \ref{l.liminf}. 
\qed

\section{RG iteration for  hierarchical $\Phi^4_3$}\label{s.Phi}

\subsection{RG Map.}\label{ss.RGmap}

We work in the dimensionless variables and with a slightly different setup as in the $2d$ hierachical Sine-Gordon model. 
Recall Definition~\eqref{e.phiN} of the $d=3$ hierarchical IR field $\phi^N$. 
The corresponding UV hierarchical field $\varphi^N$ is given in~\eqref{e.varphiN}. 
In short, for any $L\geq 2$, 
\begin{align*}
 \Lambda_N:=\{x\in\Z^d: 0 \leq x_i  < L^N, i=1,\dots, d\}  \text{ and } \bar \Lambda_N : = \frac 1 {L^N} \Lambda_N \subset [0,1]^d
\end{align*}
\begin{equation*}
\phi^N(x)=\sum_{n=0}^{N}L^{-\frac{d-2}{2}n}z^{(N-n)}_{[\frac{x}{L^n}]} \text{ and }
\varphi^N(\bar x)  = L^{\frac {d-2} 2 N} \phi^N([ L^N \bar x ])
 \end {equation*}

As in Section \ref{s.SG}, the RG flow analysis will be done in the IR setting.
Given a bounded positive function $g$ on $\R$ define the RG map $\caR$ by
\begin{equation}
(\caR g)(\phi)=\E\prod_{x\in \Lambda_1}g(L^{-\hf}\phi+z_x)
\nonumber
 \end {equation}
where $\E$ stands for  expectation in the Gaussian variables $\{z_x\}_{x\in\Lambda_1}$ with covariance matrix
 \begin{align*}
\E z_xz_y=\Gamma_{xy}=\delta_{xy}-L^{-3}.
\end{align*}

 We want to study iteration of $\caR$ starting with the function
 $$
 g(\phi)=e^{p-r\phi^2-\lambda\phi^4}
 $$
 where we suppose the parameters $p, r, \lambda$ are small. Let us first briefly discuss what to expect from $\caR( g)$. We have
 \begin{equation}
(\caR g)(\phi)=e^{L^3p-L^2r{\phi}^2-L\lambda{\phi}^4}\E\, e^{-\sigma(\phi,z)}
\label{rgmap1}
 \end {equation}
 where
 \begin{equation}
\sigma(\phi,z)=\sum_{x\in \Lambda_1}(
6\lambda L^{-1}\phi^2z_x^2+4\lambda  L^{-\hf}\phi
z_x^3+\lambda z_x^4+rz_x^2)
\nonumber
 \end {equation}
and we used $\sum z_x=0$. Now suppose first $|\phi|<\rho$ {for some choice of $\rho$ so} that $\lambda\rho^2$ is small. {This corresponds to the so-called ``small-field'' region}. Then we expect to be able to compute the integral in \eqref{rgmap1} perturbatively: 
 \begin{equation}
 \E e^{- \sigma(\phi,z)}=\sum_{l=0}^k {\frac{ (-1)^l }{l!}}\E \,\sigma(\phi,z)^l+R_k(\phi)=1+P_k(\phi)+R_k(\phi).
\label{rgmap2}
 \end {equation}
 The first term $P_k(\phi)$ is an even
 polynomial in $\phi$ of degree $2k$. All together this expression is close to $1$ and we arrive at
  \begin{equation}
  (\caR g)(\phi)
=e^{\sum_{l={0}}^ka_l\phi^{2l}+r_k(\phi)}\label{rgmap3}
 \end {equation}
where 
  \begin{equation}
a_0=L^3p
,\ \ \
 a_1=L^2r+\caO(\lambda),\ \ \ a_2=L^{4-d}\lambda+\caO(\lambda^2),\ \ \ a_l=\caO(\lambda^l),\ \ l>2
\label{rgmap3}
 \end {equation}
 Furthermore, the remainder $r_k(\phi)= \caO((\lambda\rho^2)^{k+1})$.
 
 Next, to  bound $\R(g)$ in the "large field region"  $|\phi|\geq \rho$ first complete the 2nd and 3rd terms in \eqref{rgmap1} to a square:
 \begin{align}\label{}
 \sigma(\phi,z) & = \sum_{x\in \Lambda_1}
(2 \lambda L^{-1}\phi^2z_x^2 + \lambda 
 (2 L^{-\hf}  \phi z_x + z_x^2)^2 +rz_x^2 ) \nn \\
& \geq\sum_{x\in \Lambda_1}
(2\lambda L^{-1}\phi^2z_x^2+rz_x^2).
\label{magic}
 \end {align}
Hence
 \begin{equation}
 (\caR g)(\phi)\leq e^{L^3p
-L^2r{\phi}^2-L\lambda{\phi}^4}\E \, e^{-\sum_{x\in \Lambda_1}rz_x^2}\leq e^{-\frac{1}{2}L\lambda{\phi}^4}
\label{1stlarge}.
 \end {equation}
provided $p,r=\caO(\lambda)$ and $\rho$ is large enough {(w.r.t $L$, more precisely we need $\rho \geq CL^\hf$ since as we shall see below that  $r$ will be negative).}
This motivates the following result:

 \begin{proposition}\label{rgmap}
There exist $\bar\lambda,\delta >0$ s.t the following holds for $\lambda\leq\bar\lambda$. Let for $|\phi|>\rho:=
 \lambda^{-\frac{1}{4}(1+\delta)}$
\begin{equation}
|g(\phi)|\leq e^{-\hf\lambda\phi^4}
\label{ass1}
 \end {equation}
 and for $|\phi|\leq\rho$ assume
 \begin{equation}
g(\phi)=e^{-v(\phi)}
\label{ass1a}
 \end {equation}
 with
 \begin{equation}
\label{ass22}
v(\phi)=\sum_{i=0}^{6}u_i\phi^{2i}+w(\phi)
 \end {equation}
 where
  \begin{equation}
\label{ass322}
u_2=\lambda, \ \ u_1=\caO(\lambda),\ \, u_i=\caO(\lambda^3)
 \ \ i>2.
 \end {equation}
 Furthermore suppose
  \begin{equation}
\label{ass32}
 \sup_{|\phi|\leq\rho}|w(\phi)|\leq \lambda^{3+\delta}
 \end {equation}
 Then 
 $g':=\caR(g)$ satisfies \eqref{ass1}-\eqref{ass32} with $\lambda'=u_2'$ and 
\begin{align}\label{ass3}
u_i'&=L^{3-i}u_i+p_i({\mathbf{u}})
\end {align}
 where
 \begin{align}\label{p0ite}
p_0({\mathbf{u}})=P_0(u_1,u_2)+a_0u_3+\caO(\lambda^4)
\end{align}
and $P_0$ is a polynomial of degree $3$ and   
\begin{align}
p_1&=a_1u_2+
a_2u_1u_2+a_3u_2^2+\caO(\lambda^3)\label{rite}\\
p_2&=a_4u_2^2+\caO(\lambda^3)\label{lite}\\
p_3&=a_5u_2^3 +\caO(\lambda^4),\label{p3ite}\\
p_i&=\caO(\lambda^i),\ \ i>3.\label{p4ite}
 \end {align}
 with $a_4<0$  and
 \begin{equation}
\label{ass4}
\sup_{|\phi|\leq\rho'}|w'(\phi)|\leq \lambda'^{3+\delta}
 \end {equation}
 where $\rho':=
 \lambda'^{-\frac{1}{4}(1+\delta)}$. 
 Here and in what follows $\caO(\lambda^k)$ means bounded by $C(L)\lambda^k$.
 \end{proposition}

 \noindent{\it Remark on constants}. $L$ is fixed and $\lambda<\lambda(L)$ is taken small enough. 
 Generic constants $C,c$ do not depend on $L$. We use $e^{-c(L)\rho^2}=\caO_{{n}}(\lambda^n)$ for all $n$. 

\begin{proof}
\ni
\textbf{1. Small field region. } 
 Our goal is to control $g'(\phi') = (\caR g)(\phi')$  when $|\phi'|$ is assumed to satisfy $|\phi'|\leq \rho':= (\lambda')^{-\frac{1}{4}(1+\delta)}$.  Since $\lambda'$ is not yet fixed at this stage, we instead suppose that  $|\phi'|\leq \hf L^\hf\rho$ 
 and we will make sure below that this condition is less restrictive than $|\phi'|\leq \rho'$ with a suitable choice of $\lambda'$. 

Let us now determine the function  $v'$ on this set $|\phi'|\leq \hf L^\hf\rho$. Let  $\chi(z)$ be the indicator of the event   $|z_x|\leq \hf\rho$ for all $x$ and define
\begin{equation*}
{ g'_\chi(\phi')}:=\E\big(\chi(z)\prod_{x\in \Lambda_1}g(\phi_x)\big)
 \end {equation*}
 where $\phi_x:=L^{-\hf}
{\phi'}+z_x$. We define $g_{1-\chi}$ in the same way so that $g'={g'_{\chi}+g'_{1-\chi}}$. 
 Note that on the support of $\chi$ we have $|\phi_x|\leq \rho$ for all $x$. Hence we may use \eqref{ass1a} and \eqref{ass22} to write
 \begin{equation}
{g'_\chi(\phi')}=\E( e^{-\sum_x
v(\phi_x) 
}\chi(z))=e^{-\nu({\phi'})}\E( e^{-\tau({\phi'},z)-\sum_xw(\phi_x)}\chi(z)).
\label{rgmapa}
 \end {equation}
where $\nu({\phi'}):=\sum_{i=0}^6L^{3-i}u_i {\phi'}^{{2i}}$ and
 \begin{equation*}
 \tau({\phi'},z):=\sum_{i=1}^6u_i \sum_x(\phi_x^{2i}-L^{-i}{\phi'}^{2i})=\sigma({\phi'},z)+\kappa({\phi'},z)
 \end {equation*}
 where $\sigma$ is given by \eqref{rgmap1} (with the choice $\lambda \equiv u_2$ and $r\equiv u_1$) and 
  \begin{align} 
 |\kappa({\phi'},z)|&= {| \sum_{i=3}^6 u_i (\sum_x \phi_x^{2i} - L^{-i} (\phi')^{2i})| }
  {\leq C(L)\sum_{i=3}^6 \lambda^{i} \rho^{2i}}=\caO(\lambda^3\rho^6)
 =\caO(\lambda^{\frac{3}{2}(1-\delta)}).
 \label{rgmap1111}
 \end {align}
Using \eqref{magic} we get, under $\chi(z)$,
 \begin{equation}
 \tau({\phi'},z)\geq  
 r\sum_{x\in \Lambda_1}z_x^2\geq -L^3(\hf \rho)^2|r|\geq -\caO(\lambda^{\hf(1-\delta)}).\label{rgmap111}
 \end {equation}
 Hence $e^{-\tau}$ is close to $1$ and we can write
 \begin{equation}
\E\big( e^{-\tau({\phi'},z)-\sum_xw(\phi_x)}\chi(z)\big)=\E\big(  e^{-\tau({\phi'},z)}\chi(z)\big)+ R_1({\phi'})
\label{rgmapaaa}
 \end {equation}
 where 
  \begin{align*}
|R_1({\phi'})|\leq L^3\lambda^{3+\delta}\sup_{t\in[0,1]}\E\big(  e^{-\tau({\phi'},z)-t\sum_xw(\phi_x)}\chi(z)\big)\leq 2L^3\lambda^{3+\delta}\leq 2L^{-\delta}\lambda'^{3+\delta}.
\label{rgmapab}
 \end {align*}
 using \eqref{rgmap111} and  \eqref{ass32}.
Next we expand
 \begin{equation}
\E\big(  e^{-\tau({\phi'},z)}\chi(z)\big)=\sum_{m=0}^6 {\frac{ (-1)^m}{m!} }\E\big( (\tau({\phi'},z))^m\chi(z)\big)+R_2({\phi'})
\label{rgmapb}
 \end {equation}
where using Taylor's expansion with reminder for the function $t\mapsto e^{-t \tau(\phi',z)}$,
\begin{align}
|R_2({\phi'})|\leq C\sup_{t\in [0,1]}\E\big( |\tau({\phi'},z))|^7e^{-t\tau({\phi'},z)}\chi(z)\big)\nonumber
 \end {align}
$\tau({\phi'},z)$ is a polynomial in the variables $z_x$ and $\phi'$ and we claim that it is bounded  by
 \begin{equation}
|\tau({\phi'},z))|\leq C(L)\lambda\rho^2(1+\|z\|_\infty)^{12}.
\label{rgmapd00}
 \end {equation}
 Indeed,  $\tau({\phi'},z))= \sigma({\phi'},z)+\kappa({\phi'},z)$ 
and from the definition of $\sigma$ in \eqref{rgmap1}, one easily checks that 
$$|\sigma({\phi'},z)| \leq C \lambda \rho^2 (1+ \|z\|_\infty^4).
$$
 For  $\kappa({\phi'},z)$ proceed as in\eqref{rgmap1111}:
\begin{align*}\label{}
|\kappa({\phi'},z)|&\leq\sum_{i=3}^6 |u_i| \sum_x |\phi_x^{2i} - L^{-i} (\phi')^{2i}|   \leq C(L)\sum_{i=3}^6 \lambda^{i} \big[(1 +  |\phi'|^{2i-1})(1+\|z\|_\infty^{2i} )\big] \\
 & \leq C(L) \sum_{i=3}^6 \lambda^{i}  \rho^{2i-1}(1+\|z\|_\infty^{2i} )=C(L)(1+\|z\|_\infty^{12}) \sum_{i=3}^6 \lambda^{\frac{2i+1}{4}-{\frac{2i-1}{4}\delta}}
\end{align*}
Since $\lambda\rho^2=\lambda^{\hf(1-\delta)}$ and 
$i\geq 3$ the bound \eqref{rgmapd00} holds for $\kappa({\phi'},z)$.
 
 Combining \eqref{rgmapd00}  and  \eqref{rgmap111}  we get 
\begin{align}
|R_2({\phi'})|
& \leq  C(L) {( \lambda\rho^2)^{7}}\E\big( (1+\|z\|_\infty)^{84}\chi(z)\big)\leq C(L)\lambda^{\frac{7}{2}(1-\delta)} .%
\label{rgmapc}
 \end {align}
 For the sum in \eqref{rgmapb} we note first that 
 we may dispose of the $\chi$ as the probability of large $z$ is small.
 Using \eqref{rgmapd00} get
 \begin{equation}
|R_3({\phi'})|=|\sum_{m=0}^6 {\frac{ (-1)^m}{m!} }\E\big( (\tau({\phi'},z))^m(1-\chi(z))\big)|\leq Ce^{-c(L)\rho^2}
\label{rgmapd1}
 \end {equation}
 The remaining sum is a polynomial that we write as 
 \begin{equation*}
\sum_{m=0}^6 {\frac{ (-1)^m}{m!} }\E\big( (\tau({\phi'},z))^m\chi(z)\big)=1+\sum_{{i=0}}^6 q_i {(\phi')}^{2i}+R_4({\phi'}).
 \end {equation*}
 {Note that higher degree terms appear here since $\tau$ is already of degree 11 in $\phi'$ and we have $\tau(\phi',z)^{12}$ terms here. These higher degree terms are all assigned to the rest term $R_4$, where again}
  \begin{equation}
|R_4({\phi'})|=\caO(\lambda^{\frac{7}{2}(1-\delta)})\label{rgmape}.
 \end {equation}
 This bound comes from the leading term in this rest term $R_4$ which is proportional to $(\lambda {\rho^2})^7$.
 The coefficients $q_i$ are polynomials in the $u_i$, $i=1,\dots, 6$ with coefficients linear combinations of moments of the $z_x$'s. They are bounded by
 $$
 q_0=\caO(\lambda),\ \ q_i=\caO(\lambda^{i}),\ \ i>0
 $$ 
 and in particular
   \begin{align*}
q_1&=b_1u_2
+b_2u_1u_2+b_3u_2^2+\caO(\lambda^3)\\
q_2&=b_4u_2^2+\caO(\lambda^3) .
 \end {align*}
  Using \eqref{rgmap1} we calculate
 \begin{align}\label{e.bi}
b_1 & = -6L^{-1}\sum_{x\in\Lambda_1} \E z_x^2 
=-\frac{6}{L}(L^3-1)
  \nn \\
b_4& = \hf (\frac{ 6}{L})^2 \Eb{(\sum_x z_x^2)^2}  
\end{align}
and $|b_i|\leq C(L)$ for $i=2,3$.
Combining \eqref{rgmapaaa}-
\eqref{rgmape} we conclude
 \begin{equation}
{g'_\chi(\phi')}=e^{-\nu({\phi'})}
(1+\sum_{{i=0}}^6 q_i{\phi'}^{2i}+R_5({\phi'}))
\label{rgmapf3}.
 \end {equation}
 with
\begin{equation}
|R_5({\phi'}))|=|\sum_{i=1}^4R_i({\phi'}))|\leq 3L^{-\delta}(L\lambda)^{3+\delta}.
\label{rgmapg3}
 \end {equation}
provided $\delta$ is small enough s.t. $\frac{7}{2}(1-\delta)>3+\delta$.
Next we study ${g'_{1-\chi}(\phi')}$. Since $\|g\|_\infty <C$ (N.B. this is a direct consequence of our set of assumptions on $g$) we get
  \begin{equation}
{g'_{1-\chi}(\phi')}\leq C(L)
\E\, (1-\chi(z)) \leq e^{-c(L)\rho^2}.
\label{rgmapk}
 \end {equation}
Now in \eqref{rgmapa} $ 
|\nu({\phi'})|\leq CL\lambda\rho^4 \, {\ll} \, c(L)\rho^2$ so that 
 \begin{equation}
R_6({\phi'}):={g'_{1-\chi}(\phi')e^{\nu(\phi')}}\leq e^{-c(L)\rho^2}.
\label{rgmapff3}
 \end {equation}
(We recall here that the value of constants $c(L)$ may change from line to line). Hence combining \eqref{rgmapf3} and  \eqref{rgmapff3} we conclude
 \begin{equation}
(\caR g)({\phi'})={g'_\chi(\phi')+g'_{1-\chi}(\phi')}=e^{-\nu({\phi'})}
(1+\sum_{{i=0}}^6 q_i{\phi'}^{2i}+R_7({\phi'}))
\label{rgmapf}.
 \end {equation}
 where  \eqref{rgmapg3} holds for
 $R_7$ as well, with say 4 instead of 3 on the RHS. 
 Taking logarithm  we arrive at
  \begin{equation}
v'({\phi'})=-\log(\caR g)({\phi'})=\nu({\phi'})+\sum_{{i=0}}^6 p_i {(\phi')}^{2i}+R_8({\phi'})
\label{rgmapf}.
 \end {equation}
where the $p_i$ are are polynomials in the $q_i$.  The claims for $v'(\phi')$ now follow in the region  $|\phi'|\leq \hf L^\hf\rho$ and {thus in particular in the smaller region $|\phi'|\leq \rho'$}. 
Using $-\log(1+x)=-x+ \tfrac 1 2  x^2 + \caO(x^3)$, we obtain
\begin{align*}\label{}
p_2&=-q_2 + \tfrac 1 2 q_1^2+ \caO(\lambda^3)=(- b_4+\hf b_1^2)\lambda^2+ \caO(\lambda^3)
\end{align*}
and from 
\eqref{e.bi} we get
\begin{align*}\label{}
a_4=- b_4+\hf b_1^2=-\frac{18}{L^2} \Var{X}
\end{align*}
where $X=\sum_{x\in\Lambda_1} z_x^2$. Positivity of the variance implies that $a_4$
 indeed is negative. 

For $p_1$ we obtain, writing $q_0=b_0\lambda+\caO(\lambda^2)$,
\begin{align*}\label{}
p_1&= - q_1+q_0q_1+\caO(\lambda^3)=-b_1\lambda+b_0b_1\lambda^2+\caO(\lambda^3)
\end{align*}
This gives us
\begin{align}\label{e.a1}
a_1&=6L^{-1}(L^3-1) .
\end{align}
Note  also that $a_1>0$.

\medskip
\ni
\textbf{2. Large field region. } 
 Suppose now ${|\phi'|}\geq \rho'$. If $\rho'\leq |{\phi'}|\leq   {\tfrac 1 2}  L^\hf\rho$
the claim \eqref{ass1} for $g'$ follows from our bounds for $v'$ so let $ |{\phi'}|\geq  {\tfrac 1 2} L^\hf\rho$. Under a slightly modified $\tilde \chi$ (for example $\tilde \chi(z)$ being the indicator function of the event that $|z_x| \leq {\tfrac 1 4} \rho$ for all $x$),  we then have $|\phi_x|\geq {\tfrac 1 4 }\rho$. We now distinguish two cases for each $x$: 
\bi
\item[1)] First, if $\tfrac 1 4 \rho \leq |\phi_x| \leq \rho$  then $ 
v(\phi_x)\geq\hf\lambda\phi_x^4$ {(this is because $u_0+u_1 \phi^2 \ll \lambda \phi^4$ in this regime)}.
\item[2)] Second, if $|\phi_x|\geq \rho$, then we have $|g(\phi_x)| \leq e^{-\frac  1 2 \lambda  \phi_x^4}$ 
\ei

{This allows us to obtain} 
\begin{equation*}
{g'_{\tilde \chi}(\phi')}\leq \E\big(e^{-\hf\lambda\sum_{x\in \Lambda_1}\phi_x^4} {\tilde \chi}(z)\big). 
 \end {equation*}
{Now using the same expansion based on $\phi_x = L^{-1/2} \phi' + z_x$ as in} \eqref{magic} we get
\begin{equation}
g'_{\tilde \chi}(\phi')\leq  e^{-\hf L\lambda{\phi'}^4}\E\big(e^{-c(L)\lambda \rho^2\sum_xz_x^2}\big)\leq  e^{-\hf \lambda' {\phi'}^4-c(L)\lambda\rho^2}
\label{rgmap5}
 \end {equation}
 since $\lambda'<L\lambda$ ($a_4<0$ in \eqref{lite}).  
 \medskip
 Finally consider ${g'_{1-{\tilde \chi}}}$.
Under $1-{\tilde \chi}$,  $\phi_x$ may be small so  we use the bound 
$$g(\phi_x)\leq C e^{-\hf\lambda\phi_x^4
}$$ 
valid for all  $\phi_x$. {Similarly as above} we thus obtain
 \begin{equation}
{g'_{1-\tilde \chi}}({\phi'})\leq  C(L)e^{-\hf L\lambda {\phi'}^4
}\E(1- {\tilde\chi}))\leq  e^{-\hf \lambda' {\phi'}^4}e^{-c(L)\rho^2}.
\label{rgmap6}
 \end {equation}
 Combining \eqref{rgmap5} and \eqref{rgmap6} yields the claim.  
 \end{proof}

\subsection{RG Iteration.}

We now take $g_N^{(N)}:=e^{-v_N^{(N)}}$ with  
\begin{equation*}
v_N^{(N)}(\phi)=p_N^{(N)}+r_N^{(N)}\phi^2+\lambda_N^{(N)}\phi^4
 \end {equation*}
 with $\lambda_N^{(N)}=L^{-N}\lambda$. Denote 
  \begin{align*}
\caR^{N-n}e^{-v_N^{(N)}}:=g_n^{(N)}.
\end{align*}
 We want to determine the coefficient $r_N^{(N)}$ so that $\lim_{N\to\infty}g_n^{(N)}$ exists. This is done using the recursion \eqref{p0ite} - \eqref{p4ite}. 
However it is useful to reorganise it a bit by introducing the linearized RG map $\caL$ given by $\caL v=\frac{d}{dt}|_0\caR(e^{tv})$ i.e. concretely
\begin{align*}
(\caL v)(\phi')=\E \sum_{x\in\Lambda_1}v(L^{-\hf}\phi'+z_x).
\end{align*}
 Its eigenfunctions in the vector space of polynomials are the Wick powers (Hermite polynomials)
 \begin{align*}
:(\phi)^{m}:&=\frac{d^{m}}{dt^{m}}|_{t=0}e^{t\phi-\frac{t^2}{2}G_{xx}}
\end{align*}
 where $G_{xx}$ satisfies $G_{xx}=L^{-1}G_{xx}+\Gamma_{xx}$ i.e. $G_{xx}=(1-L^{-1})^{-1}(1-L^{-3})$. Indeed, then
 \begin{align*}
\E:(L^{-1}\phi'+z_x)^{m}:&=\frac{d^{m}}{dt^{m}}|_{t=0}e^{tL^{-\hf}\phi'-\frac{t^2}{2}G_{xx}}\E e^{tz_x}=\frac{d^{m}}{dt^{m}}|_{t=0}e^{tL^{-\hf}\phi'-\frac{t^2}{2}(G_{xx}-\Gamma_{xx})}\\&
=\frac{d^{m}}{dt^{m}}|_{t=0}e^{tL^{-\hf}\phi'-\frac{t^2}{2}L^{-1}G_{xx}}=L^{-\frac{m}{2}}:(\phi)^{m}:
\end{align*}
so that
\begin{align*}
\caL :\phi^m: =L^{3-\frac{m}{2}}:\phi^m:.
\end{align*}
\begin{remark}
The reason for this notation is as follows. Let $\{\phi_x\}_{x\in\Z^3}$ be the infinite volume hierarchical GFF i.e. the Gaussian field with covariance 
 \begin{align*}
G_{xy}=\E\phi_x\phi_y=\sum_{m=0}^\infty L^{-m}\Gamma_{[\frac{x}{L^m}][\frac{y}{L^m}]}
\end{align*}
Then $\phi\stackrel{law}=L^{-\hf}\phi'+z$ where $\phi\stackrel{law}=\phi'$ and $\phi'\perp z$.
\end{remark}

We can use $ :\phi^m: $ as a basis for polynomials and write
\begin{align*}
v_n^{(N)}(\phi)=\sum_{i=0}^6v_{n,i}^{(N)}:\phi^{2i}: +\, w^N_n(\phi)
\end{align*}

If $g_n^{(N)}$ satisfies the assumptions of Proposition \ref{rgmap} then the coefficients  $u_{n,i}^{(N)}$ are linear combinations of $v_{n,i}^{(N)}$. Since $\phi^{2i}=:\phi^{2i}:+\sum_{j=0}^{i-1}c_j:\phi^{2j}:$ so the estimates \eqref{ass322} hold for $v_{n,i}^{(N)}$ as well. However, since linear part of the map $\{v_{n,i}^{(N)}\}\to \{v_{n-1,i}^{(N)}\}$ is now diagonal (given by $\{v_{n,i}^{(N)}\}\to L^{3-i}\{v_{n-1,i}^{(N)}\}$) the equation \eqref{ass3} is replaced by 
\begin{align}\label{ass3new}
v_i'&=L^{3-i}v_i+q_i({\mathbf{v}})
\end {align}
where
 \begin{align}\label{p0itenew}
q_0({\mathbf{v}})=Q_0(v_1,v_2)+\caO(\lambda^4)
\end{align}
and $Q_0$ is a polynomial with quadratic and cubic terms  and   
\begin{align}
q_1({\mathbf{v}})&=\alpha_2v_1v_2+\alpha_3v_2^2+\caO(\lambda^3)\label{ritenew}\\
q_2({\mathbf{v}})&=\alpha_4v_2^2+\caO(\lambda^3)\label{litenew}\\
q_3({\mathbf{v}})&=\alpha_5v_2^3 +\caO(\lambda^4),\label{p3itenew}\\
q_i({\mathbf{v}})&=\caO(\lambda^i),\ \ i>3.\label{p4itenew}
 \end {align}
 Note the absence of $\caO(\lambda)$ term in the mass term $v_1'$ due to the Wick ordering in the $v_2$ term. 
 
 The iteration of \eqref{ass3new}  seems to run into trouble for the $i=0,1$ terms as they expand under the linearized RG. The problem is that their expansion rate is faster or as fast as that of  the $i=2$ term. Indeed, the $v_2$ iteration suggests that we start with $v_{2,N}^{(N)}=L^{-N}\lambda :=\lambda_N$ and expect $v_{2,n}^{(N)}=\lambda_n+\caO(\lambda_n^2)$. In that case  $v_{1,N-1}^{(N)}$ receives a contribution in \eqref{ritenew} of order $\lambda_N^2$ so we are motivated to write  $v_{1,n}^{(N)}=s_n\lambda_n^2+r_n$ where $r_n\lambda_n^{-2}$ is subleading.Then \eqref{ritenew} gives $s_{n-1}=s_n+\alpha_3$ to leading order so that we would have $s_n\sim (N-n)\alpha_3$ ruining our ansatz. The solution is to fine tune the initial condition by taking  $s_N=N\alpha_3$ whereby $v_{1,n}^{(N)}=n\alpha_3\lambda_n^2$ to leading order. This the familiar second order mass counter term for the $\varphi^4_3$ QFT, see Remark \ref{} below. We have then

 \begin{proposition}\label{rgiteartion}
  Let $\lambda>0$ and define
\begin{align}\label{e.functionR}
r(\lambda):= \alpha_3\lambda^2\log_L\lambda
 \end{align}
and set
 \begin{align*}
g_N^N(\phi)=e^{-r_N:\phi^2:-\lambda_N:\phi^4:}
\end{align*}
 with $\lambda_N=L^{-N}\lambda$ and $r_N=r(\lambda_N)$. Define inductively
 \begin{align*}
g^N_{n-1}=e^{-\epsilon^N_n}\caR g^N_{n}
\end{align*}
 where $\epsilon^N_n$ is defined by requiring 
 $$v_{n-1}^N=\sum_{i=1}^6v_{n-1,i}^N:\phi^{2i}:+w_n^N(\phi)
 $$ 
 (i.e.$v_{n-1,0}^N=0$).
 Then
 the following holds for $n$ sufficiently large and $\delta$ sufficiently small.
 \begin{align*}
 &|v_{n,1}^N-r(\lambda_n)|\leq \lambda_n^{\frac{5}{2}}\\
 &v_{n,2}^N=\lambda_n+\caO(\lambda_n^2),\ \ v_{n,i}^N=\caO(\lambda_n^i),\ \ i>2\\
 &|w_n^N(\phi)|\leq \lambda_n^{3+\delta},\ \ {\rm for}\ \ |\phi|\leq\rho_n\\
 &g_n^N(\phi)\leq e^{-\hf\lambda_n\phi^4},\ \ {\rm for}\ \ |\phi|\geq\rho_n
\end{align*}
where $\rho_n:= \lambda_n^{-\tfrac 1 4(1+\delta)}$. 
Furthermore, the limit $g_n=\lim_{N\to\infty}g_n^N$ exists as well as the limits  of $v_{n,1}^N$ and $w_n^N$ and $\epsilon_n^N$. In particular
\begin{align}\label{epsin}
\epsilon_n=12\lambda_n^2(\sum_{x,y\in\Lambda_1} (G_{xy})^4-L^2(G_{00})^4)+\caO(\lambda_n^3)
\end{align}
\end{proposition}

\begin{remark}\label{}
Notice that the result holds for any positive value of $\lambda>0$. This is not so surprising in the present UV setting where the effective potential is controlled for $n\geq 0$ (or, in fact for any $n>-\infty$. It is not hard to control the IR limit $n\to -\infty $ as well provided one adds a mass term $m^2\phi^2$ an takes $\lambda$ small.  We still emphasize this point as many works in the 80's handle UV and IR simultaneously and as such would need $\lambda$ small. 
\end{remark}

\begin{remark}
We also point out that if one wanted to consider instead of~\eqref{e.functionR} an initial mass term 
\begin{align*}\label{}
r(\lambda):= \alpha_3\lambda^2\log_L\lambda + m^2 
\end{align*}
with a free parameter $m^2$, then by keeping the same setup as in our proof, this would also lead to a well defined limiting field which would also be singular w.r.t. GFF. We only stick to $m^2=0$ here for simplicity. 
\end{remark}

\begin{proof} We consider $N\geq n$ and $n$ large enough so that $\lambda_n \leq \bar\lambda$  where $\bar\lambda$ is as in Proposition \ref{rgmap}
Let $g$ stand for  $g_n^N$  and $g'$ stand for  $g_{n-1}^N$. Write $v_1=r(\lambda)+\eta$ and $v'_1=r(\lambda')+\eta'$ where $\lambda'=L\lambda$ and  $|\eta|\leq \lambda^2$. We claim $|\eta'|\leq {\lambda'}^2$. The iteration reads
\begin{align*}
r(L\lambda)+\eta'=L^2r(\lambda)+L^2\eta+\alpha_3\lambda^2+\caO(\lambda^3\log\lambda^{-1}).
\end{align*}
Using 
\begin{align*}
r(L\lambda)=\alpha_3L^2\lambda^2\log_L(L\lambda)+\caO(\lambda^3)
\end{align*}
we infer $\eta'=L^2\eta+\caO(\lambda^3\log\lambda^{-1})$ so that $|\eta'|\leq L^{-\hf}{\lambda'}^{5/2}+\caO(\lambda^3)$ and the claim follows.

\medskip
\noindent 
For the convergence, let us denote $\delta g_{n}^N=g_{n}^{N+1}-g_{n}^N$ and similarly for $\delta v_{i,n}^N$ and  $\delta w_{n}^N$.

\smallskip
\ni
\textbf{(a) Initialisation.} 
For $n=N$ we have
\begin{align*}
|\delta v_{1,N}^N|\leq \lambda_N^{\frac{5}{2}},\ \  |\delta v_{i,N}^N|\leq C(L)\lambda_N^{i}, \ \ i>1, \ \ \sup_{|\phi|\leq\rho_N}|\delta w_N^N(\phi)|\leq \lambda_N^{3+\delta}.
\end{align*}
For $|\phi|\geq \rho_N$, we note that in the estimate \eqref{1stlarge} we may improve a bit to get
\begin{align}\label{deltagN}
| g^{N+1}_N(\phi)|\leq e^{-c(L)\rho_N^4}e^{-\hf\lambda_N\phi^4}\leq e^{-c(L)L^{4N\delta}}e^{-\hf\lambda_N\phi^4}
\end{align} 
and $g_{N}^N$ obviously satisfies the same estimate and thus \eqref{deltagN} holds for $\delta g_{N}^N$ as well.  Let us fix $\epsilon=L^{-\hf\delta}$.

\smallskip
\ni
\textbf{(b) Induction.} 
We prove inductively the following:
\begin{align}\label{deltavNN}
|\delta v_{1,n}^N|\leq \epsilon\lambda_n^{\frac{9}{4}},\ \  |\delta v_{i,n}^N|\leq \epsilon \lambda_n^{i-\hf}, \ \ i>1, \ \ \sup_{|\phi|\leq\rho_n}|\delta w_n^N(\phi)|\leq  \epsilon\lambda_n^{3+\hf\delta}
\end{align}
and  
\begin{align}\label{deltagNN}
\text{for all } |\phi| \geq \rho_n \text{ one has } | \delta g^{N}_n(\phi)|\leq \epsilon\lambda_n^4e^{-\hf\lambda_N\phi^4}
\end{align}

By the preceding discussion these bounds hold for $n=N$.

\smallskip
\ni
\textbf{(c) Induction within the small field region.}
We follow exactly the same setup as in the proof of Proposition \ref{rgmap} in the previous subsection and first consider the small field region. (Will shall only sketch the few adaptations).  First, going from scale $n$ to scale $n-1$,  the small field region corresponds to assume that
$|\phi'| \leq \hf L^\hf \rho_{n} $ (which is a much wider window than $|\phi'| \leq \rho_{n-1}$). 

The first two bounds in \eqref{deltavNN} are straightforward since they contract under the linear RG. For the third term, consider $\delta v'$ where $ v'$ is given by \eqref{rgmapf}. It suffices to bound $\delta R_8$ as we did while proving Proposition \ref{rgmap}. 
We start with $\delta R_1$. Using the bounds \eqref{rgmap1111} and \eqref{rgmap111} 
$$|\delta e^{-\tau}|\leq C|\delta\tau|e^{-r\sum_{x\in \Lambda_1}z_x^2}
$$
with
 \begin{equation}
 |\delta\tau|\leq \caO(\epsilon\lambda^{\frac{3}{2}}\rho^2)(1+\|z\|_\infty)^{12}
.
\label{rgmapd000}
 \end {equation}
 (where the $\epsilon\lambda^{\frac{3}{2}}$ comes from $\delta v_2$, other contributions being smaller)
so that (recall that $\lambda\rho^2=\lambda^{\hf(1-\delta)}$)
\begin{equation*}
\E \big(|\delta e^{-\tau({\phi'},z)}|\chi(z)\big)=\caO(\epsilon\lambda^{1-\hf\delta}).
 \end {equation*}
Using the bounds for $w$ and $\delta w$ we end up with
 \begin{equation}
|\delta R_1(\phi')|\leq( L^3+\caO(\lambda^{1-\hf\delta}))\epsilon\lambda^{3+\hf\delta}
\label{rgmapaaaaa}
 \end {equation}
For $\delta R_2$ we obtain using \eqref{rgmapd00} and \eqref{rgmapd000}
\begin{align}
|\delta R_2(\phi')|\leq C\E\big(|\delta \tau|(|\tau|+|\tilde\tau|)^6e^{C\lambda\sum z_x^2}\big)=\caO(\epsilon \lambda^{3(1-\delta)} \lambda^{(1-\delta/2)})=\caO(\epsilon\lambda^{4-\frac{7}{2}\delta}).
\label{rgmapcccc}
 \end {align}
Obviously $\delta R_3=\caO(\epsilon e^{-c(L)\rho^2})$ and $\delta R_4$ satisfies \eqref{rgmapcccc} so that we end up with 
\begin{equation}
|\delta R_5(\phi')|\leq 2L^{-\delta}\epsilon\lambda'^{3+\frac{\delta}{2}}.
\label{rgmapg}
 \end {equation}
Finally  \begin{equation}
\delta R_6(\phi')=\caO( \epsilon e^{-c(L)\rho^2}).
\label{rgmapff}
 \end {equation}
 The claims for $\delta v'(\phi')$ then follow. 

\medskip
\ni
\textbf{(d) Induction within the large field region.}

Let now $|\phi'|\geq  \hf L^\hf \rho_{n}$ and consider $\delta g'_{\tilde \chi}$ exactly as in the previous section.
Telescoping (as in the identity $\tilde a\tilde b \tilde c - a b c  = (\tilde a-  a) \tilde b \tilde c + a( \tilde b-  b) \tilde c + a b (\tilde c- c)$), we have
\begin{equation*}
\delta g'_{\tilde \chi}(\phi'):=\sum_x\int \delta g(\phi_x)\prod_{y\prec x}g(\phi_y)\prod_{u\succ x}\tilde g(\phi_u)\tilde \chi(z)\mu(dz)
 \end {equation*}
where $\prec$ is any fixed ordering of $\Lambda_1$.

As we did in the previous section, notice we have $|\phi_x|\geq {\tfrac 1 4 }\rho_n$ and we again distinguish two cases for each $x$: 
\bi
\item[1)] First, if $\tfrac 1 4 \rho_n \leq |\phi_x| \leq \rho_n$, we may use the upper bound
\begin{equation*}
|\delta g_{\tilde \chi}(\phi_x)|\leq  C |\delta v(\phi_x)|e^{-\lambda\phi_x^4}\leq \caO(\epsilon)e^{-c(L)\lambda^{-\delta}}e^{-\hf \lambda\phi_x^4}\leq \epsilon\lambda^4e^{-\hf \lambda\phi_x^4}\,,
 \end {equation*}
since $|\delta v(\phi_x)|=\caO(\epsilon\lambda^{\frac{3}{2}}\rho^4)=\caO(\epsilon)$.
\item[2)] Second, if $|\phi_x|\geq \rho_n$, then we may use directly the induction hypothesis on $\delta g$ in this regime. 
\ei
We thus conclude
\begin{equation*}
|\delta g'_ {\tilde \chi}(\phi')|\leq L^3\epsilon \lambda^4 \int e^{-\hf\lambda\sum\phi_x^4} \tilde \chi(z)\mu(dz)\leq L^3\epsilon \lambda^4e^{-\hf \lambda' (\phi')^4}.
 \end {equation*}
  Finally  $\delta g'_{1-\tilde \chi}$ is a negligible and the claim for $\delta g'$ follows since $L^3\lambda^4<\hf \lambda'^4$, say.
  \medskip
  
\ni
\textbf{(e) Computation of $\epsilon_n$.}  It equals $-\lim_{N\to\infty} u_{0,n-1}^N$. We will need it to second order in $\lambda_n$. Let us again drop the indices and denote $v_n^N$ by $v$ and $v_{n-1}^N$ before we extract $\epsilon_n^N$ by $v'$. Thus to leading order $v'(\phi')$ equals $\nu(\phi')$ where
\begin{align*}
\nu(\phi')=\E \sum_{x\in\Lambda_1} (r:\phi_x^2+\lambda:\phi_x^4:)+\hf\lambda^2(\big(\E\sum_{x\in\Lambda_1} :\phi_x^4:\big)^2)-\E \big(\sum_{x\in\Lambda_1} :\phi_x^4:)^2\big)+\dots
\end{align*}  
  where $\phi_x=L^{-\hf}\phi'+z_x$ and $\E$ is over $z$.  Let $\phi'\perp z$ be  gaussian with mean zero and variance $G_{00}=(1-L^{-1})(1-L^{-3})$ and denote by $\E'$  its expectation. Under these laws $\{\phi_x\}_{x\in\Lambda_1}$   is the restriction of the gaussian  field on $\Z^3$ to $\Lambda_1$ discussed above with covariance $\tilde\E\phi_x\phi_y=G_{xy}$. Now $\epsilon_n=\E'\nu(\phi')$. Since $\tilde\E:\phi_x^{2i}:=0$ and $\E\sum_{x\in\Lambda_1} :\phi_x^4:=L:{\phi'}^4:$ we get
\begin{align*}
-\epsilon_n&=\hf\lambda_n^2\big(L^2\E' (:{\phi'}^4:)^2-\tilde\E(\sum_{x\in\Lambda_1} :\phi_x^4:)^2\big)\\&=12\lambda_n^2(L^2(G_{00})^4-\sum_{x,y\in\Lambda_1} (G_{xy})^4)
\end{align*}  
as claimed.
 \end{proof}

\smallskip
This allows us to define the hierarchical $\Phi^4_3$ field. 
\begin{definition}\label{d.Phi}
For any $\lambda>0$
, there exists a probability measure $\P^{\Phi^4_3}_{\lambda}$ on $\Omega$ which is such that for any finite depth $n$  on has 
\begin{align*}\label{}
\P^{\Phi^4_3}_{\lambda}\Big|_{\calF_n}=Z_n^{-1}\big(\prod_{x\in\Lambda_n}g_n(\phi^n)\big)
 \P^{GFF}\Big|_{\calF_n}
\end{align*}
This measure $\P^{\Phi^4_3}_{\lambda}$ is the limit in law as $N\to \infty$ of the GFF at depth $N$ weighted by 
\begin{align*}\label{}
 \exp\Big( -\int_{[0,1]^3}(\lambda :(\varphi^n(x))^4:+\alpha_3\lambda^2\log_L(L^{-N}\lambda):(\varphi^n(x))^4:)dx
\end{align*}
in the topology induced by the finite depth layers $\{\varphi^k_\hr\}_{k\geq 1}$). 
\end{definition}

{
\begin{remark}
This scaling of the coupling constants is consistent with the $\log$-counter terms needed for the construction of the standard (non-hierarchical) $\Phi^4_3$ model (see for example the classic works \cite{glimm1968boson,glimm1973positivity,feldman1974lambdaphi} or the recent ones
\cite{hairer2014theory,gubinelli2015paracontrolled,kupiainen2016renormalization}). The latter model, under an $\eps$-smoothing, is defined via the following Radon-Nykodym derivative at scale $\eps$ (against the GFF measure $\nu_{\mathrm{GFF}}(d\varphi)$ on, say, the torus $\T^3$): 
\begin{align*}\label{}
\exp\left[- \int_{\T^3} \left(\varphi_\eps(x)^4 - (\frac {C_1} \eps + C_2 \log \eps +C_3) \varphi_\eps(x)^2\right) dx\right]\,.
\end{align*}
This is precisely as in our
 hierarchical setting once we realise that the Wick ordered term is given by  $:(\varphi^n)^4:=(\varphi^n)^4-6L^NG_{00}(\varphi^n)^2+3L^{2N}(G_{00})^2$.
 
\end{remark}
}

We may now state the main result of this Section.
\begin{theorem}\label{th.mainPhi}
For any $\lambda>0$ the hierarchical $\Phi^4_3$ and the hierarchical GFF probability measures are singular, i.e. 
\begin{align*}\label{}
\P^{\Phi^4_3}_{\lambda} \perp \P^{GFF}\,. 
\end{align*}
\end{theorem}

\subsection{Construction of the singular event.}

Let us denote  
\begin{align*}
P_n:=\P^{\Phi^4_3}_{\lambda,t} |_{\caF_n},\ \ \P^0_n:=\P^{GFF}|_{\caF_n}
\end{align*}
so that
\begin{align*}
\P_n=Z^{-1}_n\big(\prod_{x\in\Lambda_n}g_n(\phi^n_x)\big)P^0_n.
\end{align*} 
where 
we recall 
\begin{equation*}
{\phi^n(x)}=\sum_{m=0}^{n}L^{-\frac{m}{2}}z^{({n-m})}_{[\frac{x}{L^m}]} .
 \end {equation*}
where $z^m$ are independent Gaussian fields on $\Lambda_m$ with covariance 
\begin{align*}
\E^m z^m_xz^m_y=\delta_{xy}-L^{-3}\delta_{[\frac{x}{L}][\frac{y}{L}]}
\end{align*}
We denote by $\E^{\leq n}$ the expectation in $\P_n^0$ and by $\E^n$ the one in the  field $z_n$. Thus with this notation
$$Z_n=\E^{\leq n}\prod_{x\in\Lambda_n}g_n(\phi^n_x):=e^{E_n}.
$$
By Proposition \ref{rgiteartion} we have
\begin{align*}
\E^n\prod_{x\in\Lambda_n}g_n(\phi^n_x)=e^{L^{3(n-1)}\epsilon_n}\prod_{x\in\Lambda_{n-1}}g_{n-1}(\phi^{n-1}_x)
\end{align*}
so that 
\begin{align*}
Z_n=e^{L^{3(n-1)}\epsilon_n}Z_{n-1}.
\end{align*}
Let $\bar n$ be the smallest $n$ s.t. $\lambda_n\leq\bar\lambda$. Then
\begin{align*}
E_n=\sum_{m=0}^{n-\bar n}L^{3(n-m-1)}\epsilon_{n-m}+\log Z_{\bar n-1}
\end{align*}
Let
\begin{align}\label{e.alpha}
\alpha:=\sum_{x,y\in\Lambda_1} (G_{xy})^4-L^2(G_{00})^4
\end{align}
Then using \eqref{epsin} we get
\begin{align}\label{epsin11}
E_n&=12\lambda_n^2\alpha L^{3(n-1)} \sum_{m=0}^{n-\bar n}L^{-3m}L^{2m}
+\caO(\lambda_n^3)L^{3n}+\log Z_{\bar n-1}\\&
=12\lambda_n^2 \alpha (1-L^{-1})^{-1}L^{3(n-1)}+\caO(1)
\end{align}
Next, we recall that for  $|\phi|\leq\rho_n$ 
 \begin{equation}
g_n(\phi)=e^{-p_n(\phi)-w_n(\phi)}
\label{ass1aN}
 \end {equation}
 with
 \begin{align*}
\sup_{|\phi|\leq\rho_n} |w_n(\phi)|\leq \lambda_n^{3+\delta}
\end{align*}
and $p_n(\phi)$ is the even polynomial of degree 12 given as
 \begin{align*}
p_n(\phi)=\sum_{i=1}^{6} v_{n,i}:\phi^{2i}:
\end{align*}
where $ v_{n,i}$ satisfy the estimates of Proposition \ref{rgiteartion}.
Set
\begin{align*}
H_n(\phi)=\sum_{x\in\Lambda_n}p_n(\phi_x)
\end{align*}

\begin{lemma}
\begin{align*}
\E H_n(\phi^n)=\caO(1), \ \ \hf \Var{H_n(\phi^n)}=E_n+\caO(1)
\end{align*}

\end{lemma}
\begin{proof}

Note that $H_n$ is Wick ordered w.r.t. to the covariance $G_{xy}=\E\phi_x\phi_y$ whereas it is the field $\phi^n$ that is taken expectation  over. A slick way to do the calculation is to expand $H_n(\phi)$ in the Wick powers $:\phi^m:_n$ where we use covariance $G^n_{xy}=\E\phi^n_x\phi^n_y$ instead. We compute:
\begin{align*}
:(\phi^n_x)^{2i}:&=\frac{d^{2i}}{dt^{2i}}|_{t=0}e^{t\phi_x-\frac{t^2}{2}G_{xx}}=\frac{d^{2i}}{dt^{2i}}|_{t=0}\left(e^{t\phi_x-\frac{t^2}{2}G^n_{xx}}e^{\frac{t^2}{2}(G^n_{xx}-G_{xx})}\right)\\&=\sum_{j=0}^ic_j:\phi^{2j}:_{n}(G^n_{xx}-G_{xx})^{i-j}.
\end{align*}
Since $G^n_{xx}-G_{xx}=(1-L^{-3})\sum_{m=n+1}^\infty L^{-m}=\caO(L^{-n})$ we have
\begin{align*}
p_n(\phi)=\sum_{i=0}^{6} v'_{n,i}:\phi^{2i}:_n
\end{align*}
where 
\begin{align*}
v'_{n,0}=\caO(L^{-3n}), \ v'_{n,1}=\caO(L^{-2n}),  \  u'_{n,i}=\caO(L^{-3n}),\ \ i>2
\end{align*}
and $v'_{n,2}=\lambda_n+\caO(\lambda_n^2)$. Hence 
$$\E H_n(\phi^n)=L^{3n}v'_{n,0}=\caO(1).
$$ 
For the variance we note that the Wick powers  satisfy
\begin{align*}
\E:\phi^{2i}:_n:\phi^{2j}:_n=(2i)!(G^n_{xy})^{2i}\delta_{ij}.
\end{align*}
Thus
\begin{align*}
\Var{H_n(\phi^n)}=\sum_{i=1}^{6} A_i(u'_{n,i})^2
\end{align*}
with
\begin{align*}
A_i=(2i)!\sum_{x,y\in\Lambda_n}(G^n_{xy})^{2i}
\end{align*}
Since $G^n_{xy}\leq C(1+d(x,y))^{-1}$ we have 
$A_1=\caO(L^{n}L^{3n})$ ,$A_i=\caO(L^{3n})$ for $i>1$ and furthermore
\begin{align*}
A_2=4!\sum_{x,y\in\Lambda_n}(G^n_{xy})^{4}=4!\sum_{x,y\in\Lambda_n}(G_{xy})^{4}+\caO(L^{2n}).
\end{align*}
Hence
\begin{align*}
\Var{H_n(\phi^n)}=4! \lambda_n^2\sum_{x,y\in\Lambda_n}(G_{xy})^{4}+\caO(1)
\end{align*}
To compute the sum we define
\begin{align*}
X_n:=\sum_{x\neq y\in\Lambda_n}(G_{xy})^{4}
\end{align*}
Using $G_{xy}=L^{-1}G_{[\frac{x}{L}][\frac{y}{L}]}+\Gamma_{xy}$ and $\Gamma_{xy}=0$ if $[\frac{x}{L}]\neq [\frac{y}{L}]$ we get the recursion
\begin{align*}
X_n&=\sum_{x\neq y\in\Lambda_n}(L^{-1}G_{[\frac{x}{L}][\frac{y}{L}]}+\Gamma_{xy})^{4}\\&=L^{-4}\sum_{[\frac{x}{L}]\neq[\frac{y}{L}]}(G_{[\frac{x}{L}][\frac{y}{L}]})^4+L^{3(n-1)}\sum_{x\neq y\in\Lambda_1}(G_{xy})^4\\&=L^2X_{n-1}+L^{3(n-1)}X_1
\end{align*}
Thus
\begin{align*}
X_n=\sum_{m=0}^{n-1}L^{3(n-m-1)}L^{2m}X_1=\frac{L^{3(n-1)}X_1}{1-\frac{1}{L}}+\caO(L^{2n})
\end{align*}
Hence
\begin{align*}
Y_n:=\sum_{x,y\in\Lambda_n}(G_{xy})^{4}=L^{3n}(G_{00})^4+\frac{L^{3(n-1)}X_1}{1-\frac{1}{L}}+\caO(L^{2n})
\end{align*}
Using $X_1=Y_1-L^3(G_{00})^4=\alpha+(L^2-L^3)(G_{00})^4$ (recall $\alpha$ was defined in~\eqref{e.alpha}), some calculation gives 
\begin{align*}
Y_n=\frac{L^{3(n-1)}\alpha}{1-\frac{1}{L}}+\caO(L^{2n})
\end{align*}

\end{proof}

Consider  now the following event:
\begin{align*}
\caE_n=\{H_n\geq 
-\frac{1}{4} \Var{H_n}\,\,\&\,\,  \|\phi^n\|_\infty\leq \rho_n\}.
\end{align*}

Then, under $\caE_n$ we have  
\begin{align*}
Z_n^{-1}\prod_{x\in\Lambda_n}g_n(\phi^n_x)=e^{-E_n-H_n-\sum_xw_n(\phi^n_x)}\leq e^{-\frac{1}{4} \Var{H_n}+\caO(1)}
\end{align*}
and thus
\begin{align*}
\P(\caE_n)\leq 
e^{-c \lambda^2 L^{2n}}
\P^0(\caE_n).
\end{align*}
On the other hand we have 
\begin{lemma}
$\P^0(\caE^c_n)\leq C\lambda^{-2}L^{-n}.$
\end{lemma}
\begin{proof} Since $\Var{H_n})\geq cL^{n}\lambda^{2}$ and $\E H_n=\mathcal{O}(1)$ we get  
$$\P^0(H_n\leq 
-\frac{1}{4} \Var{H_n})\leq C\lambda^{2}L^{-n}.$$
By a union bound 
$$
\P^0(\|\phi^n\|_\infty\geq \rho_n)\leq CL^{3n}e^{-c\rho_n^2}\leq C_k\lambda_n^{-k}
$$
for all $k$. The claim follows.
\end{proof}
Thus  $\P(\caE_n)\to 0$ whereas $\P^0(\caE_n)\to 1$ and furthermore if we consider the event 
\begin{align*}\label{}
A:= \liminf_{n \to \infty }  \calE_n \,,
\end{align*}
we obtain by Borel-Cantelli that this event is detecting the desired singularity, namely 
\begin{align*}\label{}
\P(A)=0 \text{   while    } \P^0(A)=1 \,.     
\end{align*}
\qed

\section{Final remarks and possible directions}
We list below a short list of natural directions and questions. 
\bnum
\item It would be interesting to relate the singularity for the Sine-Gordon field at the $L^2$-threshold $\beta=4\pi$  with the {\em Coleman correspondence at the free fermion point} established in
\cite{bauerschmidt2020coleman}. Recall the singularity for the non-hierarchical case has been obtained on $[4\pi, 6\pi)$ in \cite{gubinelli2024fbsde}.

\item Hierarchical $O(N)$ non-linear $\sigma$-model when $N\geq 2$ seem to be singular w.r.t the hierarchical GFF as suggested by the work \cite{gawedzki1986continuum}. 

\item $\Phi^4_3$ fields on curved spaces have recently been constructed in \cite{bailleul2023phi,hairer2023regularity}. It may be interesting to wonder whether curvature produces any ``local'' singular behaviour w.r.t to the flat case (in which ever sense). 
\item Quid of 2D Yang-Mills measures? See for example \cite{levy2003yang,chevyrev2019yang,chandra2022langevin}.
\enum

\begin{remark}\label{}
Finally, one may wonder what happens with the {\em Liouville field}. Indeed, if one does not pay attention at the effect of insertions (which are of key importance for Liouville CFT), the Liouville field is nothing but a GFF field weighted by a Radon-Nikodym derivarive of the form $\exp(-\mu \int_{\T^2} \Wick{e^{\gamma \phi(x)}} dx)$. It is not hard to check in this case that the induced non-Gaussian field remains absolutely continuous w.r.t GFF all the way to the $L^1$ threshold $\gamma_c=\sqrt{2d}$. As opposed to the Sine-Gordon field which exists beyond the existence of a complex multiplicative chaos, we do not expect an interesting singular field theory to exist beyond this $L^1$ threshold. 
\end{remark}

\medskip

\smallskip

%
%
%
%
%
%
%

\bibliographystyle{alpha}
\addtocontents{toc}{\SkipTocEntry}
\bibliography{singularity}

%

\end{document}